\documentclass[11pt]{amsart}
\usepackage{amssymb}
\usepackage{amsmath}
\usepackage{enumitem}
\usepackage{mathrsfs}
\usepackage[english]{babel}
\usepackage[all]{xy}
\usepackage{tensor}
\usepackage{hyperref}
\usepackage{marginnote}
\usepackage{mathtools}

\usepackage{tikz-cd}
\usepackage{tikz}

 \usetikzlibrary{calc}

\tikzset{
  fibre node/.style={circle, draw, inner sep=1.2pt, minimum size=6pt},
  fibre edge/.style={line width=0.8pt},
  labeltext/.style={font=\small}
}

\newcommand{\drawInFiber}[3]{%
  \def\n{#1}\def\r{#2}\def\m{#3}%
  \foreach \k in {1,...,\n}{
    \path ({360/\n*\k}: \r) node[fibre node] (v\k) {};
  }
  \foreach \k in {1,...,\n}{
    \pgfmathtruncatemacro{\knext}{mod(\k,\n)+1}
    \draw[fibre edge] (v\k) -- (v\knext);
  }
  \ifnum\m>1
    \node[labeltext] at (0,-1.2*\r) {$\times\,\m$};
  \fi
}

\usepackage{stmaryrd}

\usepackage[margin=1.1in]{geometry}

\newtheorem{theorem}{Theorem}[section]
\newtheorem{lemma}[theorem]{Lemma}
\newtheorem{proposition}[theorem]{Proposition}
\newtheorem{corollary}[theorem]{Corollary}
\newtheorem{conjecture}[theorem]{Conjecture}
\newtheorem{alphtheorem}{Theorem}
\newtheorem{alphcorollary}[alphtheorem]{Corollary}

\theoremstyle{definition}
\newtheorem*{ack}{Acknowledgements}
\newtheorem*{con}{Conventions}
\newtheorem{remark}[theorem]{Remark}
\newtheorem{example}[theorem]{Example}
\newtheorem{definition}[theorem]{Definition}

\numberwithin{equation}{section} \numberwithin{figure}{section}

\DeclareMathOperator{\Spec}{Spec}

\DeclareMathOperator{\an}{an}

\DeclareMathOperator{\Hom}{Hom}

\DeclareMathOperator{\supp}{supp}

\DeclarePairedDelimiter\floor{\lfloor}{\rfloor}
\DeclarePairedDelimiter\powerseries{[ \! [}{] \! ]}

\newcommand*\ratmap{\mathbin{\tikz [baseline=0ex,-latex, dashed, ->] \draw [densely dashed] (0em,0.58ex) -- (1.3em,0.58ex);}}

\newcommand\HH{\mathrm{H}}

\usepackage{color}

\definecolor{orange}{rgb}{1,0.5,0}

\RequirePackage[normalem]{ulem}

\title[Weakly special varieties, Campana stacks, Orbifold Mordell]{Weakly special varieties, Campana stacks, and Remarks on Orbifold Mordell}

\author{Finn Bartsch}
\address{Finn Bartsch \\
Institut für Algebra, Zahlentheorie und diskrete Mathematik \\
Leibniz Universität Hannover \\
Welfengarten 1, 30451 Hannover, Germany}
\email{bartsch@math.uni-hannover.de}

\author{Ariyan Javanpeykar}
\address{Ariyan Javanpeykar,
IMAPP Radboud University Nijmegen,
PO Box 9010, 6500GL,
Nijmegen, The Netherlands}
\email{ariyan.javanpeykar@ru.nl}



\begin{document}

\begin{abstract}
We construct the first weakly special surfaces that are not Campana-special, including the complement of the plane curve $x^2y^3 = 1$ in $\mathbb{A}^2$. We prove that the set of $\mathcal{O}_{K,S}$-integral points on this surface is non-dense for every number field $K$ and finite set $S$ of finite places of $K$ if and only if Campana's Orbifold Mordell conjecture holds for $(\mathbb{G}_m, \tfrac{1}{2}[1])$.  This basic example carries a natural $\mathbb{G}_m$-action, and the quotient stack is an Artin stack parametrizing points on a C-pair. This leads to the introduction of ``Campana stacks'', which encode morphisms of C-pairs in a manner analogous to the role of root stacks for integral points satisfying prescribed divisibility conditions.
\end{abstract}

\maketitle

\thispagestyle{empty}

\tableofcontents

\section{Introduction}

The first aim of this paper is to present new examples of weakly special varieties that are not Campana-special.
Our basic example is the complement of the plane curve $x^2y^3=1$ in $\mathbb{A}^2$.
This yields a two-dimensional weakly special but nonspecial variety and shows that the notions ``weakly special'' and ``(Campana-)special'' already differ in dimension two. 
 
\begin{alphtheorem}[{Corollary \ref{cor:ab_puncture}}]\label{thm:A2_minus_x2y3}
The surface $\mathbb{A}^2 \setminus Z(x^2y^3-1)$ is weakly special but not special. 
\end{alphtheorem}

To prove that $\mathbb{A}^2 \setminus Z(x^2y^3-1)$ is weakly special, we apply a fibration theorem (Theorem~\ref{thm:family_of_ws_vars_intro}).
This result is established in Section~\ref{section:ws}, where we also construct further weakly special nonspecial varieties.

\subsection{Weakly special varieties}

Let $X$ be a smooth projective variety over a field $k$ of characteristic zero, and let $D$ be an snc divisor on $X$.
Recall that the pair $(X,D)$ is said to be of \emph{general type} if $K_X + D$ is big.  
If $U$ is a variety over $k$, we say that $U$ is of \emph{log-general type} if there exists a proper birational surjective morphism $V\to U$ with $V$ smooth and an snc compactification $X$ of $V$ such that $(X,D)$ is of general type, where $D := X \setminus V$. 
A rational map $X \ratmap Y$ of varieties is \emph{strictly rational} if there exists a proper birational surjective morphism $X' \to X$ such that the composition $X' \to X \ratmap Y$ is a morphism.

\begin{definition}\label{def:ws}
A smooth quasi-projective variety $X$ over a field $k$ of characteristic zero with algebraic closure~$\overline{k}$ is \emph{weakly special} if no finite \'etale cover $X'\to X_{\overline{k}}$ admits a dominant strictly rational map $X'\ratmap Y$ to a positive-dimensional variety of log-general type $Y$ over~$\overline{k}$.
A geometrically integral variety $X$ over $k$ is \emph{weakly special} if there exists a proper birational surjective morphism $X'\to X$ with $X'$ smooth quasi-projective, geometrically integral, and weakly special. 
A finite type separated scheme over $k$ is \emph{weakly special} if every irreducible component of $X_{\overline{k}}$ is weakly special.
\end{definition}

If $X$ is a smooth quasi-projective connected curve over $\overline{k}$, then $X$ is weakly special if and only if it is isomorphic to $\mathbb{P}^1$, $\mathbb{A}^1$, $\mathbb{A}^1\setminus \{0\}$, or a smooth proper curve of genus one.
If $X \subseteq A$ is an integral closed subvariety of an abelian variety $A$ over $\overline{k}$, then $X$ is weakly special if and only if it is a translate of an abelian subvariety.

\subsection{Orbifold base}
Given a dominant morphism of smooth varieties $X \to S$, we follow Campana (see \cite[Definition~4.2]{Ca11}) and define a $\mathbb{Q}$-divisor on $S$ that keeps track of the inf-multiple fibres in codimension one; this divisor is referred to as the ``orbifold base divisor'' of the fibration.

Let $S$ be an integral regular noetherian scheme, and let $S^{(1)}$ denote the set of codimension one points of $S$.
For $s \in S^{(1)}$, let $D_s$ be its closure in $S$.
Let $X$ be an integral regular noetherian scheme and let $f \colon X \to S$ be a dominant finite type morphism.
For $s \in S^{(1)}$, let $X_{D_s}$ denote the scheme-theoretic inverse image of $D_s$ in $X$. 

\begin{definition} 
We define the \emph{inf-multiplicity $m_s$ of $X_s$} and the \emph{gcd-multiplicity $m_s^+$ of $X_s$} as follows.
If $X_{s}$ is empty, we define $m_s=\infty$ and $m_s^+=\infty$.
If $X_{s}$ is non-empty, then $X_{D_s}$ is a non-trivial divisor dominating $D_s$ and we can write 
\[ X_{D_s} = R + \sum_{i=1}^n a_i F_i , \]
where $R$ is an effective $f$-exceptional divisor (i.e., none of its components dominate $D_s$), each $F_i$ is an irreducible component of $X_{D_s}$ dominating $D_s$, and $a_i\in \mathbb{Z}_{\geq 1}$.
We then define $m_s = \inf(a_i)$ and $m_s^+ = \gcd(a_i)$.   
\end{definition}

Note that $m_s$ is finite if and only if $m_s^+$ is finite.
Moreover, in this case, $m_s^+$ divides $m_s$.

\begin{definition}
We say that $X_s$ is \emph{inf-multiple} if $m_s \in \mathbb{Z}_{\geq 2} \cup \{\infty\}$.
We say that $X_s$ is \emph{divisible} if $m_s^+ \in \mathbb{Z}_{\geq 2} \cup \{\infty\}$.
\end{definition} 

If $X$ is regular and $f$ is generically smooth (for instance, if everything is over a field of characteristic zero), we define the \emph{orbifold base divisor $\Delta_f$ of $f$} to be the Weil $\mathbb{Q}$-divisor $\sum_{s\in S^{(1)}} \left(1-\frac{1}{m_s}\right) D_s$ and refer to the pair $(S,\Delta_f)$ as the \emph{orbifold base of $f$}.

\subsection{Special varieties} \label{section:c-pair}
   
Let $f \colon X \ratmap Y$ be a dominant rational map of smooth varieties over a field of characteristic zero.
Following Campana \cite[Definition~1.3.1]{Ca04} (see also \cite[\S 1]{BJRSeveri}), we say that $f$ is of \emph{general type} if, for every smooth proper variety $Y'$, every birational map $Y \ratmap Y'$, and every smooth variety $X'$ with a proper birational morphism $X'\to X$ such that $f'\colon X'\ratmap Y'$ is a morphism, the $\mathbb{Q}$-divisor $K_{Y'}+ \Delta_{f'}$ is big. 

\begin{definition}\label{definition:special}  
A variety $X$ over a field $k$ of characteristic zero is \emph{(Campana-)special} if some resolution of singularities $X'$ of $X_{\overline{k}}$ admits no dominant rational map of general type onto a positive-dimensional smooth variety over $\overline{k}$.
\end{definition}

Smooth projective rationally connected varieties are special \cite[Theorem~3.22]{Ca04}.
Abelian varieties, K3 surfaces, and Enriques surfaces are also special; more generally, all smooth projective varieties of Kodaira dimension zero are special (see \cite[Theorem~5.1]{Ca04}).
Since a finite \'etale cover of a special variety is special by Campana's theorem \cite[Th\'eor\`eme~10.11]{Ca11}, every special variety is weakly special.

On the other hand, the variety $X =\mathbb{A}^2\setminus Z(x^2y^3-1)$ is not special (Theorem \ref{thm:A2_minus_x2y3}), as the dominant morphism $X\to \mathbb{A}^1\setminus \{1\}\subseteq \mathbb{P}^1$ given by $(x,y)\mapsto x^2y^3$ has orbifold base $(\mathbb{P}^1, [1] + [\infty] + \frac{1}{2}[0])$ and is hence a map of general type.

\begin{definition}
A finite type scheme $X$ over a number field $K$ is \emph{arithmetically special} (or \emph{satisfies potential density of integral points}) if there exists a finite field extension $L/K$, a finite set of finite places $S$ of $L$, and a finite type model $\mathcal{X}$ for $X$ over $\mathcal{O}_{L,S}$ such that the image of $\mathcal{X}(\mathcal{O}_{L,S})$ in $X(L)$ is dense in $X$.
\end{definition}

Lang and Vojta conjectured the non-density of integral points on positive-dimensional varieties of log-general type (see \cite[Conjecture~5.7]{Langconj} for the projective case and \cite[Conjecture~4.3]{Vojta3} for the quasi-projective case).
Following~\cite[Remark~3.2.1]{CTSSD}, an application of the Chevalley--Weil theorem predicts that a variety $X$ over a number field $K$ satisfying potential density of integral points must be weakly special. 
 
\begin{conjecture}[Lang--Vojta + Chevalley--Weil] \label{conj:lv}
Let $X$ be a variety over a number field.
If $X$ satisfies potential density of integral points, then $X$ is weakly special.  
\end{conjecture}

In the projective setting, Harris--Tschinkel, following suggestions of Abramovich and Colliot-Th\'el\`ene, asked whether the converse of Conjecture \ref{conj:lv} holds (see \cite[Conjecture~1.2]{HarrisTschinkel}).
More generally, one may ask whether a weakly special variety over a number field is arithmetically special.
We record this as a conjecture (although it is expected to fail in general; see the discussion below).

\begin{conjecture}[Weakly Special Conjecture]\label{conj:ws}
If $X$ is a weakly special variety over a number field $K$, then $X$ has a potentially dense set of integral points.
\end{conjecture}

Conjecture \ref{conj:ws} holds for abelian varieties, Enriques surfaces, and certain K3 surfaces and Calabi--Yau threefolds (see \cite{HassettSurvey, HassettTschinkel, BogHalPazTan, BT3}).
Nevertheless, it is expected to fail for several weakly special varieties.
Another conjecture, formulated by Campana in \cite{Ca04, Ca11}, is the following.

\begin{conjecture}[Campana]\label{conj:special}
Let $X$ be a variety over a number field $K$. Then $X$ is special if and only if the set of integral points on $X$ is potentially dense.
\end{conjecture}
 
The first examples of weakly special nonspecial varieties were constructed by Bogomolov--Tschinkel \cite{BT}.
These threefolds fail to be special because they dominate a Deligne--Mumford root stack of general type, indicating that the Weakly Special Conjecture contradicts the analogue of Lang--Vojta for Deligne--Mumford stacks.  
In \cite{BJR, GRTW, RTW} (resp.\ \cite{CP07, CampanaWinkelmannBrody}) some of the threefolds of Bogomolov and Tschinkel were used to disprove certain function field (resp. complex-analytic) analogues of the Weakly Special Conjecture.

Subsequently, \cite{BCJW} established the existence of nonspecial threefolds over $\mathbb{Q}$ that are weakly special but do not dominate any (Deligne--Mumford) root stack of general type.
Using these threefolds, \cite{BCJW} further shows that the Weakly Special Conjecture contradicts Campana's Orbifold Mordell Conjecture discussed below.

\subsection{Families of weakly special varieties} \label{section:weakly_special_examples}

To prove that $\mathbb{A}^2 \setminus Z(x^2 y^3-1)$ is weakly special (Theorem \ref{thm:A2_minus_x2y3}), we show that the total space of a family of weakly special varieties over a weakly special base is itself weakly special, provided that the family contains no divisible fibres.
This criterion was first proved in the projective setting in \cite[\S 2]{BCJW}; here we extend it to the quasi-projective case.
 
\begin{alphtheorem}[{Theorem~\ref{thm:family_of_ws_vars}}] \label{thm:family_of_ws_vars_intro}
Let $k$ be a field of characteristic zero.
Let $f \colon X \to S$ be a surjective morphism of smooth varieties over $k$ whose generic fibre is geometrically connected (hence non-empty).  
Assume that the following hold.
\begin{enumerate}
\item The morphism $f$ is proper or flat.
\item The variety $S$ is weakly special.
\item The set of $s$ in $S(\overline{k})$ with $X_s$ a weakly special variety is dense in $S$. 
\item For every codimension one point $s$ in $S$, the fibre $X_s$ is not divisible.
\end{enumerate}
Then $X$ is weakly special.
\end{alphtheorem}

The fact that $X =\mathbb{A}^2\setminus Z(x^2y^3-1)$ is weakly special (Theorem \ref{thm:A2_minus_x2y3}) then follows from applying Theorem \ref{thm:family_of_ws_vars_intro} to the morphism $X \to \mathbb{A}^1\setminus \{1\}$ given by $(x,y)\mapsto x^2y^3$.   
This construction yields further examples of weakly special nonspecial varieties; see Section~\ref{section:ws}.

\subsection{Campana maps and C-pairs} \label{section:orbifold_mordell} 

Let $X$ be an integral locally factorial noetherian scheme, and let $\Delta = \sum_i (1-\frac{1}{m_i}) D_i$ be a $\mathbb{Q}$-divisor, where $m_i\in \mathbb{Z}_{\geq 1}\cup \{\infty\}$.
We refer to $(X,\Delta)$ as a \emph{C-pair} and to $m_i$ as the \emph{multiplicity} of $D_i$ in $\Delta$.

If $X$ is a locally factorial variety over a field $k$ and $(X,\Delta)$ is a C-pair, we say that $(X,\Delta)$ is \emph{smooth} if $X$ is smooth and the support of $\Delta$ has simple normal crossings.
We say that a C-pair $(X,\Delta)$ is \emph{proper} if $X$ is proper over $k$. 
A smooth proper C-pair $(X,\Delta)$ is of \emph{general type} if $K_X + \Delta$ is a big $\mathbb{Q}$-divisor. 

Let $X$ be a regular noetherian scheme. Let $D_1,\ldots,D_n$ be integral divisors on $X$, and for each $1\leq i\leq n$ let $m_i\in \mathbb{Z}_{\geq 1}\cup \{\infty\}$.  Write $\Delta = \sum_{i=1}^n(1-\frac{1}{m_i})D_i$ and consider the C-pair  $(X,\Delta)$. 
The \emph{logarithmic part} $\floor{\Delta}$ of $(X,\Delta)$ is the union of those $D_i$ for which $m_i =\infty$.
Let $T$ be an integral regular noetherian scheme. A morphism $f\colon T\to X$ is a \emph{morphism of C-pairs (or Campana map)}, denoted $f\colon T\to (X, \Delta)$, if $f$ factors through $X\setminus\floor{\Delta}$ and, for every $i$ with $f(T)\not\subseteq D_i$, the coefficient of each irreducible component $E$ of $f^{-1}(D_i)$ in $f^*D_i$ is at least $m_i$.
Note that we allow for $f\colon T\to X$ to factor over $D_i$, as long as $m_i$ is finite.
Thus, our definition slightly differs from that of Campana in \cite[Definition~2.3]{Ca11}.  

It is useful to view Campana's condition functorially.
Consider a smooth variety $X$ over $\mathbb{C}$, a smooth divisor $D$ and an integer $m\geq 2$. 
Then any subfunctor of $\mathrm{Hom}(-,X)$ which on the category of smooth varieties is defined by
\[ T \mapsto \{\text{morphisms } T\to (X,(1-\tfrac{1}{m})D)\} \]
cannot be an fppf sheaf, and in particular is not be representable by a scheme; see Remark~\ref{remark:sheaf_vs_stack} for more discussion of this point.
We will see that there is nevertheless an algebraic stack parametrizing Campana maps.
The construction of this stack is based on the observation that in the definition of a morphism of C-pairs we should not only consider the condition on the multiplicities of the irreducible components of $f^* D$ but also remember the possible decompositions of $f^* D$ as a sum of effective Cartier divisors.

\subsection{Orbifold Mordell}

Let $(X, \sum_{i=1}^n(1-\frac{1}{m_i})D_i)$ be a C-pair, where $X$ is a smooth variety over a number field $K$.
Choose a finite set $S$ of finite places of $K$ and a regular model  $\mathcal{X}$ for $X$ of finite type over $\mathcal{O}_{K,S}$. Let $\mathcal{D}_i$ be the closure of $D_i$ in $\mathcal{X}$ and define $\Delta =  \sum_{i=1}^n\left(1-\frac{1}{m_i}\right)\mathcal{D}_i $.  We say that  $(X, \sum_{i=1}^n(1-\frac{1}{m_i})D_i)$  is \emph{arithmetically special}  (or \emph{satisfies potential density of integral points}) if there is a finite field extension $L/K$ and  a  finite set $T$ of finite places of $L$ such  that the image of $(\mathcal{X},   \Delta)(\mathcal{O}_{L,T})$  in $X(L)$ is dense. (This property is independent of the  choice of $S$ and   model for $X$ over $\mathcal{O}_{K,S}$.) 

\begin{conjecture}[Orbifold Mordell]\label{conj:orb_mor} 
Let $(X,\Delta)$ be a C-pair over a number field $K$ with $X$ a smooth proper curve over $K$ of genus $g$.
Assume that $2g-2 +\deg \Delta >0$.
Then the C-pair $(X,\Delta)$ does not satisfy potential density of integral points.  
\end{conjecture}

Campana's Orbifold Mordell Conjecture first appeared in \cite[Conjecture~4.5]{Ca05} and is a very special case of Campana's generalization of Lang's conjecture for C-pairs \cite[Conjecture~13.23]{Ca11}.
The abc conjecture over number fields implies Orbifold Mordell; this was shown by Smeets \cite[Theorem~5.3]{Smeets} building on the work of Elkies \cite{Elkies}, Abramovich \cite{AbramovichBirGeom} and Campana \cite[Remarque~4.8]{Ca05}.
If $X$ is of genus at least two or $\Delta$ only has infinite multiplicities, then the Orbifold Mordell Conjecture follows from Faltings's finiteness theorem \cite{Faltings2}.
The function field analogue of Orbifold Mordell was proven by Campana in characteristic zero \cite{Ca05} and Kebekus--Pereira--Smeets in positive characteristic \cite{Smeets2}.
An ``analytic'' analogue of Orbifold Mordell is established in \cite[Corollary~4]{CampanaWinkelmannBrody}.

\subsection{Relating Orbifold Mordell to the Weakly Special Conjecture}
 
Recall that the quasi-projective surface $X=\mathbb{A}^2\setminus Z(x^2y^3-1)$ over $\mathbb{Q}$ is not special  (Theorem~\ref{thm:A2_minus_x2y3}).
According to Conjecture~\ref{conj:special}, one therefore expects that $X$ does not have a potentially dense set of integral points.
To study its integral points, observe that $X$ admits a faithful $\mathbb{G}_m$-action given by
\[ \lambda \cdot (x,y) = (\lambda^3 x, \lambda^{-2} y). \]

We prove that the potential density of integral points on $X$ is \emph{equivalent} to the potential density of integral points on the quotient stack $\mathcal{X}=[X/\mathbb{G}_m]$, as well as the potential density of integral points on the C-pair $(\mathbb{A}^1\setminus \{0\}, \frac{1}{2} [1])$.
Let $R$ be a principal ideal domain and let $x\in R$.
For a prime element $p$ of $R$, write $v_p(x)$ for the largest integer $n$ such that $p^n$ divides $x$, with the convention $v_p(0)=\infty$.
For a positive integer $m$, we say that $x\in R$ is \emph{$m$-full} if for every prime $p$ we have $v_p(x)=0$ or $v_p(x)\ge m$.
 
\begin{alphtheorem}[Theorem~\ref{thm:ws_contradicts_abc_text}] \label{thm:ws_contradicts_abc}
Let $X = \mathbb{A}^2_{\mathbb{Z}}\setminus Z(x^2y^3-1)$ and let $\mathcal{X} = [X/\mathbb{G}_m]$ with respect to the $\mathbb{G}_m$-action defined above.
Then the following are equivalent. 
\begin{enumerate}
\item The variety $X_{\mathbb{Q}}$ satisfies potential density of integral points, i.e.\ there is a number field $K$ and a finite set of finite places $S$ of $K$ such that $X(\mathcal{O}_{K,S})$ is dense.
\item There is a number field $K$ and a finite set $S$ of finite places of $K$ such that the image of $\pi_0(\mathcal{X}(\mathcal{O}_{K,S}))$ in $\pi_0(\mathcal{X}(\overline{K}))$ is infinite.\footnote{Here we denote the set of isomorphism classes in a groupoid $\mathcal{G}$ by $\pi_0(\mathcal{G})$.  In the terminology of \cite[\S 4]{JLalg},  the  stack  $\mathcal{X}$ is not arithmetically hyperbolic.}   
\item The C-pair $(\mathbb{P}^1_{\mathbb{Q}}, \frac{1}{2} [0] + [1] + [\infty])$ is arithmetically special (and thus fails ``Orbifold Mordell'').
\item There is a number field $K$ and a finite set of finite places $S$ of $K$ such that $\mathcal{O}_{K,S}$ is a principal ideal domain and $\{x \in \mathcal{O}_{K,S}^\times~|~x-1~\textrm{is}~2\textrm{-full}\}$ is infinite.
\end{enumerate}
\end{alphtheorem} 

It was shown in \cite{BCJW} that there is a weakly special smooth projective threefold $V$ over $\mathbb{Q}$ such that the Orbifold Mordell Conjecture for $(\mathbb{P}^1, \frac{1}{2} [0] + \frac{1}{2} [1]  + \frac{1}{2} [2] + \frac{1}{2} [3] + \frac{1}{2} [\infty])$ implies that $V(K)$ is not dense for any number field $K$.
In particular, the Weakly Special Conjecture contradicts the Orbifold Mordell Conjecture.
This latter statement is also implied by Theorem~\ref{thm:ws_contradicts_abc}.
Theorem~\ref{thm:ws_contradicts_abc} however demonstrates a stronger connection between the Weakly Special Conjecture (as formulated in Conjecture~\ref{conj:ws}) and the Orbifold Mordell Conjecture:
in the situation discussed in Theorem~\ref{thm:ws_contradicts_abc}, not only do integral points on the surface $X$ descend to integral points on a certain C-pair, but conversely, the integral points on the C-pair can be lifted back to the surface $X$.
In particular, deciding whether $X$ satisfies potential density of integral points naturally leads to the study of integral points on C-pairs. 

The equivalence of $(i)$ and $(ii)$ is a consequence of  a generalized Chevalley--Weil theorem for $G$-torsors over stacks with $G$ an algebraic group; this version of the Chevalley--Weil theorem is established in  Appendix \ref{app:cw}.  Although the generalized Chevalley--Weil theorem gives a conceptually clear proof of the equivalence,  Theorem \ref{thm:ws_contradicts_abc} can also be proven by hand; see Section~\ref{section:remarks_on_orbi_mordell} for details.

\subsection{Campana stacks} \label{section:intro_campana}

The appearance of the stack $[X/\mathbb{G}_m]$ in our study of integral points on $X = \mathbb{A}^2\setminus Z(x^2y^3-1)$ and of the C-pair $(\mathbb{A}^1\setminus \{0\}, \frac{1}{2}[1])$ reflects a more general principle that has already appeared several times in this introduction:

\medskip

\centerline{\emph{Morphisms to C-pairs should factor through suitable algebraic stacks.}}

\medskip 
In the divisible setting, where the condition in Campana's definition of a morphism of C-pairs that multiplicities ``are at least $m_i$'' is replaced by the stronger requirement that they ``are divisible by $m_i$'', this principle is well understood.
Such morphisms factor through root stacks; see Example \ref{example:darmon} for a brief discussion.

In the general setting of morphisms of C-pairs, we construct algebraic stacks (with non-finite inertia) that naturally encode such morphisms; see Theorem \ref{thm:campana_stack_intro} below for a precise statement.

To define the Campana stack, we first consider the usual action of $\mathbb{G}_m$ on $\mathbb{A}^1$ and form the quotient stack $[\mathbb{A}^1/\mathbb{G}_m]$.
For any scheme $T$, a morphism $T \to [\mathbb{A}^1/\mathbb{G}_m]$ then corresponds to a \emph{generalized effective Cartier divisor} on $T$, that is, a pair $(\mathcal{L}, s)$ consisting of a line bundle $\mathcal{L}$ on $T$ and a global section $s \in \Gamma(\mathcal{L}, T)$ up to scaling; for more details, see \cite[\S 10.3]{OlssonBook} and \cite{AGV, Cadman}. 
Let $X$ be a scheme and let $D$ be an effective Cartier divisor on $X$.
If $f_D\colon X\to [\mathbb{A}^1/\mathbb{G}_m]$ is the associated moduli morphism, then the $n$-th root stack $\sqrt[n]{D/X}$ is defined by the Cartesian diagram
\[
\xymatrix{ \sqrt[n]{D/X} \ar[rr] \ar[d] & &  [\mathbb{A}^1/\mathbb{G}_m] \ar[d]^{z\longmapsto z^n} \\
 X \ar[rr]_{f_D} & & [\mathbb{A}^1/\mathbb{G}_m]}
\]
That is, the stack $\sqrt[n]{D/X}$ is obtained by pulling back the morphism $X\to [\mathbb{A}^1/\mathbb{G}_m]$ along the map $[\mathbb{A}^1/\mathbb{G}_m]\to [\mathbb{A}^1/\mathbb{G}_m]$ given by $z\mapsto z^n$; the latter map  may be viewed as the ``universal divisible fibre of multiplicity $n$''.

Let $r\geq 1$ and let $a=(a_1,\ldots,a_r)$ be an $r$-tuple of positive integers.
We define the morphism $p_a\colon \mathbb{A}^r\to \mathbb{A}^1$ by
\[ p_a(x_1,\ldots,x_r)=x_1^{a_1}\cdots x_r^{a_r}. \]
This induces a morphism of quotient stacks
\[ P_a\colon [\mathbb{A}^1/\mathbb{G}_m]^r \to [\mathbb{A}^1/\mathbb{G}_m]. \]
The morphism $P_a$ may be viewed as the ``universal inf-multiple fibre with multiplicities $a_1,\ldots,a_r$''. 

\begin{definition}\label{def:campana_stack_intro}
We define the \emph{Campana stack $\mathcal{C}_{X,D,a}$ of $(X,D,a)$} via the following Cartesian diagram.
\[
\xymatrix{ \mathcal{C}_{X,D,a}  \ar[rr] \ar[d] & & [\mathbb{A}^1/\mathbb{G}_m]^r \ar[d]^{P_a} \\ X \ar[rr]_{f_D} & & [\mathbb{A}^1/\mathbb{G}_m]}
\]
We define the \emph{Campana space $U_{X,D,a}$ of $(X,D,a)$} via the following Cartesian diagram.
\[
\xymatrix{ U_{X,D,a} \ar[rr] \ar[d]_{\mathbb{G}_m^r-\text{torsor}} & &  \mathbb{A}^r \ar[d]^{\mathbb{G}_m^r-\text{torsor}} \\ 
\mathcal{C}_{X,D,a}  \ar[rr]   & & [\mathbb{A}^1/\mathbb{G}_m]^r  }
\]
\end{definition}

Now let $\Delta=(1-\frac{1}{m})D$, where $D$ is a prime divisor on $X$ and $m>1$ is an integer.
The Campana stack $\mathcal{C}_{X,\Delta}$ of $(X,\Delta)$ is defined by $\mathcal{C}_{X,\Delta}=\mathcal{C}_{X,D,a}$, where $a=(m,m+1,\ldots,2m-1)$ is the minimal set of generators of the semigroup $\mathbb{Z}_{\geq m}$.\footnote{We warn the reader that $\mathcal{C}_{X,\Delta}$ is not a Deligne--Mumford stack.}
(This choice of $a$ encodes the condition that multiplicities in a morphism of C-pairs are required to be at least $m$.)
The Campana space $U_{X,\Delta}=U_{X,D,a}$ admits a morphism of C-pairs $U_{X,\Delta}\to (X,\Delta)$ which is a $\mathbb{G}_m^m$-torsor over $X\setminus\supp(\Delta)$.

In general, the stacks $\mathcal{C}_{X,\Delta}$ and the spaces $U_{X,\Delta}$ are defined by iterating the above construction (see Definition \ref{def:campana_stack_2}).
Our first observation is that the stack $\mathcal{C}_{X,\Delta}$ does in fact satisfy the desired property (see Theorem \ref{thm:C-pair-Campana-stack}):

\begin{alphtheorem}[{Theorem~\ref{thm:C-pair-Campana-stack}}] \label{thm:campana_stack_intro}
Let $T$ be an integral regular noetherian scheme.
Then, a morphism $T \to X$ defines a morphism of C-pairs $T \to (X,\Delta)$ if and only if it factors through $\mathcal{C}_{X,\Delta}$.
\end{alphtheorem}

Intuitively, the Campana space $U_{X,\Delta}$ may be viewed as an ``untwisting'' of the C-pair structure on $(X,\Delta)$.
In this way, questions about the arithmetic or geometry of C-pairs can be translated into questions about quasi-projective varieties.
For example, the Campana space of a C-pair detects potential density of integral points on the C-pair.  

\begin{alphtheorem}[{Theorem~\ref{thm:U_is_ar_spec_iff_XDelta_is_ar_spec}}]\label{thm:potential_density_campana_space}
Let $(X,\Delta)$ be a C-pair with $X$ a smooth variety over a number field $K$, and let $U_{X,\Delta}$ be its Campana space.
Then the variety $U_{X,\Delta}$ has a potentially dense set of integral points if and only if $(X, \Delta)$ has a potentially dense set of integral points.
\end{alphtheorem}

Theorem \ref{thm:potential_density_campana_space} yields the following consequence.

\begin{alphcorollary}[{Corollary~\ref{cor:campana_conjectures_equivalence}}] \label{cor:orb_mordell_and_campana_conj}
Assume that, for every number field $K$, every smooth quasi-projective variety over $K$ satisfying potential density of integral points is special.
Then Campana's Orbifold Mordell conjecture (Conjecture~\ref{conj:orb_mor}) holds.  
\end{alphcorollary}

Theorem~\ref{thm:potential_density_campana_space} and Corollary~\ref{cor:orb_mordell_and_campana_conj} also admit the following analytic analogues.

\begin{alphtheorem}[{Theorem~\ref{thm:U_dense_entire_curve_iff_XDelta_dense_entire_curve}}]
Let $(X,\Delta)$ be a C-pair with $X$ a smooth quasi-projective variety over $\mathbb{C}$, and let $U_{X,\Delta}$ be its Campana space.
Then $U_{X,\Delta}$ has a dense entire curve if and only if $(X,\Delta)$ has a dense entire curve.
\end{alphtheorem}

\begin{alphcorollary}[{Corollary~\ref{cor:campana_conjectures_equivalence_analytic}}]
Assume that every smooth quasi-projective variety over $\mathbb{C}$ admitting a dense entire curve is Campana-special.
Then Campana's generalized Green--Griffiths--Lang conjecture (Conjecture~\ref{conj:gen_ggl}) holds.
\end{alphcorollary}

A basic question in Campana's theory of special varieties and C-pairs is whether a given pair $(X, \Delta)$ occurs as the orbifold base of some morphism $U \to X$.
The following result shows that such a morphism exists and can be constructed via the Campana space.
 
\begin{alphtheorem}[{Theorem~\ref{thm:properties_of_Campana_space}}] \label{thm:existence_of_U_intro}
Let $(X,\Delta)$ be a C-pair with $X$ a regular scheme over $\mathbb{Q}$ and $\supp(\Delta)$ a simple normal crossings divisor. 
Let $U$ be the Campana space of $(X,\Delta)$.
Then $U$ is regular and the morphism $U \to X$ is affine, flat, and of finite type with image $X\setminus\floor{\Delta}$.
Moreover, it satisfies the following properties.

\begin{enumerate}
\item $U\to X$ induces a morphism of C-pairs $U\to (X,\Delta)$ and its orbifold base divisor is $\Delta$.  
\item For every $x\in X\setminus \supp(\Delta)$, the fibre of $U\to X$ over $x$ is a smooth affine torus over the residue field $k(x)$.
\item The morphism $U\to X\setminus\floor{\Delta}$ has no divisible fibres.
\end{enumerate}
\end{alphtheorem}

The proof of Theorem~\ref{thm:existence_of_U_intro} and the construction of the Campana stack are motivated by the observation that, for every integer $m>1$, the morphism $(x_1,\ldots,x_m)\mapsto x_1^{m}\cdots x_m^{2m-1}$ defines an affine flat fibration $U\to \mathbb A^1_{\mathbb C}$ with orbifold base divisor $(1-\frac{1}{m})[0]$ and no divisible fibres.
It is unclear whether there exists a flat \emph{projective} fibration $V \to \mathbb A^1_{\mathbb C}$ with the same properties (any projective extension $\overline U\to \mathbb A^1_{\mathbb C}$ of $U \to \mathbb{A}^1_{\mathbb{C}}$ of our morphism would admit a section by Graber--Harris--Starr \cite{GraberHarrisStarr} and thus cannot have an inf-multiple fiber).

Combining the fibration criterion (Theorem \ref{thm:family_of_ws_vars_intro}) with Theorem \ref{thm:existence_of_U_intro} yields the following result.

\begin{alphcorollary}[{Corollary~\ref{cor:ws_text}}] \label{cor:ws_intro}
Let $(X, \Delta)$ be a smooth C-pair over a field of characteristic zero such that $X\setminus\floor{\Delta}$ weakly special.
Then the Campana space $U$ of $(X,\Delta)$ is a smooth weakly special variety.
If $(X, \Delta)$ is of general type and $\dim X > 0$, then $U$ is not special.
\end{alphcorollary}

\begin{con}
A variety over a field $k$ is a geometrically integral finite type separated $k$-scheme.
Concerning algebraic stacks, we follow the conventions of the Stacks Project \cite{stacks-project}.
\end{con}

\begin{ack}
A.J.\ was supported by the Netherlands Organization for Scientific Research NWO, through grant OCENW.M.23.206.
A.J.\ gratefully acknowledges support from the Laboratoire de Math\'ematiques de Bretagne Atlantique and a CNRS Poste Rouge. 
F.B.\ is grateful to Lisanne Taams for several discussions on stacky interpretations of Campana's theory of morphisms to C-pairs.
We thank Fr\'ed\'eric Campana, Stefan Kebekus, Erwan Rousseau, and Olivier Wittenberg for many valuable discussions.
\end{ack}

\noindent \textbf{AI Disclosure.} All mathematics in this paper is human-generated.  The second-named author used ChatGPT 5.4 during the early stages of the writing of this paper to improve typesetting and structure.  The entire paper is written by the authors.

\section{Families of quasi-projective weakly special varieties}\label{section:ws}
 
In this section, we prove Theorem \ref{thm:family_of_ws_vars_intro} (stated below as Theorem \ref{thm:family_of_ws_vars}).
  
\begin{lemma}\label{lemma:stein}
Let $X \to S$ be a proper surjective morphism of integral normal noetherian schemes over~$\mathbb{Q}$.
Suppose $S$ is regular.
Assume that, for every point $s$ of codimension one and every connected component $F$ of $X_s$, the divisor $F$ on $X \times_S \Spec(\mathcal{O}_{S,s})$ is not divisible.
Then, the Stein factorization $T \to S$ of $X \to S$ is finite \'etale. 
\end{lemma}
\begin{proof}
This is \cite[Lemma~2.1]{BCJW} except that $X$ is supposed to be regular in \emph{loc. cit.}. This assumption can however be replaced by $X$ normal as its proof shows.
\end{proof}

Recall that we say that a rational map $X \ratmap Y$ of varieties is \emph{strictly rational} if there exists a proper birational surjective morphism $X' \to X$ such that the composition $X' \to X \ratmap Y$ is a morphism.
We will need the following criterion for a rational map to be strictly rational.
(In the terminology of \cite[Tag~0ETQ]{stacks-project}, this criterion says that being a strictly-rational map is ``local for the h topology'' on the source.)

\begin{lemma} \label{lemma:strictly_rational_criterion}
Let $X$, $Y$, and $Z$ be varieties.
Let $f \colon X \to Y$ be a surjective morphism and let $g \colon Y \ratmap Z$ be a rational map such that $g \circ f \colon X \ratmap Z$ is strictly rational.
Assume that $f$ is flat or proper.
Then $g$ is strictly rational.
\end{lemma}
\begin{proof}
Let $X' \to X$ be a proper birational surjective morphism such that $X' \to X \ratmap Z$ is a morphism.

We first treat the case that $f$ is proper.
In this case, precomposing $f$ with $X' \to X$, we may assume that $X = X'$, i.e.\ we may assume that $g \circ f$ is a morphism $X \to Z$.
Let $\Gamma \subseteq X \times Z$ be its graph and let $f(\Gamma)$ denote its image in $Y \times Z$.
Since $f$ is proper, $f(\Gamma) \subseteq Y \times Z$ is a closed subset of $Y \times Z$, hence a variety.
The morphism $f(\Gamma) \to Y$ is proper (since $f$ is proper), surjective (since $f$ is surjective), and birational (since $g$ is a rational map), and the morphism $f(\Gamma) \to Z$ is a model for the rational map $g$.
Hence $g$ is strictly rational.

We now treat the case that $f$ is flat.
In this case, let $Z \subseteq \overline{Z}$ be a compactification.
Then $g$ induces a strictly rational map $Y \ratmap \overline{Z}$.
Let $Y' \to Y$ be a proper surjective birational map such that $Y' \to Y \ratmap \overline{Z}$ is a morphism.
Consider the fibre product $X'' := X \times_Y Y'$.
Being a base change of $f$, the natural morphism $X'' \to Y'$ is flat.
Consequently, every irreducible component of $X''$ dominates $Y'$, which implies that $X''$ is irreducible.
Moreover, $X''$ comes equipped with a strictly rational map $X'' \ratmap Z$ and a morphism $X'' \to \overline{Z}$, which generically agree, and thus must agree everywhere.
It follows that the morphism $X'' \to \overline{Z}$ factors over $Z$.
Since $X'' \to Y'$ is a base change of $f$, it is surjective, and hence the morphism $Y' \to \overline{Z}$ also factors over $Z$.
This exactly means that $g$ is strictly rational.
\end{proof}

\begin{remark}
In the statement of Lemma~\ref{lemma:strictly_rational_criterion}, the assumption that $f$ is flat or proper is necessary.
For example, one can take $Y$ to be a smooth surface, and $X$ to be the blowup of $Y$ at some point $y \in Y$, minus a closed point of the exceptional divisor.
In this case, the natural morphism $f \colon X \to Y$ is neither flat nor proper.
Then, taking $X = Z$ and $g$ to be the inverse of the blowdown morphism, we find that $g \circ f$ is the identity morphism (so in particular strictly rational), but $g$ is not strictly rational. 
\end{remark}

\begin{lemma} \label{lemma:ws_family_step_one}
Let $k$ be a field of characteristic zero and let $X$ and $Y$ be normal varieties over $k$.
Let $f \colon X \to Y$ be a surjective morphism with geometrically connected generic fibre.
Suppose that $f$ is proper or that $f$ is flat.
Suppose that $Y$ is weakly special and that the set of $y$ in $Y(\overline{k})$ with $X_y$ a weakly special variety is dense in $Y$.
Then $X$ does not admit a dominant strictly rational map onto a positive-dimensional variety of log-general type.
\end{lemma}
\begin{proof}
We may assume that $k$ is algebraically closed.
Let $g \colon X \ratmap Z$ be a dominant strictly rational map onto a smooth variety of log-general type; we want to show that $Z$ is a point.

Let $X' \to X$ be a proper birational map such that the composition $X' \to X \ratmap Z$ is a morphism.
For a general point $y \in Y$, the natural morphism $X'_y \to X_y$ is proper birational, and thus, there is a dense set of $y$ for which the fibre $X'_y$ is weakly special.
The scheme-theoretic images of the $X'_y$ under the morphism $X' \to Z$ form a covering family of closed subvarieties of the log-general type variety $Z$, and are thus of log-general type.
Hence, for a general $y \in Y$, the scheme-theoretic image of $X'_y$ in $Z$ must be zero-dimensional.
As $f$ has geometrically connected generic fibre, this means that $g$ contracts the general fibre of $f$ to a point.
Consequently, the rational map $X \ratmap Z$ factors over a rational map $Y \ratmap Z$, which, by Lemma~\ref{lemma:strictly_rational_criterion}, is strictly rational.
Since $Y$ is weakly special, this forces $Z$ to be a point, as desired.
\end{proof}

\begin{theorem}\label{thm:family_of_ws_vars}
Let $k$ be a field of characteristic zero.
Let $f \colon X \to S$ be a surjective morphism of smooth varieties over $k$ whose generic fibre is geometrically connected (hence non-empty).  
Assume that the following hold.
\begin{enumerate}
\item The morphism $f$ is proper or flat.
\item The variety $S$ is weakly special.
\item The set of $s$ in $S(\overline{k})$ with $X_s$ a weakly special variety is dense in $S$. 
\item For every codimension one point $s$ in $S$, the fibre $X_s$ is not divisible.
\end{enumerate}
Then $X$ is weakly special.
\end{theorem}
\begin{proof} 
We may assume that $k$ is algebraically closed.
Let $X' \to X$ be a finite étale morphism with $X'$ connected.
Then $X'$ is smooth.
To prove the theorem, it suffices to show that $X'$ does not admit a dominant strictly rational map onto a positive-dimensional variety of log-general type.
 
Let $\overline{X}$ be a smooth (integral) variety containing $X$ as a dense open subscheme such that $X \to S$ extends to a proper morphism $\overline{X} \to S$.
Let $\overline{X'} \to \overline{X}$ be the normalization of $\overline{X}$ in the function field of $X'$.
Since the generic fibre of $X \to S$ is geometrically connected, the generic fibre of $\overline{X} \to S$ is geometrically connected as well.
Consequently, for every $s$ of codimension one, the fibre $\overline{X}_s$ is connected.

Let $s$ be a point of codimension one in $S$ and let $F$ be a connected component of $\overline{X'}_s$.
We claim that $F$ is not divisible. To show this, we prove that $F \cap X_s'$ is not divisible.
First, the morphism $F \to \overline{X}_s$ is finite, since it is the composition of the closed immersion $F \subseteq \overline{X'}_s$ and the finite morphism $\overline{X'}_s \to \overline{X}_s$.
In particular, $F \to \overline{X}_s$ is closed.
By \cite[Tag~0F32]{stacks-project}, the morphism $\overline{X'} \to \overline{X}$ is universally open.
Thus, $F \to \overline{X}_s$ is open, since it is the composition of the open immersion $F \subseteq \overline{X'}_s$ and the open morphism $\overline{X'}_s \to \overline{X}_s$.
Since $\overline{X}_s$ is connected, this implies that $F \to \overline{X}_s$ is surjective (its image being open and closed).
As $F \cap X_s' \to X_s$ is the base change of $F \to \overline{X}_s$ along $X_s \subseteq \overline{X}_s$, it follows that $F \cap X_s' \to X_s$ is finite surjective.
Moreover, since $X' \to X$ is étale, the morphism $X_s' \to X_s$ is étale.
Since $F \cap X_s'$ is open in $X_s'$, it follows that $F \cap X_s' \to X_s$ is étale as well.
That is, $F \cap X_s' \to X_s$ is finite surjective and étale.
Therefore, the non-divisibility of $X_s$ implies that $F \cap X_s'$ is not divisible, as claimed.

As $s$ and $F$ were arbitrary, it now follows from Lemma~\ref{lemma:stein} that the Stein factorization $T \to S$ of $\overline{X'}\to S$ is finite étale.
In particular, since $S$ is weakly special, it follows that $T$ is weakly special as well.
Moreover, as the fibres of $X' \to T$ are finite étale covers of the fibres of $X \to S$, it follows that the set of $t$ in $T(k)$ with $X'_t$ a weakly special variety is dense in $T$.
The result now follows from Lemma~\ref{lemma:ws_family_step_one}.
\end{proof}

\begin{remark}
For non-smooth varieties, our definition of a weakly special variety (Definition~\ref{def:ws}) differs from that in \cite[Definition~2.1]{CDY} and \cite{Ca04}.
For an explicit example illustrating this, see \cite[\S 2]{BJL}.
\end{remark}

\begin{remark}
Some assumption concerning divisible fibres is necessary in Theorem~\ref{thm:family_of_ws_vars}.
A well-known example illustrating this is given in \cite[Remark~2.5]{BCJW}. 
\end{remark}

\begin{corollary}
Let $k$ be a field of characteristic zero.
Let $X$ and $S$ be smooth varieties over $k$.
Let $X \to S$ be a smooth surjective morphism with geometrically connected generic fibre.
If $S$ is weakly special and the set of $s$ in $S(\overline{k})$ with $X_s$ weakly special is dense, then $X$ is weakly special.
\end{corollary}
\begin{proof}
As smooth morphisms are flat, and smooth surjective morphisms have no divisible fibres, this follows from Theorem~\ref{thm:family_of_ws_vars}.
\end{proof}

\begin{corollary}\label{corollary:xa1etc}
Let $a_1 \leq \cdots \leq a_n$ be positive integers with $\gcd(a_1,\ldots,a_n) = 1$.
Define $X = \mathbb{A}^{n} \setminus Z(x_1^{a_1} \cdot \ldots \cdot x_n^{a_n}-1)$ and let $\Sigma \subset X$ be a closed subset of codimension at least two.
Then $U = X\setminus \Sigma$ is weakly special.
Furthermore, $U$ is special if and only if $a_1 = 1$. 
\end{corollary}
\begin{proof}
Consider the morphism $X \to \mathbb{A}^1\setminus \{1\}$ given by $(x_1,\ldots,x_n) \mapsto x_1^{a_1} \cdot \ldots \cdot x_n^{a_n}$.
It is surjective.
Since $\Sigma$ has codimension at least two in $X$, it cannot contain a fibre of $X\to \mathbb{A}^1\setminus\{1\}$; hence $U\to \mathbb{A}^1\setminus\{1\}$ remains surjective. 
Also, for every $s\neq 0$, the fibre $X_s$ of $X\to \mathbb{A}^1\setminus \{1\}$ over $s$ is isomorphic to $\mathbb{G}_m^{n-1}$.    

We claim that the morphism $U \to \mathbb{A}^1\setminus \{1\}$ has no divisible fibres.
Indeed, for all $s\neq 0$, the fibre $U_s$ is smooth (as it is a dense open of $X_s$).
Moreover, the fibre $U_0$ over $s=0$ decomposes as $\sum_{i=1}^n a_i F_i$, where the $F_i$ are the irreducible components of $U_0$.
Since $\gcd(a_1,\ldots,a_n)=1$ by assumption, we see that $U_0$ is not divisible either.

As $\Sigma$ is of codimension at least two in $X$, we see that for all $s$ in a dense open subset of $\mathbb{A}^1\setminus \{1\}$, the fibre $U_s$ of $U\to \mathbb{A}^1\setminus \{1\}$ over $s$ is a dense open subset of $\mathbb{G}_m^{n-1}$ whose complement is of codimension at least two.
In particular, as it is a ``big open'' of a smooth special variety, it is special by \cite[Theorem~G]{BJL}, hence weakly special.
Since $U \to \mathbb{A}^1\setminus \{1\}$ is flat, we thus see that it satisfies the hypotheses of Theorem~\ref{thm:family_of_ws_vars}.
Consequently, $U$ is weakly special. 
 
If $a_1 > 1$, the orbifold base of $U\to \mathbb{A}^1\setminus \{1\}$ is of general type, so that $U$ is not special.   
If $a_1 = 1$, then the morphism has no inf-multiple fibres, and thus $U$ is special by \cite[Lemma~2.10]{Bartsch}.
\end{proof}

We state the special case $n=2$ separately; Theorem~\ref{thm:A2_minus_x2y3} from the introduction is then the special case that $a=2$, $b=3$, and $\Sigma = \emptyset$.  

\begin{corollary}\label{cor:ab_puncture}
Let $1 \leq a \leq b$ be coprime integers.
Let $\Sigma \subset \mathbb{A}^2 \setminus Z(x^ay^b-1)=X$ be finite.
Then $U = X \setminus \Sigma$ is weakly special.
Furthermore, $U$ is special if and only if $a = 1$.
\end{corollary}
 
\begin{remark}
Let $X$ be a weakly special variety and let $f \colon X \to \mathbb{G}_m$ be a fibration of general type (with no divisible fibres) with orbifold base $(\mathbb{G}_m, \frac{1}{2} [1])$.
For $n\geq 1$, define $X_n$ via the following Cartesian diagram.
\[ \xymatrix{ X_n \ar[rr] \ar[d]_{f_n} & & X \ar[d]^{f} \\ \mathbb{G}_m \ar[rr]_{z\longmapsto z^n} & & \mathbb{G}_m} \]
Then $X_n$ is weakly special and $f_n\colon X_n\to \mathbb{G}_m$ is a fibration of general type with ``larger'' orbifold base $(\mathbb{G}_m, \frac{1}{2} \left([1] + [\zeta_n] +\ldots + [\zeta_n^{n-1}]\right))$, where $\zeta_n$ is a primitive $n$-th root of unity.
\end{remark}

\begin{remark} 
The weakly special varieties constructed here share a common feature: their fibrations of general type arise as quotients of affine varieties by torus actions (with possibly negative weights).
This perspective motivates the notion of a Campana stack introduced below (see Section \ref{section:campana_stacks} and also Section \ref{section:remarks_on_orbi_mordell}).
\end{remark}

\subsection{Elliptic fibrations}\label{section:elliptic_fibrations}

Every smooth projective weakly special surface is special.
This stems from the fact that every inf-multiple fibre of an elliptic fibration (or more generally abelian fibration) is divisible \cite[Remarque~1.3]{Ca05}.
Indeed, the Kodaira-N\'eron classification of singular fibres of an elliptic fibration shows that the only inf-multiple fibres are those of type $\tensor[_m]{\mathrm{I}}{_n}$.
   
Nevertheless, other examples of weakly special nonspecial surfaces can be obtained from applying Theorem \ref{thm:family_of_ws_vars} to Jacobian\footnote{Recall that an elliptic fibration $X\to S$ is said to be \emph{Jacobian} if it has a section. If $X$ and $S$ are regular, a Jacobian elliptic fibration $X\to S$ has no inf-multiple fibres in codimension one.} elliptic fibrations over $\mathbb{P}^1$ with $\mathrm{II}^\ast$, $\mathrm{III}^\ast$ and $\mathrm{IV}^\ast$ fibres, by removing reduced components of the singular fibres.
This construction of inf-multiple but non-divisible fibres on a family of curves was the starting observation of this work.

\begin{corollary} \label{cor:elliptic_fibrations}
Let $X$ be a smooth quasi-projective surface over $k$, let $C$ be a smooth quasi-projective special curve and let $X\to C$ be a proper surjective morphism whose general fibre is a genus one curve.
Suppose that $X \to C$ has no multiple fibres and has a singular fibre of type $\mathrm{II}^\ast, \mathrm{III}^\ast$, or $\mathrm{IV}^\ast$.
Let $F$ denote the union of all reduced components of that singular fibre and let $U = X \setminus F$.
Then the following statements hold. 
\begin{enumerate}
\item The smooth quasi-projective surface $U$ is weakly special.
\item If $C$ is $\mathbb{A}^1\setminus \{0\}$ or projective of genus one, the surface $U$ is not special.
\end{enumerate} 
\end{corollary}
\begin{proof}
Removing the reduced components of a fibre of type II$^\ast$, III$^\ast$ or IV$^\ast$ results in a nonreduced, but not divisible, fibre.
Since both $C$ and the general fibres of $U \to C$ are weakly special, the first statement thus follows by applying Theorem~\ref{thm:family_of_ws_vars} to the morphism $U \to C$. 

On the other hand, removing the reduced components of a fibre of type II$^\ast$, III$^\ast$ or IV$^\ast$ results in an inf-multiple fibre of inf-multiplicity two.
Thus, $U \to C$ has a non-trivial orbifold base divisor $\Delta$, supported on the image of $F$.
If $C$ is $\mathbb{A}^1\setminus \{0\}$ or projective of genus one, the pair $(C, \Delta)$ is of general type.
So in that case $U \to C$ is a morphism of general type and $U$ is not special.
\end{proof}

\subsection{The Ellenberg--Mazur threefold}

Let $r \geq 2$ be a positive integer.
We define $V_r$ to be the hypersurface
\[ X_0^rY_0^{r+1}+X_1^rY_1^{r+1}+X_2^rY_2^{r+1}=0 \]
in $\mathbb{P}^2 \times \mathbb{P}^2$, and refer to it as the \emph{$r$-th Ellenberg--Mazur threefold}.
Let $V_r^o \subseteq V_r$ denote its smooth locus.

\begin{proposition}
For every $r \geq 2$, the variety $V_r^o$ is weakly special.
If $r \geq 4$, then $V_r^o$ is not special.
\end{proposition}
\begin{proof}
Consider the rational map $V_r \ratmap \mathbb{P}^1$ defined by $(X,Y)\mapsto [X_0^rY_0^{r+1}:X_1^rY_1^{r+1}]$.
Its indeterminacy locus is contained in the singular locus of $V_r$, and hence it defines a morphism $\pi \colon V_r^o \to \mathbb{P}^1$.
The morphism $\pi$ is smooth over $\mathbb{P}^1\setminus \{0,-1,\infty\}$, and each smooth fibre is dominated by $\mathbb{G}_m^2$.
Moreover, its fibre over a point in $\{0,-1,\infty\}$ has precisely two irreducible components, one of multiplicity $r$ and one of multiplicity $r+1$.
In particular, there are no divisible fibres in codimension one.
As the general fibre is dominated by $\mathbb{G}_m^2$, hence is weakly special, and the base $\mathbb{P}^1$ is weakly special, it follows from Theorem~\ref{thm:family_of_ws_vars} that $V_r^o$ is weakly special.

On the other hand, the inf-multiplicity over each of the three points $0,-1,\infty$ is $r$.
Hence the orbifold base  divisor of $\pi$ is $\left(1-\frac1r\right)([0]+[-1]+[\infty])$.
In particular, for $r\geq 4$, the morphism $\pi$ is flat and has an orbifold base of general type, so that $V_r^o$ is not special.
\end{proof}

\section{Campana stacks}\label{section:campana_stacks} 

In this section, we construct the Campana stack associated to a C-pair $(X,\Delta)$ in several steps.
We begin by recalling the construction of the Campana stack $\mathcal{C}_{X,D,a}$ attached to an effective Cartier divisor $D$ on $X$ and an ordered tuple of positive integers $a=(a_1,\ldots,a_r)$ (see Definition \ref{def:campana_stack_intro}).
Roughly speaking, the tuple $a$ records the multiplicities of the irreducible components of a fibre of the form $\sum_{i=1}^r a_i F_i$.
We then pass from tuples to subsets of $\mathbb{Z}_{\geq 1}$ closed under addition such as $\mathbb{Z}_{\geq m}$.
Viewing $\mathbb{Z}_{\geq m}$ as the semigroup generated by $m,m+1,\ldots,2m-1$, we define the Campana stack attached to a pair $(X,(1-\frac{1}{m})D)$.
Finally, given a C-pair $(X,\Delta)$, by combining these constructions for the components of $\Delta$, we obtain the Campana stack $\mathcal{C}_{X,\Delta}$; see Definition \ref{def:campana_stack_2}.

\subsection{Universal inf-multiple fibre}
 
Let $r\geq 1$ and let $a=(a_1,\ldots, a_r)$ be an $r$-tuple of positive integers.
As in Section \ref{section:intro_campana}, we define $p_a \colon \mathbb{A}^r \to \mathbb{A}^1$ to be the morphism defined by 
\[ p_a(x_1,\ldots,x_r) = x_1^{a_1} \cdot \ldots \cdot x_r^{a_r}. \]
Observe that the fibre of $p_a$ over $0$ is connected and decomposes as $\sum_{i=1}^r a_i F_i$, where $F_i = Z(x_i)$ is the $i$-th coordinate hyperplane in $\mathbb{A}^r$.
Consider the morphism of algebraic stacks 
\[ P_a \colon [\mathbb{A}^r/\mathbb{G}_m^r] = [\mathbb{A}^1/\mathbb{G}_m]^r \to [\mathbb{A}^1/\mathbb{G}_m] \] 
given by $((L_1,s_1),\ldots, (L_r,s_r))\mapsto (L_1^{\otimes a_1}\otimes \ldots \otimes L_r^{\otimes a_r}, s_1^{\otimes a_1}\otimes \ldots \otimes s_r^{\otimes a_r})$; we view $P_a$ as  the ``universal inf-multiple fibre with multiplicities $a_1, \ldots, a_r$''.
The composed morphism $\mathbb{A}^r\to [\mathbb{A}^r/\mathbb{G}_m^r] = [\mathbb{A}^1/\mathbb{G}_m]^r \to [\mathbb{A}^1/\mathbb{G}_m]$ factors as $q\circ p_a$, where $q\colon \mathbb{A}^1\to [\mathbb{A}^1/\mathbb{G}_m]$ is the quotient morphism.
We now recall the definition of the Campana stack $\mathcal{C}_{X,D,a}$ (see Definition \ref{def:campana_stack_intro}); note that we allow for $X$ to be an algebraic stack here.

\begin{definition}\label{def:campana_stack}
Let $X$ be an algebraic stack, let $D$ be a generalized effective Cartier divisor on $X$, and let $f_D\colon X\to [\mathbb{A}^1/\mathbb{G}_m]$ be the associated moduli map.
Then the \emph{Campana stack $\mathcal{C}_{X,D,a}$ of $(X,D,a)$} is defined by the following Cartesian diagram.
\[
\xymatrix{ \mathcal{C}_{X,D,a}  \ar[rr] \ar[d] & & [\mathbb{A}^1/\mathbb{G}_m]^r \ar[d]^{P_a} \\ X \ar[rr]_{f_D} & & [\mathbb{A}^1/\mathbb{G}_m]}
\]
The \emph{Campana space $U_{X,D,a}$ of the triple $(X,D,a)$} is the algebraic stack defined by the following Cartesian diagram.
\[
\xymatrix{ U_{X,D,a} \ar[rr] \ar[d]_{\mathbb{G}_m^r-\text{torsor}} & &  \mathbb{A}^r \ar[d]^{\mathbb{G}_m^r-\text{torsor}} \\ 
\mathcal{C}_{X,D,a}  \ar[rr]   & & [\mathbb{A}^1/\mathbb{G}_m]^r  }
\]
\end{definition}

\begin{remark} \label{remark:representability}
The morphism $\mathbb{A}^r \to [\mathbb{A}^r/\mathbb{G}_m^r] \to [\mathbb{A}^1/\mathbb{G}_m]$ is affine.
Hence $U_{X,D,a} \to X$ is affine as well, so that in particular, it is a relative algebraic space\footnote{In other words, $U_{X,D,a} \to X$ is representable by algebraic spaces.} (this justifies the terminology ``Campana space'').
Consequently, if $X$ is an algebraic space (resp.\ a scheme), then $U_{X,D,a}$ is an algebraic space (resp.\ a scheme) and if $X$ is affine (resp.\ quasi-projective) over a scheme $S$, then $U_{X,D,a}$ is affine (resp.\ quasi-projective) over $S$.
\end{remark}

\begin{remark} \label{remark:reduce_to_root_stack}
If $r=1$ and $a_1 = n$, then $\mathcal{C}_{X,D,a}$ is the $n$-th root stack $\sqrt[n]{D/X}$.
For any $r \geq 1$, if $n= \gcd(a_1,\ldots,a_r)$, then $\mathcal{C}_{X,D,a} \to X$ factors over $\mathcal{X}_n=\sqrt[n]{D/X}$.
In fact, $\mathcal{C}_{X,D,a}$ is the Campana stack $\mathcal{C}_{\mathcal{X}_n,\mathcal{D},b}$, where $\mathcal{D}$ denotes the universal root divisor on $\mathcal{X}_n$ (so that $n\mathcal{D}$ is the pullback of $D$ to $\mathcal{X}_n$) and $b = (b_1,\ldots,b_r)$ with $b_i = \frac{a_i}{n}$. 
\end{remark}

\begin{lemma} \label{lemma:c-stack-finitetype}
Let $X$ be an algebraic stack and let $D$ be an effective Cartier divisor on $X$.
Then the morphism $\mathcal{C}_{X,D,a}\to X$ is of finite type and an isomorphism over $X\setminus \supp(D)$.  
\end{lemma}
\begin{proof}
By definition, the morphism $\mathcal{C}_{X,D,a}\to X$ is the base change of $P_a \colon [\mathbb{A}^r/\mathbb{G}_m^r]\to [\mathbb{A}^1/\mathbb{G}_m]$ along the morphism $X \to [\mathbb{A}^1/\mathbb{G}_m]$ associated to $D$.
So since $P_a$ is of finite type,  the morphism $\mathcal{C}_{X,D,a}\to X$ is of finite type as well.
Moreover, $P_a$ is an isomorphism over the open substack $[(\mathbb{A}^1\setminus \{0\})/\mathbb{G}_m]$ of $[\mathbb{A}^1/\mathbb{G}_m]$, so $\mathcal{C}_{X,D,a}\to X$ is an isomorphism over the preimage of $[(\mathbb{A}^1\setminus \{0\})/\mathbb{G}_m]$ under $X \to [\mathbb{A}^1/\mathbb{G}_m]$, which is exactly $X \setminus \supp(D)$.
\end{proof}

\begin{lemma}\label{lemma:structure_of_Campana_stacks}
Let $X$ be a scheme and let $D$ be an effective Cartier divisor on $X$.
The stack $\mathcal{C}_{X,D,a}$ is an algebraic stack whose inertia groups are subgroups of $\mathbb{G}_m^r$.
If $X$ has affine diagonal, then $\mathcal{C}_{X,D,a}$ has affine diagonal.
\end{lemma}
\begin{proof}
Since any fibre product of algebraic stacks is an algebraic stack, we see that $\mathcal{C}_{X,D,a}$ is an algebraic stack.
Since $X\to [\mathbb{A}^1/\mathbb{G}_m]$ is representable (as $X$ is a scheme), the morphism $\mathcal{C}_{X,D,a}\to [\mathbb{A}^1/\mathbb{G}_m]^r$ is representable. 
Since the morphism $\mathcal{C}_{X,D,a}\to [\mathbb{A}^1/\mathbb{G}_m]^r$ is representable, the inertia groups of $\mathcal{C}_{X,D,a}$ embed into those of $[\mathbb{A}^1/\mathbb{G}_m]^r$, hence are subgroups of $\mathbb{G}_m^r$.  

Since the stack $[\mathbb{A}^r/\mathbb{G}_m]^r$ has affine diagonal, the morphism $P_a\colon [\mathbb{A}^1/\mathbb{G}_m]^r\to [\mathbb{A}^1/\mathbb{G}_m]$ has affine diagonal.
Therefore, the morphism $\mathcal{C}_{X,D,a}\to X$ has affine diagonal.
In particular, if $X$ has affine diagonal, then $\mathcal{C}_{X,D,a}$ has affine diagonal as well. 
\end{proof}

\begin{lemma}\label{lemma:campana_stack_of_a_1}
Let $X$ be a scheme and let $D$ be an effective Cartier divisor on $X$.
Let $a = (a_1,\ldots,a_r)$ be an $r$-tuple of positive integers with $\gcd(a_1,\ldots,a_r)=1$.
Let $T$ be a reduced scheme and let $f \colon T \to X$ be a morphism satisfying $f(T) \subseteq \supp(D)$.
Then $f$ factors over $\mathcal{C}_{X,D,a}$.
\end{lemma}
\begin{proof}
A morphism $T\to [\mathbb{A}^1/\mathbb{G}_m]$ corresponds to a generalized effective Cartier divisor $(L,\rho)$ on $T$.
Likewise, a morphism $T\to [\mathbb{A}^1/\mathbb{G}_m]^r$ amounts to giving $r$ such divisors $(L_1,\rho_1),\ldots,(L_r,\rho_r)$ and the morphism $T \to [\mathbb{A}^1/\mathbb{G}_m]$ associated to $(L,\rho)$ factors through
\[ P_a \colon [\mathbb{A}^1/\mathbb{G}_m]^r \to [\mathbb{A}^1/\mathbb{G}_m] \]
if and only if
\[ (L,\rho) \cong \big(L_1^{\otimes a_1}\otimes \cdots \otimes L_r^{\otimes a_r},\;
\rho_1^{\otimes a_1}\otimes \cdots \otimes \rho_r^{\otimes a_r}\big) \]
for some $(L_1, \rho_1), \ldots, (L_r, \rho_r)$ as above.

Now, since our morphism $f$ satisfies $f(T) \subseteq \supp(D)$ and $T$ is reduced, we find that the composition $T \to X \to [\mathbb{A}^1/\mathbb{G}_m]$ corresponds to the generalized effective Cartier divisor $(L,0)$ with $L$ being the line bundle $f^* \mathcal{O}_X(D)$.
Since $\gcd(a_1,\ldots, a_r)=1$, we may write $1 = \lambda_1 a_1 + \ldots +\lambda_r a_r$ for some integers $\lambda_i\in \mathbb{Z}$.
Hence we can take $L_i = L^{\lambda_i}$ and $\rho_i = 0$ in the above and find that the composition $T \to X \to [\mathbb{A}^1/\mathbb{G}_m]$ factors through $P_a$.
By definition of $\mathcal{C}_{X,D,a}$, this is equivalent to the claim.
\end{proof}

\begin{lemma}\label{lemma:campana_stack_of_a_2}
Let $X$ be a scheme and let $D$ be an effective Cartier divisor on $X$.
Let $a = (a_1,\ldots,a_r)$ be an $r$-tuple of positive integers.
Let $T$ be an integral scheme and let $f \colon T \to X$ be a morphism whose image is not contained in the support of $D$.
Then $f$ factors over $\mathcal{C}_{X,D,a}$ if and only if there exist effective Cartier divisors $D_1,\ldots, D_r$ such that $f^*D = \sum_{i=1}^r a_i D_i$.
\end{lemma}
\begin{proof}
Since $f\colon T\to X$ does not factor over $\supp(D)$, the composition $T \to X \to [\mathbb{A}^1/\mathbb{G}_m]$ corresponds to the (non-generalized!) effective Cartier divisor $f^*D$.
By definition of $\mathcal{C}_{X,D,a}$, the morphism $f$ factors through $\mathcal{C}_{X,D,a}$ if and only if $f^*D$, viewed as a $T$-point of $[\mathbb{A}^1/\mathbb{G}_m]$, lies in the image of $P_a$, that is, if and only if there exist generalized effective Cartier divisors $D_1,\ldots,D_r$ on $T$ such that
\[ f^*D = \sum_{i=1}^r a_i D_i. \]
Since $f^*D$ is a non-generalized effective Cartier divisor, each $D_i$ appearing in such a decomposition must necessarily be a non-generalized effective Cartier divisor.
\end{proof}


\begin{definition}
Let $X$ be an algebraic stack, let $n\geq 1$ be an integer, let $D_1,\ldots, D_n$ be effective Cartier divisors on $X$, let $r_1,\ldots, r_n$ be positive integers, and let $a^i$ be an $r_i$-tuple of positive integers.
We define the \emph{Campana stack of the tuple $(X,(D_i,a^i)_{i=1}^n)$} to be  
\[ \mathcal{C}_{X,(D_i,a^i)_{i=1}^n} = \mathcal{C}_{X,D_1,a^1}\times_X \ldots \times_X \mathcal{C}_{X,D_n,a^n}. \]
We define the \emph{Campana space of the tuple $(X,(D_i,a^i)_{i=1}^n)$} to be 
\[ U_{X,(D_i,a^i)_{i=1}^n} = \ U_{X,D_1,a^1}\times_X \ldots \times_X U_{X,D_n,a^n}. \]
\end{definition}

\subsection{Semigroups attached to a divisor}  

As we will see now, the Campana stack of a C-pair $(X,\Delta)$ is a specific instance of the Campana stacks constructed above.     

In this paper, a \emph{semigroup} will be a (possibly empty) subset of $\mathbb{Z}_{\geq 1}$ closed under addition.   
Given a semigroup $M$, we recall that $M$ has a unique minimal set of generators, whose elements are called the \emph{atoms} of $M$.
By passing from a semigroup to its minimal set of generators, we are led to the Campana stack of a semigroup attached to a divisor.

\begin{definition}
Let $X$ be an algebraic stack and let $D$ be an effective Cartier divisor on $X$.
Let $M$ be a semigroup.
Let $\{a_1,\ldots,a_r\}$ be the (possibly empty) minimal set of positive integers which generate $M$ as a semigroup.
If $r>0$, define $a_M = (a_1,\ldots,a_r)$ and set
\[ \mathcal{C}_{X,D,M} := \mathcal{C}_{X,D,a_M}, \quad U_{X,D,M} = U_{X,D,a_M}. \]
If $r=0$ (i.e., $M=\emptyset$), define $\mathcal{C}_{X,D,M} = X\setminus \supp(D)$ and $U_{X,D,M} = X\setminus \supp(D)$.
(The chosen ordering of the set of atoms of $M$ does not change the stack.)
More generally, let $D_1,\ldots,D_n$ be effective Cartier divisors on $X$.
For each $i$, let $M_i$ be a semigroup.
We define  
\[ \mathcal{C}_{X,(D_i,M_i)_i} = \mathcal{C}_{X,D_1,M_1} \times_X \cdots \times_X \mathcal{C}_{X,D_n,M_n}, \quad
U_{X,(D_i,M_i)_i} = U_{X,D_1,M_1} \times_X \cdots \times_X U_{X,D_n,M_n}. \]
\end{definition}

\begin{proposition}\label{prop:C_pair_vs_Campana_stack}
Let $X$ be an integral noetherian scheme and let $D$ be an effective Cartier divisor on $X$.
Let $M$ be a semigroup.
Let $T$ be an integral regular noetherian scheme and let $f\colon T\to X$ be a morphism whose image is not contained in the support of $D$.
Then $f$ factors through $\mathcal{C}_{X,D,M}$ if and only if 
\[f^* D = \sum_{j=1}^\ell \alpha_j E_j,\]
where the $E_j$ are the irreducible components of $\supp f^*D$ and $\alpha_j \in M$ for every $j$.
\end{proposition}
\begin{proof}
Write $a = (a_1, \ldots, a_r)$ for the tuple of atoms of $M$.
By Lemma \ref{lemma:campana_stack_of_a_2}, the morphism $f$ factors through $\mathcal{C}_{X,D,a} = \mathcal{C}_{X,D,M}$ if and only if there exist effective Cartier divisors $D_1,\ldots,D_r$ on $T$ such that $f^*D = \sum_{i=1}^r a_i D_i$.
If such $D_i$ exist, it is clear that the multiplicity of every prime divisor appearing in the Weil divisor $f^*D$ lies in $M$.
Conversely, if for every irreducible component $E$ of $\supp f^*D$, the coefficient of $E$ in $f^*D$ lies in $M$, then we write
\[ f^*D = \sum_{j=1}^\ell m_j E_j \]
with $E_1,\ldots,E_\ell$ the irreducible components of $f^{-1}(D)$ and $m_j\in M$.
Expressing each $m_j$ as
\[ m_j = \sum_{i=1}^r n_{j,i} a_i \]
with $n_{j,i} \geq 0$, it follows that
\[ f^*D = \sum_{i=1}^r a_i \left(\sum_{j=1}^\ell n_{j,i} E_j\right), \]
which exhibits $f^*D$ as $\sum a_i D_i$ for suitable effective Cartier divisors $D_i$.
\end{proof}

\subsection{Campana stacks of C-pairs}
 
For every integer $m\geq 1$, the subset $\mathbb{Z}_{\geq m}$ is a semigroup whose set of atoms is $\{m, m+1,\ldots, 2m-1\}$.

\begin{definition} \label{def:campana_stack_2}
Let $X$ be a scheme, let $D_1,\ldots, D_n$ be effective Cartier divisors and let $m_i\in \mathbb{Z}_{\geq 1} \cup\{\infty\}$ for $i=1,\ldots, n$.
Write $\Delta  = \sum_{i=1}^n \left(1-\frac{1}{m_i}\right) D_i$.
Define $M_i := \mathbb{Z}_{\geq m_i}$, with the convention that $\mathbb{Z}_{\geq \infty} = \emptyset$. 
We define the \emph{Campana stack $\mathcal{C}_{X,\Delta}$ of the C-pair $(X,\Delta)$} to be $\mathcal{C}_{X,\Delta} = \mathcal{C}_{X,(D_i,M_i)_{i=1}^n}$ and we define the \emph{Campana space $U_{X,\Delta}$} of $(X,\Delta)$ to be $U_{X,\Delta} = U_{X,(D_i,M_i)_{i=1}^n}$.
\end{definition}

As a straightforward consequence of Proposition \ref{prop:C_pair_vs_Campana_stack} we see that the Campana stack of $(X,\Delta)$ parametrizes maps to the C-pair (first stated as Theorem \ref{thm:campana_stack_intro} in Section \ref{section:c-pair}), under suitable regularity assumptions. 

\begin{theorem}\label{thm:C-pair-Campana-stack}
Let $(X,\Delta)$ be a C-pair and let $T$ be an integral regular noetherian scheme.
A morphism $f\colon T \to X$ defines a morphism of C-pairs $T\to (X,\Delta)$ if and only if $T\to X$ factors over $\mathcal{C}_{X,\Delta}$. 
\end{theorem}
\begin{proof}
Write $\Delta = (1-\frac{1}{m_i})D_i$ and define $M_i$ as in Definition \ref{def:campana_stack_2}.
By Lemma~\ref{lemma:campana_stack_of_a_1}, we may assume that the image of $T\to X$ is not contained in the support of $\Delta$.
Now, by Proposition~\ref{prop:C_pair_vs_Campana_stack}, the morphism $f\colon T\to X$ is a morphism of C-pairs $T\to (X,\Delta)$ if and only if, for every $i$, the morphism $f\colon T\to X$ factors over $\mathcal{C}_{X,D_i,M_i}$ for every $i$.
By definition of $\mathcal{C}_{X,\Delta}$, the latter is equivalent to $f\colon T\to X$ factoring over $\mathcal{C}_{X,\Delta}$.
\end{proof}

\begin{remark} \label{remark:sheaf_vs_stack} 
Consider the C-pair $(X,\Delta)$ with $X=\mathbb{A}^1$ and $\Delta = \tfrac{1}{2}[0]$ over $\mathbb{C}$.
Then any subfunctor of $h_X = \mathrm{Hom}(-,X)$ which is defined by 
\[ T \mapsto \{\text{Campana maps } T\to (X,\Delta)\} \]
on smooth schemes over $\mathbb{C}$ does not satisfy fppf descent and therefore is not representable by a scheme.
Indeed, the identity morphism $\mathbb{A}^1\to\mathbb{A}^1$ is not a morphism to $(\mathbb{A}^1,\Delta)$, since the pullback of $[0]$ has multiplicity one at the origin.
However, after pulling back along the fppf cover $\mathbb{A}^1\to\mathbb{A}^1$ defined by $z\mapsto z^2$, the composite becomes a morphism of C-pairs.
Thus the Campana condition does not hold for the identity morphism itself, but does hold after an fppf covering, which implies that the corresponding functor fails to be an fppf sheaf.

The stack $\mathcal{C}_{\mathbb{A}^1, \Delta}$ introduced above resolves this issue.
A morphism to the Campana stack records not only the underlying morphism to $X$, but also additional data encoding a decomposition of the pullback divisor.
In the above example, we can spell this out explicitly.
Namely, with $T' = T = \mathbb{A}^1$ and with the fppf covering $T'\to T$ defined by $z\mapsto z^2$, we consider the composed morphism $T'\to T\to X=\mathbb{A}^1$ and the fibre product $T'\times_T T' = \Spec k[x,y]/(x^2-y^2)$.
We now see that the two pullbacks of $[0] = \{z=0\}$ to the fibre product are $\{x^2=0\}$ and $\{y^2=0\}$, respectively.
These are the same Cartier divisors and they can both be written as two times some effective Cartier divisor.
A map to the Campana stack however also asks one to fix such a decomposition as $2D$ with $D$ a Cartier divisor on the source.
For the morphism $T\to \mathbb{A}^1$, this ends up being $D=\{z=0\}$.
However, the two pullbacks of this $D$ to the fibre product are $\{x=0\}$ and $\{y=0\}$, respectively.
These are not the same Cartier divisor, and thus there is no descent datum.
\end{remark}

\subsection{Regularity of Campana stacks} \label{subsect:regularity_campana_stacks}

Recall that if $n\geq 1$ and $D$ is a regular effective Cartier divisor on a regular scheme $X$ of characteristic zero, then the root stack $\sqrt[n]{D/X}$ is regular.
The analogous statement holds for Campana stacks and is a consequence of the local description of Campana stacks and the following lemma.

\begin{lemma}\label{lemma:comm_alg}
Let $(R,\mathfrak{m})$ be an equicharacteristic regular local noetherian ring and let $f\in \mathfrak{m} \setminus \mathfrak{m}^2$.
Fix positive integers $a_1,\dots,a_r$.
Then
\[ A := R[x_1,\dots,x_r]/(x_1^{a_1}\cdots x_r^{a_r}-f) \]
is regular.
\end{lemma}
\begin{proof}
We may assume that $R$ is complete \cite[Tag 00OF]{stacks-project}, and thus that $R$ is a power series ring $R \cong K \powerseries{y_1,\dots,y_d}$ for some field $K$.  
Let $g := x_1^{a_1}\cdots x_r^{a_r}-f \in R[x_1,\dots,x_r]$.
As $f\in \mathfrak m\setminus \mathfrak m^2$, the power series $f$ contains a nonzero linear term.
After reordering the variables if necessary, we may assume that the coefficient of $y_1$ in $f$ is nonzero.
In particular, $\frac{\partial f}{\partial y_1}$ is a unit in $R$.
Consider the $K$-derivation $\frac{\partial}{\partial y_1}\colon R \to R$, which extends uniquely to a derivation
\[ \frac{\partial}{\partial y_1} \colon R[x_1,\dots,x_r] \to R[x_1,\dots,x_r] \]
by letting it act trivially on the variables $x_i$.
We then compute $\frac{\partial g}{\partial y_1} = -\frac{\partial f}{\partial y_1}$. 
Since $\frac{\partial f}{\partial y_1}$ is a unit in $R$, it follows that $\frac{\partial g}{\partial y_1}$ is a unit in $R[x_1,\dots,x_r]$, and hence also in $A$.
By the Jacobian criterion \cite[Tag 07PF]{stacks-project}, this implies that $A$ is regular.
\end{proof}

\begin{definition}\label{def:K}
Let $a = (a_1,\ldots,a_r)$ be an $r$-tuple of positive integers.
Let $\mu_a \colon \mathbb{Z}^r \to \mathbb{Z}$ be the homomorphism given by $(v_1,\ldots,v_r)\mapsto \sum_{i=1}^r a_i v_i$.
Let $K_a$ be the kernel of $\mu_a$.  
\end{definition}
 
If $\gcd(a_1,\ldots,a_r)=1$, then $\mu_a$ is surjective.   
In this case, we  have a  short exact sequence 
\[ 0 \to K_a \to \mathbb{Z}^r \to \mathbb{Z}\to 0 \]
and we choose a splitting $\sigma_a\colon \mathbb{Z}\to \mathbb{Z}^r$.
Such a splitting induces a splitting of the torus $\mathbb{G}_m^r = (\mathbb{G}_m \otimes K_a) \times \mathbb{G}_m$.
In particular, the action of $\mathbb{G}_m \otimes K_a$ on $\mathbb{A}^r$ induces a morphism of stacks $[\mathbb{A}^r/(\mathbb{G}_m\otimes K_a)] \to [\mathbb{A}^r/\mathbb{G}_m^r]$ and the following Cartesian diagram.
\[
 \xymatrix{  [ \mathbb{A}^r/( \mathbb{G}_m\otimes  K_a)] \ar[rr] \ar[d]_{(x_1,\ldots,x_r)\longmapsto x_1^{a_1}\cdot \ldots \cdot x_r^{a_r}}  & & [\mathbb{A}^r/\mathbb{G}_m^r]  = [\mathbb{A}^r/(\mathbb{G}_m\times (\mathbb{G}_m\otimes K_a))]\ar[d]^{P_a}    \\
 \mathbb{A}^1 \ar[rr]  & & [\mathbb{A}^1/\mathbb{G}_m]  }
\]

Recall that an algebraic stack $X$ is \emph{regular} if there is a smooth surjective morphism $P\to X$ with $P$ a regular locally noetherian scheme.  By \cite[Tag~06QN]{stacks-project}, being regular is local in the fppf topology.

\begin{theorem} \label{thm:regularity} 
Let $X$ be a regular noetherian algebraic stack over $\mathbb{Q}$ and let $D$ be a regular effective Cartier divisor on $X$.  
Suppose that $\gcd(a_1,\ldots,a_r) = 1$. 
Then $\mathcal{C}_{X,D,a}$ is regular.
\end{theorem}
\begin{proof}
Write $K = K_a$.
Since the statement is local on $X$, we may assume that $X = \Spec R$ is affine and that $R$ is a regular local ring with maximal ideal $\mathfrak{m}$.  
Moreover, we may assume that $D$ corresponds to an element $f \in R$.
If $f\not \in \mathfrak{m}$, then $D$ is trivial and $\mathcal{C}_{X,D,a}=X$.
Thus, we may assume that $f\in \mathfrak{m}$.
In that case, we have that $f\in \mathfrak{m}\setminus \mathfrak{m}^2$, as $D = \Spec(R/f)$ is regular.
The moduli map $X \to [\mathbb{A}^1/\mathbb{G}_m]$ associated to $D$ factors over the morphism $X\to \mathbb{A}^1$ given by $f$.
We obtain the following diagram in which every square is Cartesian.  
\[
\xymatrix{V \ar[r] \ar[d]  &  \mathbb{A}^r  \ar[d] & & \\ 
\mathcal{C}_{X,D,a} \ar[r] \ar[d]  &  [\mathbb{A}^r/(\mathbb{G}_m\otimes K)] \ar[d] \ar[r] &  [\mathbb{A}^r/\mathbb{G}_m^r] \ar[d] \\ \Spec R  = X \ar[r]^{f}  &  \mathbb{A}^1 \ar[r]  & [\mathbb{A}^1/\mathbb{G}_m]}
\] 
In this diagram, the vertical map $\mathbb{A}^r\to \mathbb{A}^1$ is the morphism $p_a\colon (x_1,\ldots,x_r) \mapsto x_1^{a_1}\cdot \ldots \cdot x_r^{a_r}$.
Thus $V$ is isomorphic to $\Spec R[x_1,\ldots,x_r]/(x_1^{a_1}\cdot \ldots \cdot x_r^{a_r} - f)$, and is thus regular since $f\in \mathfrak{m}\setminus \mathfrak{m}^2$ (see Lemma \ref{lemma:comm_alg}).
Since $V\to \mathcal{C}_{X,D,a}$ is a $\mathbb{G}_m^{r-1}$-torsor (as $\mathbb{G}_m\otimes K\cong \mathbb{G}_m^{r-1}$), we see that $\mathcal{C}_{X,D,a}$ is regular, as required.
\end{proof}

\begin{corollary}\label{cor:regularity}
Let $X$ be a regular noetherian algebraic stack over $\mathbb{Q}$ and let $D$ be a regular divisor on $X$.   
Then $\mathcal{C}_{X,D,a}$ is regular and the Campana space $U_{X,D,a}$ is regular.  
\end{corollary}
\begin{proof}
Let $n=\gcd(a_1,\ldots,a_r)$ and define $b=(\frac{a_1}{n},\ldots, \frac{a_r}{n})$.
Let $\mathcal{X}_n = \sqrt[n]{D/X}$ be the $n$-th root stack of $X$ along $D$.
Note that $\mathcal{X}_n$ is regular.
(This follows for example from the description in \cite[10.3.10]{OlssonBook} as shown in \cite[Proposition~8.3]{Moerman}.)
Let $D_n$ denote the universal root divisor on $\mathcal{X}_n$ (so that $nD_n$ is the pullback of $D$ to $\mathcal{X}_n$).
As $D_n$ is a $\mu_n$-gerbe over $D$ and $n$ is invertible on $X$, $D_n$ is regular (as a stack).
In particular, the Campana stack $\mathcal{C}_{\mathcal{X}_n,D_n,b}$ is regular by Theorem \ref{thm:regularity}.
Since $\mathcal{C}_{X,D,a} = \mathcal{C}_{\mathcal{X}_n, D_n,b}$, we conclude that $\mathcal{C}_{X,D,a}$ is regular.
Finally, since the Campana space $U_{X,D,a}$ is a $\mathbb{G}_m^r$-torsor over $\mathcal{C}_{X,D,a}$ and the latter is regular, we conclude that $U_{X,D,a}$ is regular.
\end{proof}

\subsection{The morphism from the Campana space to $X$}
Let $X$ be an integral regular noetherian scheme over $\mathbb{Q}$ and let $D$ be a prime divisor on $X$.
Let $a = (a_1,\ldots,a_r)$ be an $r$-tuple of positive integers and write $U = U_{X,D,a}$ for the associated Campana space.
Recall that the morphisms $p_a \colon \mathbb{A}^r\to \mathbb{A}^1$ and $P_a \colon [\mathbb{A}^r/\mathbb{G}_m^r] \to [\mathbb{A}^1/\mathbb{G}_m]$ are both defined by suitably interpreting the formula $(x_1,\ldots,x_r) \mapsto x_1^{a_1}\cdots x_r^{a_r}$.
In particular, the composite morphism
\[\xymatrix{ q_a \colon \mathbb{A}^r \ar[r]^{\mathrm{quot}} & [\mathbb{A}^r/\mathbb{G}_m^r] \ar[r]^{P_a} & [\mathbb{A}^1/\mathbb{G}_m]} \]
factors as $q\circ p_a$, where $q \colon \mathbb{A}^1\to[\mathbb{A}^1/\mathbb{G}_m]$ is the quotient map. 

We want to understand the scheme-theoretic fibre of $U \to X$ over the generic point of $D$.
To do so, define $V = \mathbb{A}^r \times_{q_a,[\mathbb{A}^1/\mathbb{G}_m],q} \mathbb{A}^1$.
Since $q_a = q \circ p_a$, we obtain
\[ V\cong \mathbb{A}^r\times_{p_a,\mathbb{A}^1} (\mathbb{A}^1\times_{[\mathbb{A}^1/\mathbb{G}_m]}\mathbb{A}^1). \]
As $q\colon\mathbb{A}^1\to[\mathbb{A}^1/\mathbb{G}_m]$ is a $\mathbb{G}_m$-torsor, we see that $\mathbb{A}^1\times_{[\mathbb{A}^1/\mathbb{G}_m]}\mathbb{A}^1\cong\mathbb{A}^1\times\mathbb{G}_m$, and hence $V \cong \mathbb{A}^r\times\mathbb{G}_m$.

Let $F = \sum_{i=1}^r a_iF_i$ be the fibre of $p_a$ over $0$, where $F_i=Z(x_i)$.
Then $F$ is a connected divisor with $r$ irreducible components of multiplicities $a_1,\ldots,a_r$.
Let
\[ X' = X\times_{[\mathbb{A}^1/\mathbb{G}_m]}\mathbb{A}^1 \]
and define $V'$ similarly so that the following diagram is commutative and all squares are Cartesian.
\begin{center}
\begin{tikzcd}
V' \arrow[dd] \arrow[rr] \arrow[rd] && V \arrow[rd] \arrow[dd] & \\
& {U_{X,D,a}} \arrow[rr] \arrow[dd] && \mathbb{A}^r \arrow[d] 
\arrow[dd, bend left =60, "q_a"] \\
X' \arrow[rd] \arrow[rr] && \mathbb{A}^1 \arrow[rd, "q" '] & {[\mathbb{A}^r/\mathbb{G}_m^r]} \arrow[d, "P_a"] \\
& X \arrow[rr, "f_D" '] && {[\mathbb{A}^1/\mathbb{G}_m]} 
\end{tikzcd}
\end{center}

Under the identification $V \cong \mathbb{A}^r\times\mathbb{G}_m$,
the morphism $V\to\mathbb{A}^1$ is given by
$(x,\lambda)\mapsto \lambda p_a(x)$.
Thus $V_0=F\times\mathbb{G}_m$ scheme-theoretically, and its
irreducible components $F_i\times\mathbb{G}_m$ have multiplicities $a_i$.
By the Cartesian diagram above, the morphism $V' \to X'$ is obtained from $V\to\mathbb{A}^1$ by base change along $X'\to\mathbb{A}^1$, and $V'\to X'$ is the pullback of $U\to X$ along $X'\to X$.
Consequently, the scheme-theoretic fibre of $U \to X$ over $D$ decomposes as $\sum_{i=1}^r a_i E_i$, where $E_1,\ldots,E_r$ are its irreducible components.
 
\begin{proposition}\label{prop:U_maps_to_X}
Let $X$ be an integral regular noetherian scheme over $\mathbb{Q}$, let $D \subseteq X$ be a prime divisor, and let $a=(a_1,\ldots,a_r)$ be an $r$-tuple of positive integers with $r \geq 1$.
Let $U = U_{X,D,a}$ be the Campana space.
Then the following statements hold.
\begin{enumerate}
\item $U$ is an integral noetherian scheme.
\item The morphism $U \to X$ is affine, flat, surjective, and of finite type.
\item Over the open subset $X \setminus \supp(D)$, the morphism $U \to X$ is a $(\mathbb{G}_m)^r$-torsor. In particular, $U \to X$ is smooth over this open subset.
\item If $D$ is regular (as a scheme), then the scheme $U$ is regular, and the orbifold base divisor of $U \to X$ is $(1-\frac{1}{\inf(a_1,\ldots,a_r)})D$.
\item The gcd-multiplicity of $U \to X$ over the generic point of $D$ is $\gcd(a_1,\ldots,a_r)$.
\end{enumerate}
\end{proposition}
\begin{proof}
With the notation as above, the morphism $V\to\mathbb{A}^1$ is flat, since it is the base change
of the flat morphism $q_a=q\circ p_a$ along $q$.
Since both $V$ and $\mathbb{A}^1$ are affine, this morphism is also affine, and it is clearly surjective and of finite type.
In particular, its base change $V' \to X'$ is flat, surjective, affine and of finite type, and thus by fppf descent, the morphism $U \to X$ has these properties as well; this proves $(ii)$.

If $\mathcal{C}_{X,D,a}$ denotes the Campana stack, the morphism $U_{X,D,a} \to \mathcal{C}_{X,D,a}$ is a $(\mathbb{G}_m)^r$-torsor by construction. 
Since the morphism $\mathcal{C}_{X,D,a} \to X$ is an isomorphism over $X \setminus \supp(D)$, we find that $(iii)$ holds.

As $X$ is noetherian and $U \to X$ is finite type, $U$ is noetherian.
Since $X$ is integral and $U \to X$ is flat, we have that every irreducible component of $U$ dominates $X$ under $U \to X$.
Since the generic fibre of $U \to X$ is irreducible, there is a unique irreducible component of $U$ dominating $X$.
Consequently, $U$ is integral; this shows $(i)$.

If $D$ is regular as a scheme, Corollary~\ref{cor:regularity} implies that $U$ is regular.
Moreover, since the morphism $U \to X$ is smooth with nonempty fibres outside $\supp(D)$, the orbifold base divisor of $U \to X$ must be supported on $D$.
To prove $(iv)$ and $(v)$, it thus suffices to consider the fibre of $U \to X$ over the generic point of $D$.
But this is done in the discussion preceding the statement of this proposition.
\end{proof}

\section{Campana stacks of generalized C-pairs}

Campana's theory studies multiplicity conditions such as ``at least $m$'' and ``divisible by $m$''. As we saw in the previous section, these conditions admit natural stacky interpretations via Campana stacks and root stacks. It is therefore natural to ask for a broader framework allowing other multiplicity conditions.

In this section, we allow the multiplicities to lie in subsets of $\mathbb{Z}_{\geq 1}$ which are given as finite unions of subsemigroups. This leads to the notion of a generalized C-pair. Such data naturally give rise to Campana stacks built from suitable open subsets of $\mathbb{A}^r$; see Section \ref{section:campana_stacks_general}. 

This point of view is related in spirit to Moerman's notion of $M$-points, where $M$ is a commutative monoid \cite{Moerman}, which also allows for more general multiplicity conditions. The framework considered here is, however, of a different nature.

\subsection{Generalized C-pairs}

We begin by fixing the combinatorial data which will control the allowed multiplicities. The basic idea is simple: instead of requiring a multiplicity to be at least $m$, or divisible by $m$, we allow it to lie in a more general subset of $\mathbb{Z}_{\geq 1}$, provided this subset is presented as a finite union of semigroups.

A \emph{(sub)semigroup (of $\mathbb{Z}_{\geq 1}$)} is a (possibly empty) subset of $\mathbb{Z}_{\geq 1}$ which is closed under addition.
Any such semigroup has a minimal finite set of additive generators, whose elements are called the atoms of $M$.
For example, if $m\geq 1$, then the atoms of the semigroup $\mathbb{Z}_{\geq m}$ are $m, \ldots, 2m-1$.

Let $X$ be a regular noetherian scheme and let $D_1,\ldots,D_n$ be (pairwise distinct) integral divisors on $X$.
For each $i$, let $M_i$ be a subset of $\mathbb{Z}_{\geq 1}$ together with a decomposition
\[ M_i=M_{i,1}\cup \ldots \cup M_{i,r_i} \]
as a finite union of (possibly empty) semigroups.\footnote{More precisely, $M_i$ is equipped with an integer $r_i\geq 0$ and semigroups $M_{i,1},\ldots,M_{i,r_i}$ such that $M_i=M_{i,1}\cup \ldots \cup M_{i,r_i}$.}
We call the data $(X,(D_i,M_i)_i)$ a \emph{generalized C-pair}.

For notational convenience, we suppress the integers $r_i$ and the chosen decomposition of each $M_i$.
We stress, however, that different decompositions may lead to different notions of morphism (see Remark \ref{remark:diff_morphisms}).
 
\begin{remark}\label{remark:canonical_decomposition}
If $M$ is a semigroup, unless explicitly stated otherwise, we always use the canonical decomposition with $r = 1$ and $M_1 = M$.
\end{remark}

\begin{definition}\label{def:C-pair-ass-to-gen}
We define the \emph{C-pair divisor} $\Delta_{(X,(D_i,M_i)_i)}$ of $(X,(D_i,M_i)_i)$, or simply $\Delta$, to be the $\mathbb{Q}$-divisor
\[ \Delta := \sum_i \left(1-\frac{1}{\inf(M_i)}\right) D_i. \]
(We use the conventions that the infimum of the empty set is $\infty$ and that $\frac{1}{\infty} = 0$.)
We refer to $(X,\Delta)$ as the C-pair \emph{associated} to $(X,(D_i,M_i)_i)$.
\end{definition}

\begin{definition} \label{def:C-pair-gen-morphism}
Let $(X,(D_i,M_i)_i)$ be a generalized C-pair with associated C-pair $(X,\Delta)$, and let $T$ be an integral scheme.
A \emph{morphism of generalized C-pairs} $T \to (X,(D_i,M_i)_i)$ is a morphism $f \colon T \to X$ such that $f$ factors through $X\setminus\floor{\Delta}$ and, such that, for every $i$ with $f(T) \not\subseteq D_i$, one can write
\[ f^* D_i = \sum_{j=1}^{r_i} E_{i,j} \quad \text{with} \quad E_{i,j} = \sum_{k=1}^{\ell_{i,j}} a_{i,j,k} E_{i,j,k}, \]
where $\ell_{i,j}\geq 0$, the $E_{i,j,k}$ are effective Cartier divisors, each coefficient $a_{i,j,k}$ lies in $M_{i,j}$, and the divisors $E_{i,j}$ have pairwise disjoint support for $1\leq j\leq r_i$.\footnote{We use the convention that the empty sum equals $0$.}
\end{definition}

As we noted above, different decompositions of $M_i$ as a finite union of semigroups can lead to different notions of a morphism to the corresponding generalized C-pair.

\begin{remark} \label{remark:diff_morphisms}
Consider $M = \mathbb{Z}_{\geq 2} \setminus \{5\}$. This can be written as
\[ \langle 2,7\rangle \cup \langle 3\rangle
\quad \text{or as} \quad
\langle 2\rangle \cup \langle 3,4,7\rangle. \]
Let $X = \mathbb{A}^1$ and $D = \{0\}$.
Let $\mathcal{X}$ (resp. $\mathcal{Y}$) be the generalized C-pair associated to the first (resp. second) decomposition.

Let $T$ be the open subset of $\mathbb{A}^3$ with coordinates $x,y,z$ obtained by removing $Z(x)\cap Z(z)$ and $Z(y)\cap Z(z)$. The morphism $T\to \mathbb{A}^1$ given by $(x,y,z)\mapsto x^2y^7z^3$ defines a morphism of generalized C-pairs to $\mathcal{X}$, but not to $\mathcal{Y}$.

The fibre over $0$ has three irreducible components $F_2, F_7, F_3$ with multiplicities $2,7,3$, given by $Z(x), Z(y), Z(z)$. The components $F_2$ and $F_7$ intersect, while $F_2\cap F_3=\emptyset$ and $F_7\cap F_3=\emptyset$.
\end{remark}

\begin{remark}
A semigroup $M$ determines a connected inf-multiple fibre whose irreducible components are indexed by the atoms of $M$, with multiplicities given by those atoms.
More generally, a union of semigroups $M = M_1 \cup \ldots \cup M_r$ corresponds to an inf-multiple fibre with $r$ connected components, where the $i$-th connected component  decomposes into irreducible components whose multiplicities are given by the atoms of $M_i$. This explains the geometric meaning of Remark \ref{remark:diff_morphisms}.
\end{remark}

Let us give some examples of generalized C-pairs to illustrate the definition of a morphism to such an object.
Firstly, if $M_i = \mathbb{Z}_{\geq 1}$ for every $i$, then every morphism $T\to X$ is a morphism to the generalized C-pair $(X, (D_i, M_i)_i)$, so we recover all $T$-points of $X$.
Secondly, if $M_i = \emptyset$ for every $i$, then a morphism $T\to X$ is a morphism to $(X, (D_i, M_i)_i)$ if and only if it factors through $X\setminus \bigcup_i D_i$, so we recover integral points with respect to $\bigcup_i D_i$.
The following further examples are the most relevant in this paper.

\begin{example}[Campana points] \label{example:campana}
For every $i$, let $m_i \in \mathbb{Z}_{\ge 1}\cup\{\infty\}$ and define
\[ M_i := \mathbb{Z}_{\ge m_i} \quad \text{with the convention that } \mathbb{Z}_{\geq \infty} = \emptyset. \]
We take the canonical decomposition of $M_i$ (since it is already a semigroup).
Let
\[ \Delta := \sum_i \Bigl(1-\frac{1}{m_i}\Bigr) D_i. \]
Then $(X,\Delta)$ is a C-pair in the sense of Section~\ref{section:c-pair}.
Let $T$ be an integral regular noetherian scheme and let $f\colon T\to X$ be a morphism with $f(T)\not\subseteq \supp(\Delta)$.
Then $f$ defines a morphism of generalized C-pairs $T \to (X,(D_i,M_i)_i)$ if and only if
\[ f^* D_i = \sum_{j=1}^{\ell_i} a_{i,j} E_{i,j} \]
for some integers $a_{i,j} \in M_i = \mathbb{Z}_{\geq m_i}$ and some effective Cartier divisors $E_{i,j}$.
Equivalently, $f$ is a morphism of C-pairs $T \to (X,\Delta)$ in the sense of Section~\ref{section:orbifold_mordell}.
\end{example}

\begin{example}[Darmon points] \label{example:darmon}
For every $i$, let $m_i\in \mathbb{Z}_{\geq 1}\cup \{\infty\}$ and define
\[ M_i := \langle m_i \rangle \quad \text{with the convention that } \langle \infty \rangle = \emptyset. \]
Let $\Delta := \sum_i \left(1-\frac{1}{m_i}\right) D_i$ be the associated C-pair divisor.
Let $T$ be an integral regular noetherian scheme and let $f\colon T\to X$ be a morphism with $f(T)\not\subseteq \supp(\Delta)$.
Then $f$ defines a morphism of generalized C-pairs $T \to (X,(D_i,M_i)_i)$ if and only if $f$ is a Darmon point of $(X,\Delta)$ (also called a classical Campana point in the literature), i.e., for every $i$ and every irreducible component $E$ of $f^{-1}(D_i)$, the multiplicity of $E$ in $f^* D_i$ is divisible by $m_i$.
Equivalently, as explained in \cite[\S 8]{Moerman}, such an $f$ is a classical Campana point of $(X, \Delta)$ if and only if it factors through the fibre product of root stacks
\[ \sqrt[m_1]{D_1/X}\times_X \cdots \times_X \sqrt[m_n]{D_n/X}. \]
\end{example}
 
\begin{example}[Abramovich points]\label{example:firmament}
For every $i$, let $a_i\leq b_i$ be positive integers and define
\[ M_i := \langle a_i \rangle \cup \langle b_i \rangle. \]
If $a_i > 1$ and $b_i$ are coprime, then this is not a semigroup, but a union of two semigroups.
Let $T$ be an integral regular noetherian scheme and let $f\colon T\to X$ be a morphism with $f(T)\not\subseteq \supp(\Delta)$.
Then $f$ defines a morphism $T \to (X,(D_i,M_i)_i)$ if and only if for every $i$, there exist disjoint effective Cartier divisors $E_i$ and $F_i$ with
\[ f^* D_i = a_i E_i + b_i F_i. \]
If $T$ is one-dimensional then $f$ is a morphism of generalized C-pairs if and only if, for every $i$ and every irreducible component $E$ of $f^{-1}(D_i)$, the multiplicity of $E$ in $f^* D_i$ is divisible by $a_i$ or $b_i$ (since distinct prime divisors on a Dedekind scheme are disjoint).
This condition coincides with Abramovich's notion of maps to a ``firmament'' \cite[Definition~2.4.3]{AbramovichBirGeom}.
However, we stress that the equivalence between morphisms $T \to (X,(D_i,M_i)_i)$ and Abramovich's firmament maps only holds when $\dim T = 1$.
The multiplicity condition alone (without any disjointness conditions) cannot be encoded by an algebraic stack.
In contrast, the existence of a decomposition $f^* D_i = a_i E_i + b_i F_i$ with $E_i$ and $F_i$ disjoint can be interpreted stack-theoretically; see Section~\ref{section:campana_stacks_general}.
\end{example}

\subsection{Finite unions of semigroups} \label{section:campana_stacks_general}

Let $X$ be a scheme and let $D$ be an effective Cartier divisor on $X$. 
In this subsection, let $M$ be a subset of $\mathbb{Z}_{\geq 1}$ together with an integer $s\geq 1$ and a decomposition
\[ M = M_1 \cup \ldots \cup M_s, \]
where each $M_i$ is a non-empty semigroup. 

For every $i$, let $\{a^i_1,\ldots,a^i_{r_i}\}$ be the minimal set of generators of $M_i$. Define $r = r_1+\cdots+r_s$ and write $a=(a^1,\ldots,a^s)$, viewed as an $r$-tuple of positive integers. We define $\mathcal{C}_{X,D,M_1\cup\ldots\cup M_s}$ (or simply $\mathcal{C}_{X,D,M}$) as an open substack of $\mathcal{C}_{X,D,a}$.

Let $Z=Z_{M_1\cup \ldots \cup M_s}$ be the union of all intersections $Z(x_i)\cap Z(x_j)$ in $\mathbb{A}^r$ such that the indices $i$ and $j$ do not lie in the same block of $[1,r]$ determined by the partition
\[
[1,r] = [1,r_1]\cup [r_1+1,r_1+r_2]\cup \cdots \cup [r_1+\cdots+r_{s-1}+1,r].
\]
Then $\mathbb{A}^r\setminus Z$ is stable under the action of $\mathbb{G}_m^r$ on $\mathbb{A}^r$. Let
\[
P_a^\circ \colon [(\mathbb{A}^r \setminus Z)/\mathbb{G}_m^r]\to [\mathbb{A}^1/\mathbb{G}_m]
\]
be the restriction of $P_a$.

\begin{definition}
We define the \emph{Campana stack of the generalized C-pair $(X,(D,M_1\cup \ldots \cup M_s))$} via the following Cartesian diagram   
\[
\xymatrix{ 
\mathcal{C}_{X,D, M_1\cup \ldots\cup M_s}  \ar[rr] \ar[d] & & [(\mathbb{A}^r \setminus Z)/\mathbb{G}_m^r] \ar[d]^{P_a^\circ} \\ X \ar[rr]_{f_D} & & [\mathbb{A}^1/\mathbb{G}_m]}
\] We define the \emph{Campana space $U_{X,D,M_1\cup \ldots\cup M_s}$ of the generalized C-pair $(X,(D,M_1\cup \ldots \cup M_s))$} via the following Cartesian diagram 
\[
\xymatrix{U_{X,D,M_1\cup \ldots \cup M_s} \ar[rr] \ar[d] & & \mathbb{A}^r\setminus Z \ar[d]  \\ 
\mathcal{C}_{X,D,M_1\cup \ldots \cup M_s} \ar[rr]  & &  [(\mathbb{A}^r \setminus Z)/\mathbb{G}_m^r]}
\]
\end{definition}

Since $\mathbb{A}^r\setminus Z$ is open in $\mathbb{A}^r$, the stack $\mathcal{C}_{X,D,M}$ is an open substack of $\mathcal{C}_{X,D,a}$, and $U_{X,D,M}$ is   an open  subscheme of  $U_{X,D,a}$.
 
\begin{proposition}\label{prop:C_pair_vs_Campana_stack_gen_0}
Let $X$ be a noetherian scheme, let $D$ be an effective Cartier divisor on $X$, and let
\[ M = M_1\cup \ldots \cup M_s \]
be a finite union of non-empty semigroups in $\mathbb{Z}_{\geq 1}$.
Let $T$ be an integral scheme, and let $f\colon T\to X$ be a morphism whose image is not contained in $\supp(D)$.
Then $f$ defines a morphism of generalized C-pairs $T \to (X,(D,M))$ if and only if it factors through $\mathcal{C}_{X,D,M}$.
\end{proposition}
\begin{proof}
For each $j=1,\ldots,s$, let $\{a^j_1,\ldots,a^j_{r_j}\}$ be the minimal set of generators of $M_j$.
Set $r=r_1+\cdots+r_s$, and let $a=(a^1,\ldots,a^s)$ be the resulting $r$-tuple of positive integers.
Since $f(T)\not\subseteq \supp(D)$, the composition
\[ T \to X \to [\mathbb{A}^1/\mathbb{G}_m] \]
corresponds to the effective Cartier divisor $f^*D$ on $T$.
By the defining properties of $\mathcal{C}_{X,D,a}$ (see Lemma \ref{lemma:campana_stack_of_a_2}), a lift of $f$ to $\mathcal{C}_{X,D,a}$ is equivalent to a decomposition
\[ f^* D=\sum_{j=1}^s \sum_{k=1}^{r_j} a^j_k F_{j,k}, \]
where the $F_{j,k}$ are effective Cartier divisors on $T$.

Grouping the terms block by block, this is equivalent to writing
\[ f^* D=\sum_{j=1}^s E_j, \qquad E_j=\sum_{k=1}^{r_j} a^j_k F_{j,k}. \]
Since every element of $M_j$ is a finite sum of the generators $a^j_1,\ldots,a^j_{r_j}$, this is equivalent to requiring that every $E_j$ admits a decomposition
\[ E_j=\sum_{k=1}^{\ell_j} b_{j,k}E_{j,k}, \]
where $\ell_j\geq 0$, the $E_{j,k}$ are effective Cartier divisors, and the coefficients $b_{j,k}$ lie in $M_j$.

By construction, $\mathcal{C}_{X,D,M}$ is obtained from $\mathcal{C}_{X,D,a}$ by requiring that the corresponding map $T \to [\mathbb{A}^r/\mathbb{G}_m^r]$ land in the open substack $[(\mathbb{A}^r\setminus Z)/\mathbb{G}_m^r]$.
This is equivalent to saying that, whenever two indices belong to different blocks, the corresponding divisors on $T$ have disjoint support.
In other words, the divisors $E_1,\ldots,E_s$ have pairwise disjoint support.

Thus a lift of $f$ to $\mathcal{C}_{X,D,M}$ is equivalent to a decomposition
\[ f^* D=\sum_{j=1}^s E_j \]
such that each $E_j$ is a finite sum of effective Cartier divisors with coefficients in $M_j$, and such that the supports of $E_1,\ldots,E_s$ are pairwise disjoint.
This is precisely the condition in the definition of a morphism of generalized C-pairs.
\end{proof}

\begin{definition}
Let $(X,(D_i,M_i)_{i=1}^n)$ be a generalized C-pair, where $M_i = M_{i,1}\cup \ldots \cup M_{i,s_i}$.
We define its \emph{Campana stack} to be 
\[ \mathcal{C}_{X,(D_i,M_i)_i}=\mathcal{C}_{X,(D_i,M_{i,1}\cup \ldots \cup M_{i,s_i})_i} = \mathcal{C}_{X,D_1,M_{1,1}\cup \ldots \cup M_{1,s_1}  }\times_X \ldots \times_X \mathcal{C}_{X,D_n,M_{n,1}\cup \ldots \cup M_{n,s_n}}. \]
We define its \emph{Campana space} to be 
\[ U_{X,(D_i,M_i)_i}= U_{X,(D_i,M_{i,1}\cup \ldots \cup M_{i,s_i})_i} = U_{X,D_1,M_{1,1}\cup \ldots \cup M_{1,s_1}  }\times_X \ldots \times_X U_{X,D_n,M_{n,1}\cup \ldots \cup M_{n,s_n}}. \] 
\end{definition}

Proposition \ref{prop:C_pair_vs_Campana_stack_gen_0} immediately implies the following stacky interpretation of morphisms of generalized C-pairs. 

\begin{proposition}\label{prop:C_pair_vs_Campana_stack_gen}
Let $(X,(D_i,M_i)_i)$ be a generalized C-pair and let $T$ be an integral scheme.
A morphism $f\colon T\to X$ of schemes whose image is not contained in $\bigcup_i \supp(D_i)$  is a morphism of generalized C-pairs $T\to (X,(D_i,M_i)_i)$ if and only if $T\to X$ factors over  $
\mathcal{C}_{X,(D_i,M_i)_i} $.
\end{proposition}

\subsection{Finite unions of cyclic semigroups}

If $M$ is a finite union of cyclic semigroups, the associated Campana stacks encode conditions of the form ``divisible by one of several integers'', and are closely related to Abramovich's firmaments. 
 
For $1 \leq i < j \leq r$, let $Z_{i,j} = Z(x_i) \cap Z(x_j)$ be the closed subset of $\mathbb{A}^r$ defined by $x_i = x_j = 0$.
Define $Z = \bigcup_{1 \leq i < j \leq r} Z_{i,j}$ to be the union of these codimension two closed subsets. 
Let $a = (a_1,\ldots,a_r)$ be an $r$-tuple of positive integers.
Consider the finite union 
\[ M = \langle a_1 \rangle \cup \ldots \cup \langle a_r \rangle, \]
which is a finite union of cyclic subsemigroups of $\mathbb{Z}_{\geq 1}$.
Let $P_a^{\circ} \colon [(\mathbb{A}^r \setminus Z)/\mathbb{G}_m^r] \to [\mathbb{A}^1/\mathbb{G}_m]$ be the restriction of $P_a$ to the open substack $[(\mathbb{A}^r \setminus Z)/\mathbb{G}_m^r]$.
In the following lemma, we show that $P_a^\circ$ has finite relative inertia (see \cite[Tag~036X]{stacks-project} and \cite[Tag~050P]{stacks-project} for the definition of this concept).

\begin{lemma}\label{lemma:finite_inertia}
The relative inertia stack of $P_a^{\circ}\colon [(\mathbb{A}^r \setminus Z)/\mathbb{G}_m^r] \to [\mathbb{A}^1/\mathbb{G}_m]$ is finite over $[(\mathbb{A}^r \setminus Z)/\mathbb{G}_m^r]$.
In particular, in characteristic zero, the morphism $P_a^\circ$ is a Deligne--Mumford morphism \cite[04YW]{stacks-project}.
\end{lemma}
\begin{proof}
Put $U=\mathbb A^r\setminus Z$, $G=\mathbb G_m^r$, $H=\mathbb G_m$,
$X=[U/G]$, and $Y=[\mathbb A^1/H]$. The morphism $P_a^\circ\colon X\to Y$
is induced by the map $x\mapsto\prod_{j=1}^r x_j^{a_j}$ and the homomorphism
$\varphi\colon G\to H$ given by $t\mapsto\prod_{j=1}^r t_j^{a_j}$.

Since finiteness is local on the target for the fppf topology, it suffices
to prove that $I:=I_{X/Y}\times_X U$ is finite over $U$. On this atlas,
a relative automorphism of a point $x$ of $U$ is an element $t\in G$ such that $t\cdot x=x$
and $\varphi(t)=1$. Thus $I$ is the closed subscheme of $U\times G$ defined by
\[
  x_j(t_j-1)=0 \quad (1\leq j\leq r),
  \qquad \prod_{j=1}^r t_j^{a_j}=1.
\]

For each $i$, let $U_i\subseteq U$ be the open subset where $x_j$ is
invertible for every $j\neq i$. These opens cover $U$, since at most one
coordinate vanishes at any point of $U$. Over $U_i$, the equations above
give $t_j=1$ for $j\neq i$, and the remaining equations are
\[
  t_i^{a_i}=1, \qquad x_i(t_i-1)=0.
\]
Thus $I\times_U U_i$ is the closed subscheme of $U_i\times\mu_{a_i}$
cut out by $x_i(t_i-1)=0$. Since $\mu_{a_i}$ is finite over the base,
$I\times_U U_i$ is finite over $U_i$. This proves that $I_{X/Y}$ is
finite over $X$.

In characteristic zero, $\mu_{a_i}$ is   \'etale, so  that $I_{X/Y}\to X$ is unramified. Hence $P_a^\circ$
is a Deligne--Mumford morphism \cite[Tag~0CL0]{stacks-project}.
\end{proof}
  
We now make the connection to Abramovich's firmament points more precise (see Example \ref{example:firmament}).
 
\begin{proposition} \label{prop:campana_cyclic}
Let $X$ be a noetherian scheme and let $D$ be an effective Cartier divisor on $X$.
Let $T$ be an integral regular noetherian one-dimensional scheme. 
Let $f \colon T \to X$ be a morphism with $f(T) \not\subseteq \supp(D)$.
Then the following are equivalent.
\begin{enumerate}
\item The morphism $f \colon T \to X$ defines a morphism of generalized C-pairs $T \to (X, D, \langle a_1 \rangle \cup \ldots \cup \langle a_r \rangle)$.
\item The morphism $f \colon T \to X$ factors through $\mathcal{C}_{(X, (D, \langle a_1 \rangle \cup \ldots \cup \langle a_r \rangle))}$.
\item For every irreducible component $E$ of $f^{-1}(D)$, the coefficient of $E$ in $f^* D$ is divisible by $a_i$ for some $i = 1, \ldots, r$.
\end{enumerate}  
\end{proposition}
\begin{proof}
The equivalence of the $(i)$ and $(ii)$ is Proposition~\ref{prop:C_pair_vs_Campana_stack_gen_0}.  
The equivalence of $(i)$ and $(iii)$ follows from the definition of a morphism of generalized C-pairs together with the fact that any two distinct prime divisors on $T$ are disjoint. 
\end{proof}
 
\begin{remark}
The subset of ``Abramovich'' $O_{K,S}$-points of the C-pair $(X,(1-\frac{1}{m})D)$ where $m=\min(a_1,\ldots, a_r)$ over the number ring $\mathcal{O}_{K,S}$ are those where the multiplicity condition becomes ``divisible by some $a_i$'' (see Example~\ref{example:firmament}).
We learn from Lemma~\ref{lemma:finite_inertia} and Proposition~\ref{prop:campana_cyclic} that these correspond to $\mathcal{O}_{K,S}$-points of the non-separated stack $\mathcal{C}_{X,D, \langle a_1 \rangle \cup \ldots \cup \langle a_r \rangle}$ with finite inertia.
\end{remark}

\subsection{Regularity: the case of an snc divisor}   
We now study the local geometry of Campana stacks. 
As we will show, the Campana stack associated to a regular scheme $X$ of characteristic zero, an snc divisor $D=\sum_{i=1}^n D_i$, and tuples $a_i$ of positive integers is also regular (thereby generalizing Corollary \ref{cor:regularity}).
To prove this, we reduce to a local computation in the case of a regular local ring and the following generalization of Lemma \ref{lemma:comm_alg}.
   
\begin{lemma}\label{lemma:comm_alg_2}  
Let $(R,\mathfrak m)$ be an equicharacteristic regular local noetherian ring and let $f_1,\ldots,f_n \in R$ be part of a regular system of parameters of $R$.
For each $1\leq i \leq n$, fix positive integers $a_{i,1},\dots,a_{i,r_i}$.
Write $r=r_1+\cdots+r_n$ and define $s_0:=0$ and $s_i:=r_1+\cdots+r_i$ for $1 \leq i \leq n$.
For each $i$, set
\[ P_i := x_{s_{i-1}+1}^{a_{i,1}}\cdots x_{s_i}^{a_{i,r_i}} - f_i \in R[x_1,\ldots,x_r]. \]
Then
\[ A := R[x_1,\ldots,x_r]/(P_1,\ldots,P_n) \]
is regular.
\end{lemma}
\begin{proof}  
We may assume that $R$ is complete \cite[Tag~00OF]{stacks-project}, so that $R \cong K \powerseries{y_1,\dots,y_d}$ for some field $K$.
Since $f_1,\ldots,f_n$ are part of a regular system of parameters of $R$, after a change of coordinates in $K\powerseries{y_1,\dots,y_d}$ we may assume that $f_i=y_i$ for $1\leq i\leq n$.
In particular, $\partial f_i/\partial y_i=1$ is a unit in $R$.

Consider the Jacobian matrix $(\partial P_i/\partial y_j)_{1\le i,j\le n}$.
Since the monomial $x_{s_{i-1}+1}^{a_{i,1}}\cdots x_{s_i}^{a_{i,r_i}}$ does not depend on the variables $y_j$, we have $\partial P_i/\partial y_j=-\partial f_i/\partial y_j$.
Hence the Jacobian matrix equals $-(\partial f_i/\partial y_j)_{1\leq i,j\leq n}$.

Modulo $\mathfrak m$ this matrix is diagonal with diagonal entries $-1$.
Thus its determinant is a unit in $R$, and therefore also in $R[x_1,\dots,x_r]$ and in $A$.
By the Jacobian criterion \cite[Tag~0GEE]{stacks-project}, it follows that $A$ is regular.
\end{proof}

\begin{proposition}
Let $X$ be a regular algebraic stack over $\mathbb{Q}$, and let $D_1,\ldots, D_n$ be regular divisors.
For each $1\leq i\leq n$, let $r_i$ be a positive integer and let $a_i = (a_{i,1},\ldots, a_{i,r_i})$ be an $r_i$-tuple of positive integers with $\gcd(a_{i,1},\ldots, a_{i,r_i})=1$.
Assume that $\sum_{i=1}^n D_i$ has simple normal crossings.
Then, the algebraic stack $\mathcal{C}_{X,D_1,a_1}\times_X \ldots \times_X \mathcal{C}_{X,D_n,a_n}$ is regular.
\end{proposition} 
\begin{proof}  
Write $\mathcal{C} = \mathcal{C}_{X,D_1,a_1}\times_X \ldots \times_X \mathcal{C}_{X,D_n,a_n}$.
To show that $\mathcal{C}$ is regular, we may assume that $X = \Spec R$, where $R$ is a regular local ring with maximal ideal $\mathfrak{m}$.
Then the divisor $D_i$ corresponds to an element $f_i\in R$.
If $D_i$ is empty, then we may discard it.
Thus, as we assumed $\sum_{i=1}^n D_i$ to have simple normal crossings, we may assume that $f_1,\ldots, f_n \in R$ form part of a regular system of parameters of $R$. 
 
Now, the moduli map $X \to [\mathbb{A}^1/\mathbb{G}_m]\times \ldots \times [\mathbb{A}^1/\mathbb{G}_m]$ associated to $D = \sum_{i=1}^n D_i$ factors over the morphism $X \to \mathbb{A}^1\times \ldots \times \mathbb{A}^1$ given by $(f_1,\ldots, f_n)$.
Write $K_i = K_{a_i}$; see Definition \ref{def:K}.
Then we obtain the following diagram in which all squares are Cartesian.
\[ \xymatrix{
V \ar[r] \ar[d]  &  \prod_{i=1}^n \mathbb{A}^{r_i}  \ar[d]  & \\ 
\mathcal{C}  \ar[r] \ar[d]  & \prod_{i=1}^n  [\mathbb{A}^{r_i}/(\mathbb{G}_m\otimes K_i)] \ar[d] \ar[r]  & \prod_{i=1}^{n} [\mathbb{A}^{r_i}/\mathbb{G}_m^{r_i}] \ar[d] \\
\Spec R  = X \ar[r]^{(f_1,\ldots,f_n)}  &  \prod_{i=1}^n \mathbb{A}^1 \ar[r]  & \prod_{i=1}^n [\mathbb{A}^1/\mathbb{G}_m]
} \]
In this diagram, the vertical map $\prod_{i=1}^n\mathbb{A}^{r_i}\to \prod_{i=1}^n \mathbb{A}^1$ is the morphism $p_{a_1}\times \ldots \times p_{a_n}$.   
Write $r = r_1 + \ldots + r_n$ and define $s_0:=0$ and $s_i := r_1 + \ldots + r_i$ for $1 \leq i \leq n$.
For each $i$, set
\[ P_i := x_{s_{i-1}+1}^{a_{i,1}}\cdots x_{s_i}^{a_{i,r_i}} - f_i \in R[x_1,\ldots,x_r]. \]
Then $V$ is isomorphic to $\Spec R[x_1,\ldots,x_r]/(P_1,\ldots,P_n)$ which is regular by Lemma~\ref{lemma:comm_alg_2}.
Since $V$ is regular and $V \to \mathcal{C}$ is smooth surjective, we see that $\mathcal{C}$ is regular, as required.  
\end{proof}

\begin{corollary}\label{cor:campana_stack_regular_snc}
Let $X$ be a regular algebraic stack over $\mathbb{Q}$, and let $D_1,\ldots, D_n$ be regular divisors.
For each $1\leq i\leq n$, let $r_i$ be a positive integer and let $a_i$ be an $r_i$-tuple of positive integers.
Assume that $\sum_{i=1}^n D_i$ has simple normal crossings.
Then, the algebraic stack $\mathcal{C}_{X,D_1,a_1}\times_X \ldots \times_X \mathcal{C}_{X,D_n,a_n}$ is regular, and $U_{X,D_1,a_1}\times_X \ldots \times_X U_{X,D_n,a_n}$ is regular.
\end{corollary} 
\begin{proof}
Similar to Corollary~\ref{cor:regularity}.  
\end{proof}
 
\begin{corollary} \label{cor:campana-stack-of-gen-C-pair-is-regular}
Let $X$ be a regular algebraic stack over $\mathbb{Q}$, and let $(X,(D_i,M_i)_{i=1}^n)$ be a generalized C-pair with each $M_i$ a semigroup such that the divisor $\sum_i D_i$ has simple normal crossings.
Then $\mathcal{C}_{X,(D_i,M_i)_i}$ is regular and the associated Campana space $U_{X,(D_i,M_i)_i}$ is regular. 
\end{corollary}
\begin{proof}
This follows from Corollary \ref{cor:campana_stack_regular_snc}.  Recall that $M_i$ is equipped with a decomposition as a union of semigroups, say $M_i = M_{i,1}\cup \ldots \cup M_{i,s_i}$. Let $a_{i,j}$ denote the   $r_{i,j}$-tuple of atoms of $M_{i,j}$.  Let $a_i = (a_{i,1}, \ldots, a_{i,s_i})$ and recall  that $\mathcal{C}_{X,D_i,M_i}$ is an open substack of $\mathcal{C}_{X,D_i,a_i}$. In particular, $\mathcal{C}_{X,(D_i,M_i)_i}$ is an open substack of the regular stack $\mathcal{C}_{X,D_1,a_1}\times_X \ldots \times_X \mathcal{C}_{X,D_n,a_n}$, and is thus regular. 
\end{proof}

\section{Properties of the Campana space and applications} \label{section:generalized_C-pair_theorems}

\subsection{Geometry of the Campana space}

In this section, we study the geometry of the Campana space associated to a generalized C-pair.
We start by proving the following more precise result on the Campana space associated to a generalized C-pair $(X,(D_i,M_i)_i)$, from which Theorem \ref{thm:existence_of_U_intro} follows.

\begin{theorem} \label{thm:properties_of_Campana_space}
Let $(X,(D_i,M_i)_{i=1}^n)$ be a generalized C-pair, where $X$ is an integral regular noetherian scheme over $\mathbb{Q}$ and the $D_i$ are prime divisors on $X$.
Let $U = U_{X,(D_i,M_i)_{i=1}^n}$ be the Campana space of $(X,(D_i,M_i)_i)$, and let $(X,\Delta)$ be the associated C-pair (Definition \ref{def:C-pair-ass-to-gen}).
Then the following statements hold.
\begin{enumerate}
\item The natural morphism $U \to X$ induces a morphism of generalized C-pairs $U \to (X,(D_i,M_i)_i)$, and hence a morphism of C-pairs $U \to (X,\Delta)$.
\item $U \to X$ is quasi-affine, flat, and of finite type, with image $X \setminus \floor{\Delta}$. If every $M_i$ is a semigroup, the morphism $U \to X$ is affine.
\item Over the open subset $X \setminus \bigcup_i D_i$, the morphism $U \to X$ is a $(\mathbb{G}_m)^k$-torsor for some integer $k$. In particular, $U \to X$ is smooth over this open subset.
\item Suppose that $\sum_{i=1}^n D_i$ has simple normal crossings. Then $U$ is regular and the orbifold base divisor of $U \to X$ is $\Delta$.
\item Suppose that for some $i$, the subsemigroup of $\mathbb{Z}_{\geq 1}$ generated by $M_i$ has finite complement and that $D_i$ is distinct from every $D_j$, $j \neq i$. Then the fibre of $U \to X$ over the generic point of $D_i$ is not divisible.  
\end{enumerate}
\end{theorem}
\begin{proof}
Let $U_i := U_{X,D_i,M_i}$ for $1 \leq i \leq n$.
By definition, $U = U_1\times_X \ldots \times_X U_n$.

We first prove $(i)$.
For every $i$, the morphism $U_i \to X$ induces a morphism of generalized C-pairs $U_i\to (X,(D_i,M_i))$ by Proposition \ref{prop:C_pair_vs_Campana_stack_gen}.
Composing with $U \to U_i$, we obtain a morphism of generalized C-pairs $U\to (X,(D_i,M_i))$ for every $i$.
Hence we obtain a morphism of generalized C-pairs $U \to (X,(D_i,M_i)_i)$, and hence a morphism of C-pairs $U \to (X,\Delta)$.

We now prove the remaining statements by induction on $n$. 
For $n = 1$, write $D = D_1$ and $M = M_1$.
Choose a decomposition $M = N_1 \cup \ldots \cup N_s$ into semigroups, and let $a$ be the tuple of atoms.
Then $U_{X,D,M}$ is an open subscheme of $U_{X,D,a}$ obtained by removing only codimension two strata, with the removed strata mapping to $D$ under $U \to X$.
If $M$ is already a semigroup (that is, if $s=1$ in the decomposition of $M$ given above), no such strata need to be removed.
Hence the irreducible components of the fibres and their multiplicities are unchanged.
Consequently, for $n=1$, assertions $(ii)$--$(v)$ can be immediately deduced from Proposition~\ref{prop:U_maps_to_X}.

We now treat the case $n > 1$.
Write $U' = U_{X,(D_i,M_i)_{i=1}^{n-1}}$ and $U_n = U_{X,D_n,M_n}$.
Then $U = U' \times_X U_n$.
By induction, the morphisms $U' \to X$ and $U_n \to X$ are quasi-affine, flat, and of finite type, and thus so is $U \to X$.
If every $M_i$ is a semigroup, then the same argument shows that $U \to X$ is in fact affine.
Moreover, the image of $U \to X$ is given as the intersection of the images of $U' \to X$ and $U_n \to X$; this inductively implies $(ii)$.
Assertion $(iii)$ also immediately follows from this description of $U$.

In $(iv)$, the assertion that $U$ is regular follows immediately from Corollary~\ref{cor:campana-stack-of-gen-C-pair-is-regular}.
It also follows from $(iii)$ that the orbifold base divisor of $U \to X$ is supported on $\bigcup_i D_i$.
So to prove $(iv)$, it suffices to study the fibre of $U \to X$ over the generic point $\eta_i$ of some $D_i$.
Note that if $\sum_i D_i$ has simple normal crossings, we have that $D_i \neq D_j$ for $i \neq j$.
Reordering the $D_i$, we may assume that $i = n$.
Then, in a neighborhood of $\eta_i$, the morphism $U' \to X$ is smooth and surjective (by $(iii)$) and the fibre of $U_n \to X$ over $\eta_i$ has inf-multiplicity $\inf(M_i)$ and gcd-multiplicity $\gcd(M_i)$.
Since precomposing with a smooth surjective morphism changes neither the inf-multiplicity nor the gcd-multiplicity, it follows that the fibre of $U \to X$ over $\eta_i$ also has inf-multiplicity $\inf(M_i)$ and gcd-multiplicity $\gcd(M_i)$.
This argument implies both $(iv)$ and $(v)$.
\end{proof}

\begin{remark}
Assume that $(X,\Delta)$ is a smooth projective C-pair over an algebraically closed field of characteristic zero.
Let $U = U_{X,\Delta}$ be the associated Campana space.
In the terminology of \cite[Definition~1.2]{BJRSeveri} (which extends the definitions in \cite[Subsection~1.3.1]{Ca04}), it follows from Theorem~\ref{thm:properties_of_Campana_space} that the Kodaira dimension of $U \to X$ equals the Kodaira dimension of $(X,\Delta)$.
Moreover, if $(X,\Delta)$ is a special C-pair (in the sense of \cite{Ca11}), then $U$ is special.
\end{remark}

Theorem~\ref{thm:properties_of_Campana_space} allows one to transfer geometric properties of the associated C-pair $(X,\Delta)$ to the Campana space $U$.
In particular, one obtains the following consequence for weak specialness and specialness; note that it implies Corollary~\ref{cor:ws_intro}.
    
\begin{corollary}\label{cor:ws_text}
Let $(X,(D_i,M_i)_{i=1}^n)$ be a generalized C-pair with $X$ a smooth variety over a field $k$ of characteristic zero, and let $(X,\Delta)$ be the associated C-pair.
Assume that $\sum_{i=1}^n D_i$ has simple normal crossings.

\begin{enumerate}
\item Assume that $X\setminus \floor{\Delta}$ is weakly special and that, for every $1\leq i \leq n$ with $M_i$ nonempty, the subsemigroup of $\mathbb{Z}_{\geq 1}$ generated by $M_i$ has finite complement. Then the Campana space $U$ of $(X,(D_i,M_i)_{i=1}^n)$ is a smooth weakly special variety over $k$.
\item If $(X,\Delta)$ is of general type and $\dim X>0$, then $U$ is not special.
\end{enumerate}
\end{corollary} 
\begin{proof}
That $U$ is a smooth variety over $k$ was already shown in Theorem~\ref{thm:properties_of_Campana_space}.
Moreover, Theorem \ref{thm:properties_of_Campana_space} also says that the morphism $U \to X$ is flat with image $X\setminus \floor{\Delta}$ and has no divisible fibres over any codimension one point of $X\setminus \floor{\Delta}$.
Therefore, the first statement follows from Theorem~\ref{thm:family_of_ws_vars}.
    
To prove the second statement, note that $U \to X$ induces a flat and dominant morphism with orbifold base divisor $\Delta$.
It follows from \cite[Remark~2.11]{BJRSeveri} that $U \to X$ is a morphism of general type.
Thus, $U$ is not special.
\end{proof}

\subsection{Integral points of generalized C-pairs}

Integral points on the Campana space of a generalized C-pair and those of the generalized C-pair itself are closely related; see Theorem \ref{thm:U_is_ar_spec_iff_XDelta_is_ar_spec} for a precise statement.
We begin by formulating the appropriate notion of potential density of integral points for generalized C-pairs.

Let $(X,(D_i,M_i)_i)$ be a generalized C-pair, where $X$ is a smooth variety over a number field $K$.   
We say that $(X,(D_i,M_i)_i)$ is \emph{arithmetically special} (or \emph{satisfies potential density of integral points}) if there is a finite field extension $L/K$, a finite set $S$ of finite places of $L$, and a regular model  $\mathcal{X}$ for $X_L$ of finite type over $\mathcal{O}_{L,S}$ such that, for $\mathcal{D}_i$ the closure of $D_i$ in $\mathcal{X}$, the image of the set $(\mathcal{X}, (\mathcal{D}_i,M_i)_i)(\mathcal{O}_{L,S})$ of morphisms of generalized C-pairs   
\[ \Spec \mathcal{O}_{L,S}\to (\mathcal{X}, (\mathcal{D}_i,M_i)_i) \]
is dense in $X_L$.
Note that if $(X,(D_i,M_i)_i)$ is arithmetically special, then the associated C-pair $(X,\Delta)$ is also arithmetically special.

Campana's observation that, given a morphism of C-pairs $U\to (X,\Delta)$, the integral points on $U$ induce integral points on $(X,\Delta)$ extends naturally to generalized C-pairs.
We first show that the property of being a morphism of generalized C-pairs spreads out over a suitable open subset of the base.  
    
\begin{lemma}\label{lemma:spread_out_generalized_cpair}
Let $B$ be an integral regular noetherian scheme with function field $K$. 
Let $(X,(D_i,M_i)_i)$ be a generalized C-pair over $K$, where $X$ is smooth over $K$.  
Let $U$ be a smooth variety over $K$ and let
\[ f\colon U\to (X,(D_i,M_i)_i) \]
be a morphism of generalized C-pairs.      
Then there exists a dense open subset $B^{\circ} \subseteq B$, a smooth model $\mathcal{X}$ of $X$ over $B^{\circ}$, and a smooth model $\mathcal{U}$ of $U$ over $B^{\circ}$ such that $f$ extends to a morphism of generalized C-pairs
\[ \mathcal{F}\colon \mathcal{U}\to (\mathcal{X}, (\mathcal{D}_i,M_i)_i), \]
where $\mathcal{D}_i$ denotes the closure of $D_i$ in $\mathcal{X}$.
\end{lemma}
\begin{proof}   
It suffices to treat the case $n=1$, so we write $D=D_1$ and $M=M_1$.
Recall that, by definition of a generalized C-pair, the subset $M$ of $\mathbb{Z}_{\geq 1}$ comes equipped with a decomposition $M = N_1 \cup \ldots \cup N_r$ where each $N_i$ is a semigroup.  
Shrinking $B$ if necessary, we may choose a morphism $F\colon \mathcal{U}\to \mathcal{X}$ extending $f\colon U\to X$, where $\mathcal{U}$ and $\mathcal{X}$ are smooth finite type $B$-schemes.  

If $M$ is empty, then the image of $f$ is contained in $X\setminus \supp(D)$, and after shrinking $B$ we may arrange that the extension $F\colon \mathcal{U}\to \mathcal{X}$ factors over $\mathcal{X}\setminus \supp(\mathcal{D})$, as required.

Assume now that $M$ is nonempty. If $f$ factors over $\supp(D)$, then $F$ factors over $\supp(\mathcal{D})$, and the conclusion clearly holds.
Thus, we may assume that the image of $f$ is not contained in $\supp(D)$.
After replacing $B$ by a dense open subset, we may choose smooth models $\mathcal{U}$ and $\mathcal{X}$ over $B$ such that $f$ extends to a morphism $F\colon \mathcal{U}\to \mathcal{X}$.
Note that its image is not contained in $\supp(\mathcal{D})$, where $\mathcal{D}$ is the closure of $D$ in $\mathcal{X}$.
Since $f$ is a morphism of generalized C-pairs, we may write
\[ f^*D = \sum_{j=1}^{r}\ \sum_{k=1}^{\ell_j} a_{jk} E_{jk}, \]
where each $\ell_j$ is a nonnegative integer, the $E_{jk}$ are effective Cartier divisors on $U$, the coefficients satisfy $a_{jk}\in N_j$, and the divisors
\[ E_j := \sum_{k=1}^{\ell_j} a_{jk} E_{jk} \]
have pairwise disjoint support for $j=1,\ldots,r$.
Let $\mathcal{E}_{jk}$ denote the schematic closure of $E_{jk}$ in $\mathcal{U}$.
The Cartier divisor $F^*\mathcal{D}$ on $\mathcal{U}$ agrees with $f^*D$ on the generic fibre over $B$. Hence we can write
\[ F^*\mathcal{D} = \sum_{j=1}^{r}\ \sum_{k=1}^{\ell_j} a_{jk}\, \mathcal{E}_{jk} + V, \]
where $V$ is an effective divisor supported in fibres over a closed subset of $B$. Shrinking $B$ further, we may assume that $V=0$, and thus obtain an equality of Cartier divisors
\[ F^*\mathcal{D} = \sum_{j=1}^{r}\ \sum_{k=1}^{\ell_j} a_{jk}\, \mathcal{E}_{jk}. \]
Also, shrinking $B$ further again, the divisors
\[ \mathcal{E}_j := \sum_{k=1}^{\ell_j} a_{jk}\, \mathcal{E}_{jk} \]
have pairwise disjoint support (since disjointness of supports holds on the generic fibre). 
Thus $F$ is a morphism of generalized C-pairs $\mathcal{U}\to (\mathcal{X},(\mathcal{D},M))$, as required.
\end{proof}

\begin{proposition}\label{prop:campana_integral_points}
Let $ (X,(D_i,M_i)_i)$ be a generalized C-pair over a number field $K$, where $X$ is a smooth variety over $K$. 
Let $U$ be a smooth variety and let $U\to (X,(D_i,M_i)_i)$ be a morphism of generalized C-pairs.  
Then there is a finite set of finite places $S$ of $K$, a smooth model $\mathcal{U}$ for $U$ over $\mathcal{O}_{K,S}$, a smooth model $\mathcal{X}$ for $X$ over $\mathcal{O}_{K,S}$ such that, 
for every  number field $L/K$ and finite set $T$ of finite places of $L$ containing all the places of $L$ lying over $S$, the image of $\mathcal{U}(\mathcal{O}_{L,T})\to \mathcal{X}(\mathcal{O}_{L,T})$ is contained in $(\mathcal{X},(\mathcal{D}_i,M_i)_i)(\mathcal{O}_{L,T})$, where  $\mathcal{D}_i$ is the closure of  $D_i$ in $\mathcal{X}$.
\end{proposition}
\begin{proof}
By Lemma \ref{lemma:spread_out_generalized_cpair}, we may choose $S$, $\mathcal{U}$ and $\mathcal{X}$ such that $U\to X$ extends to a morphism $\mathcal{U}\to(\mathcal{X}, (\mathcal{D}_i,M_i)_i)$ of generalized C-pairs over $\mathcal{O}_{K,S}$.  
Now, let $\Spec \mathcal{O}_{L,T}\to \mathcal{U}$ be an $\mathcal{O}_{L,T}$-point of $\mathcal{U}$. 
Since $\mathcal{U}\to (\mathcal{X}, (\mathcal{D}_i,M_i)_{i=1}^n)$ is a morphism of generalized C-pairs, the composition \[\Spec \mathcal{O}_{L,T}\to \mathcal{U}\to \mathcal{X}\] is also a morphism of generalized C-pairs \[\Spec \mathcal{O}_{L,T}\to (\mathcal{X},(\mathcal{D}_i,M_i)_{i=1}^n).\]  This completes the proof of the proposition.
\end{proof}

\begin{remark}
In the setting of C-pairs, Proposition~\ref{prop:campana_integral_points} was first written down by Abramovich in \cite[Proposition~2.1.8]	{AbramovichBirGeom}.
In that context it is the natural arithmetic analogue of Campana's observations in the analytic and geometric setting, as observed by Peyre (see \cite[Remark~9.22]{Ca04}).
Our proof of Proposition~\ref{prop:campana_integral_points} follows the same ``spreading out'' argument.
\end{remark}

The previous proposition implies in particular that potential density of integral points on a smooth variety mapping dominantly to a generalized C-pair induces potential density on the generalized C-pair.

\begin{corollary}\label{cor:potential_density_of_U_to_gen_c-pair}
Let $ (X,(D_i,M_i)_i)$ be a generalized C-pair over a number field $K$, where $X$ is a smooth variety over $K$. 
Let $U$ be a   variety and let $U\to (X,(D_i,M_i)_i)$ be a dominant morphism of generalized C-pairs.   If $U$ satisfies potential density of integral points, then $(X,(D_i,M_i)_i)$ satisfies potential density of integral points. 
\end{corollary}
\begin{proof} We may assume that $U$ is smooth. Since $U\to X$ is dominant, its image in $X$ is dense.  By potential density of integral points on $U$, replacing $K$ by a finite field extension if necessary,  we have that $\mathcal{U}(\mathcal{O}_{K,S})$ is dense in $U$ for some finite set $S$ of finite places of $K$ and some finite type model $\mathcal{U}$ for $U$ over $\mathcal{O}_{K,S}$. Enlarging $K$ and $S$ if necessary, it follows from Proposition \ref{prop:campana_integral_points} that, for some finite type model $\mathcal{X}$ over $\mathcal{O}_{K,S}$ with $\mathcal{D}_i$ the closure of $D_i$ in $\mathcal{X}$, the $\mathcal{O}_{K,S}$-points on $(\mathcal{X},(\mathcal{D}_i,M_i)_i)$ are dense in $X$. 
\end{proof}

For the Campana space, the converse also holds. This yields the following equivalence, which is the main arithmetic result of this subsection. 

\begin{theorem} \label{thm:U_is_ar_spec_iff_XDelta_is_ar_spec}
Let $(X,(D_i,M_i)_i)$ be a generalized C-pair, where $X$ is a smooth variety over $K$.
Let $U$ be the Campana space of $(X,(D_i,M_i)_i)$. 
Then $U$ satisfies potential density of integral points if and only if $(X,(D_i,M_i)_i)$ satisfies potential density of integral points.  
\end{theorem}
\begin{proof}  
By Theorem~\ref{thm:properties_of_Campana_space}, the morphism $U \to X$ induces a morphism of generalized C-pairs $U \to (X,(D_i,M_i)_i)$.
Therefore, if $U$ satisfies potential density of integral points, it follows from Corollary \ref{cor:potential_density_of_U_to_gen_c-pair} that $(X,(D_i,M_i)_i)$ satisfies potential density of integral points.

Now, assume that $(X,(D_i,M_i)_i)$ satisfies potential density of integral points.  
Let $S$ be a finite set of finite places of $K$ and let $\mathcal{X}\to \Spec \mathcal{O}_{K,S}$ be a smooth finite type model of $X$ over $\mathcal{O}_{K,S}$ and let $\mathcal{D}_i$ be the closure of $D_i$ in $\mathcal{X}$.    
Enlarging $K$ and $S$ if necessary, we may assume that $\mathcal{O}_{K,S}$ is a principal ideal domain with infinite unit group and that the set of $\mathcal{O}_{K,S}$-points of the generalized C-pair $(\mathcal{X},(\mathcal{D}_i,M_i)_i)$ is dense in $X$.

Consider the Campana stack $\mathcal{C} = \mathcal{C}_{\mathcal{X},(\mathcal{D}_i,M_i)_i}$ of the C-pair $(\mathcal{X},(\mathcal{D}_i,M_i)_i)$, and let $\mathcal{U}$ be the associated Campana space.
By construction, the natural morphism $\mathcal{U}\to \mathcal{C}$ is a $\mathbb{G}_m^r$-torsor for some $r>0$.
By Proposition~\ref{prop:C_pair_vs_Campana_stack_gen}, every $\mathcal{O}_{K,S}$-point $\Spec \mathcal{O}_{K,S} \to (\mathcal{X}, (\mathcal{D}_i,M_i)_i)$ which does not factor through the support of $\sum_i \mathcal{D}_i$ lifts to $\mathcal{C}$.
Moreover, since $\mathcal{U}\to \mathcal{C}$ is a $\mathbb{G}_m^r$-torsor and $\mathcal{O}_{K,S}^\times$ is infinite, the fibre of $\mathcal{U}\to \mathcal{X}$ over such an $\mathcal{O}_{K,S}$-point has a dense set of $\mathcal{O}_{K,S}$-points.
(Here we use that each $\mathbb{G}_m$-torsor over $\mathcal{O}_{K,S}$ is trivial, as $\mathcal{O}_{K,S}$ is a principal ideal domain.)
This shows that $\mathcal{U}(\mathcal{O}_{K,S})$ is dense, as required.
\end{proof}

\begin{remark}[Faltings + $\epsilon$ and generalized C-pairs]
Let $(X,\sum_{i=1}^n(1-\frac{1}{m_i})D_i)$ be a C-pair of general type, where $X$ is a smooth proper curve of genus $g$ over a number field $K$ and $m_i\in \mathbb{Z}_{\geq 1}\cup \{\infty\}$.
 
Let $S$ be a finite set of finite places of $K$, and let $\mathcal{X}$ be a proper regular model of $X$ over $\mathcal{O}_{K,S}$.
Writing $\Delta = \sum_{i=1}^n(1-\frac{1}{m_i})\mathcal{D}_i$, where $\mathcal{D}_i$ denotes the closure of $D_i$ in $\mathcal{X}$, Campana's Orbifold Mordell conjecture predicts that $(\mathcal{X},\Delta)(\mathcal{O}_{K,S})$ is finite.

Inside $(\mathcal{X},\Delta)(\mathcal{O}_{K,S})$, one has the subset of Darmon points, i.e., the $\mathcal{O}_{K,S}$-points of the generalized C-pair $(\mathcal{X},(\mathcal{D}_i, m_i\mathbb{Z}_{\geq 1})_{i=1}^n)$.
As first observed by Darmon \cite{Darmon} by using a version of the Chevalley--Weil theorem for finite \'etale morphisms of stacks together with Faltings's theorem \cite{Faltings2}, this subset is finite.

More generally, let $U$ be the Campana space of a generalized C-pair $(X,(D_i,M_i)_i)$ such that the semigroup generated by each $M_i$ is cofinite.
If $X\setminus \floor{\Delta}$ is weakly special, then $U$ is weakly special.
The cofiniteness condition excludes cyclic semigroups $m\mathbb{Z}$ for $m>1$, and thus places the corresponding integral points  between those of $X\setminus \supp \Delta$ and those of the associated C-pair $(X,\Delta)$.

The aforementioned ``Faltings + $\epsilon$'' arguments apply to generalized C-pairs in which each $M_i$ is either empty or generates a cyclic semigroup (as in the Darmon setting), but they do not readily extend to more general semigroup structures.
\end{remark}

\subsection{Campana's generalized Bombieri--Lang conjecture} 

The equivalence in Theorem \ref{thm:U_is_ar_spec_iff_XDelta_is_ar_spec} allows one to reformulate conjectures on integral points of C-pairs in terms of the corresponding Campana spaces.
We now apply this to Campana's generalized (weak) Bombieri--Lang conjecture.

\begin{conjecture}[Campana's generalized (weak) Bombieri--Lang conjecture]\label{conj:gen_lang}
Let $(X,\sum_{i=1}^r (1-\frac{1}{m_i})D_i)$ be a smooth proper C-pair of general type over a number field $K$.
Let $S$ be a finite set of finite places of $K$ and let $\mathcal{X}\to \Spec \mathcal{O}_{K,S}$ be a smooth proper model of $X$ over $\mathcal{O}_{K,S}$ and let $\mathcal{D}_i$ be the closure of $D_i$ in $\mathcal{X}$.   
Write $\Delta =\sum_{i=1}^r\left(1-\frac{1}{m_i}\right)\mathcal{D}_i$.
If $\dim X \geq 1$, then $(\mathcal{X},\Delta)(\mathcal{O}_{K,S})$ is not dense in $X$.
\end{conjecture}
  
We now prove that Campana's generalized (weak) Bombieri--Lang conjecture is a consequence of his conjecture that all varieties satisfying potential density of integral points are special; the converse of this statement is shown in \cite[Theorem~D]{BartschOrb}.  

\begin{corollary}\label{cor:campana_conjectures_equivalence}
Assume that every smooth  variety which satisfies potential density of integral points is special.
Then Campana's generalized (weak) Bombieri--Lang conjecture (Conjecture \ref{conj:gen_lang}) holds.  
\end{corollary}
\begin{proof}
Let $(X,\Delta)$ be a smooth proper C-pair of general type over a number field.
Assume that $(X,\Delta)$ satisfies potential density of integral points.
Then, the Campana space $U_{X,\Delta}$ also satisfies potential density of integral points (Theorem \ref{thm:U_is_ar_spec_iff_XDelta_is_ar_spec}).  
Thus, the variety $U_{X,\Delta}$ is special (by our assumption).
Since $U_{X,\Delta}\to (X,\Delta)$ is a morphism of C-pairs and $(X,\Delta)$ is of general type, it follows from the fact that $U_{X,\Delta}$ is special that $\dim X = 0$.    
\end{proof}

Since Campana's Orbifold Mordell conjecture (Conjecture \ref{conj:orb_mor}) is the one-dimensional case of Conjecture \ref{conj:gen_lang},  we have that Corollary \ref{cor:campana_conjectures_equivalence} implies Corollary \ref{cor:orb_mordell_and_campana_conj}.
Thus the study of integral points on smooth proper C-pairs may be reduced, via their Campana spaces, to the study of integral points on certain   varieties. A related result was proven by Abramovich and V\'arilly-Alvarado in \cite[Corollary~3.3]{AbramovichVA}: Vojta's conjecture (for varieties) implies Campana's generalized (weak) Bombieri--Lang conjecture.

\subsection{Campana's generalized Green--Griffiths--Lang conjecture}

In this section, we briefly explain how to use the Campana space in order to relate Campana's C-pair version of the Green--Griffiths--Lang conjecture to his conjecture about dense entire curves on Campana-special varieties.
In order to avoid relying on the language of analytic stacks, we use an algebraic perspective on entire curves to do so:
Namely, we use the fact that for any quasi-projective variety $X$ over $\mathbb{C}$ and any noncompact Riemann surface $S$, holomorphic maps $S \to X^{\an}$ are naturally in bijection with scheme morphisms $\Spec \mathcal{O}(S) \to X$; we refer to Appendix~\ref{app:entire_curves} for a more detailed discussion of this point.

Given a generalized C-pair $(X, (D_i, M_i)_i)$ with $X$ a finite type scheme over $\mathbb{C}$ and a complex-analytic space $S$, Definition~\ref{def:C-pair-gen-morphism} readily adapts to give a natural notion of a \emph{holomorphic map} $S \to (X, (D_i, M_i)_i)$.
For $S$ a noncompact Riemann surface, this notion  agrees with that of a morphism $\Spec \mathcal{O}(S) \to (X, (D_i, M_i)_i)$ as defined in Definition~\ref{def:C-pair-gen-morphism}.
We thus find that for $S$ a noncompact Riemann surface, a morphism $S \to X^{\an}$ defines a holomorphic map into the generalized C-pair $(X, (D_i, M_i)_i)$ if and only if its corresponding morphism $\Spec \mathcal{O}(S) \to X$ can be lifted to a morphism $\Spec \mathcal{O}(S) \to \mathcal{C}_{X, (D_i,M_i)_i}$ to the relevant Campana stack.
Since any line bundle on $\Spec \mathcal{O}(S)$ is trivial (see Appendix \ref{app:entire_curves}), we find that the fibre product $\Spec \mathcal{O}(S) \times_{\mathcal{C}_{X, (D_i,M_i)_i}} U_{X, (D_i,M_i)_i}$, where $U_{X, (D_i,M_i)_i}$ denotes the Campana space, is a trivial $\mathbb{G}_m^r$-torsor over $\Spec \mathcal{O}(S)$.
Consequently, every morphism $\Spec \mathcal{O}(S) \to \mathcal{C}_{X, (D_i,M_i)_i}$ can be lifted to a morphism $\Spec \mathcal{O}(S) \to U_{X, (D_i,M_i)_i}$.
If $S = \mathbb{C}$ and the morphism $\Spec \mathcal{O}(\mathbb{C}) \to X$ is dominant, we can moreover ensure that the lifted morphism $\Spec \mathcal{O}(\mathbb{C}) \to U_{X, (D_i,M_i)_i}$ is dominant as well -- such lifts are parametrized by $r$-tuples of nowhere vanishing holomorphic functions on $\mathbb{C}$ and it suffices to choose these $r$ holomorphic functions algebraically independent among each other and algebraically independent from the subfield $\mathbb{C}(X) \subseteq \mathcal{M}(\mathbb{C})$ induced by the dominant morphism $\Spec \mathcal{O}(\mathbb{C}) \to X$; as $\mathcal{M}(\mathbb{C})$ contains infinitely many algebraically independent nowhere vanishing holomorphic functions, this is easily done.
In summary, we have proven the following result, where we say that $(X, (D_i, M_i)_i)$ has a \emph{dense entire curve} if there is a holomorphic map $\mathbb{C} \to (X, (D_i, M_i)_i)$ such that the underlying map $\mathbb{C} \to X$ has Zariski-dense image.

\begin{theorem} \label{thm:U_dense_entire_curve_iff_XDelta_dense_entire_curve}
Let $(X,(D_i,M_i)_i)$ be a generalized C-pair with $X$ a smooth quasi-projective variety over $\mathbb{C}$.
Let $U$ be the Campana space of $(X,(D_i,M_i)_i)$. 
Then $U$ admits a dense entire curve if and only if $(X,(D_i,M_i)_i)$ admits a dense entire curve.
\end{theorem}

We now apply this to study Campana's generalized Green--Griffiths--Lang conjecture.

\begin{conjecture}[Campana's generalized Green--Griffiths--Lang conjecture]\label{conj:gen_ggl}
Let $(X,\Delta)$ be a smooth projective C-pair of general type over $\mathbb{C}$ with $\dim X \geq 1$. 
Then $(X, \Delta)$ does not admit a dense entire curve.
\end{conjecture}

Combining Theorem~\ref{thm:U_dense_entire_curve_iff_XDelta_dense_entire_curve} with the method used to prove Corollary~\ref{cor:campana_conjectures_equivalence} then readily yields the following result.

\begin{corollary}\label{cor:campana_conjectures_equivalence_analytic}
Assume that every smooth quasi-projective variety admitting a dense entire curve is Campana-special.
Then Campana's generalized Green--Griffiths--Lang conjecture (Conjecture \ref{conj:gen_lang}) holds.  
\end{corollary}

In \cite[Theorem~C]{BartschOrb}, it is shown that the converse statement to Corollary~\ref{cor:campana_conjectures_equivalence_analytic} also holds, i.e.\ if Campana's generalized Green--Griffiths--Lang conjecture holds, then every smooth quasi-projective variety admitting a dense entire curve is Campana-special.

\section{Remarks on Campana's Orbifold Mordell Conjecture}\label{section:remarks_on_orbi_mordell}

In this section, we discuss two explicit quotient-stack constructions attached to the surface
\[ X = \mathbb{A}^2_{\mathbb{Z}}\setminus Z(x^2y^3-1), \]
and explain how they relate potential density of integral points on (an open subset of) this surface to Campana points on a certain (generalized) C-pair.
In particular, we prove Theorem~\ref{thm:ws_contradicts_abc} from the introduction, stated below as Theorem~\ref{thm:ws_contradicts_abc_text}.

Consider the action of $\mathbb{G}_m$ on $\mathbb{A}^2$ defined by 
\[ \lambda \cdot (x,y) = (\lambda^3x, \lambda^{-2}y).\]
The quotient $[\mathbb{A}^2/\mathbb{G}_m]$ is then the Campana stack associated to the C-pair $(\mathbb{A}^1, \frac{1}{2}[0])$.
(This is implicit in the discussion in Section~\ref{subsect:regularity_campana_stacks}; the subgroup $K \subseteq \mathbb{Z}^2$ considered there is in this case the cyclic subgroup generated by $(3,-2)$.)
Consequently, integral points on the stack $[\mathbb{A}^2/\mathbb{G}_m]$ are closely linked to those on the C-pair $(\mathbb{A}^1, \frac{1}{2}[0])$. 

The $\mathbb{G}_m$-action on $\mathbb{A}^2$ restricts to a $\mathbb{G}_m$-action on the open subscheme $X = \mathbb{A}^2 \setminus Z(x^2y^3-1)$.
The quotient $[X/\mathbb{G}_m]$ is the Campana stack associated to the C-pair $(\mathbb{A}^1 \setminus \{1\}, \frac{1}{2}[0])$.

Let $R$ be a UFD, and recall that an element $x \in R$ is called \emph{$m$-full} if every prime appearing in a prime factorization of $x$ appears with exponent $\geq m$.
Using the isomorphism $(\mathbb{A}^1 \setminus \{1\}, \frac{1}{2}[0]) \cong (\mathbb{A}^1 \setminus \{0\}, \frac{1}{2}[1])$, we can interpret $R$-points on this C-pair as elements $x \in R^\times$ having the property that $x-1$ is $2$-full.
From this perspective, we quickly arrive at the following result.

\begin{theorem}\label{thm:ws_contradicts_abc_text}
Let $X = \mathbb{A}^2_{\mathbb{Z}}\setminus Z(x^2y^3-1)$ and let $\mathcal{X} = [X/\mathbb{G}_m]$ with respect to the $\mathbb{G}_m$-action defined above.
Then the following are equivalent. 
\begin{enumerate}
\item The variety $X_{\mathbb{Q}}$ satisfies potential density of integral points.  
\item There is a number field $K$ and a finite set $S$ of finite places of $K$ such that the image of $\pi_0(\mathcal{X}(\mathcal{O}_{K,S}))$ in $\pi_0(\mathcal{X}(\overline{K}))$ is infinite.  
\item The C-pair $(\mathbb{P}^1_{\mathbb{Q}}, \frac{1}{2} [0] + [1] + [\infty])$ is arithmetically special (and thus fails ``Orbifold Mordell'').
\item There is a number field $K$ and a finite set of finite places $S$ of $K$ such that $\mathcal{O}_{K,S}$ is a principal ideal domain and $\{x \in \mathcal{O}_{K,S}^\times~|~x-1~\textrm{is}~2\textrm{-full}\}$ is infinite.
\end{enumerate}
\end{theorem}
\begin{proof} 
For every number field $K$, there is a finite set of finite places $S_0$ of $K$ such that, for every $S$ containing $S_0$, the ring of $S$-integers $\mathcal{O}_{K,S}$ is a principal ideal domain.
In particular, the equivalence of $(iii)$ and $(iv)$ immediately follows from the explicit interpretation of points on C-pairs given above, after identifying $(\mathbb{A}^1 \setminus \{0\}, \frac{1}{2}[1])$ with $(\mathbb{A}^1 \setminus \{1\}, \frac{1}{2}[0])$ and $(\mathbb{P}^1, \frac{1}{2}[0] + [1] + [\infty])$.

The implication $(i) \implies (ii)$ immediately follows from the fact that $X \to \mathcal{X}$ is surjective.
For the converse implication, we simply note that $X\to \mathcal{X}$ is a $\mathbb{G}_m$-torsor, so that the claim follows from the Chevalley--Weil theorem for $\mathbb{G}_m$-torsors over stacks (Proposition~\ref{prop:cw_integrally_affine} in the appendix).
(For the sake of completeness, let us briefly indicate a direct argument.
Assuming $(ii)$ holds, choose a number field $K$ and a finite set $S$ of finite places of $K$ such that $\mathcal{O}_{K,S}$ has trivial class group and an infinite unit group, and such that that the image of $\pi_0(\mathcal{X}(\mathcal{O}_{K,S}))$ in $\pi_0(\mathcal{X}(\overline{K}))$ is infinite, hence dense.
Since every $\mathbb{G}_m$-torsor over $\Spec \mathcal{O}_{K,S}$ is trivial, we see that every $\mathcal{O}_{K,S}$-point of $\mathcal{X}$ lifts to $X$.
Since $\mathbb{G}_m(\mathcal{O}_{K,S})$ is dense in $\mathbb{G}_{m,K}$, this shows that $X(\mathcal{O}_{K,S})$ is dense, as required.)

Identifying the quotient stack $\mathcal{X}$ with the Campana stack associated to $(\mathbb{P}^1_{\mathbb{Q}}, \frac{1}{2} [0] + [1] + [\infty])$, the equivalence of $(ii)$ and $(iii)$ is an immediate consequence of Theorem~\ref{thm:C-pair-Campana-stack}.

(Alternatively, one can also prove the equivalence of $(i)$ and $(iv)$ by hand, by using the elementary observation that every $2$-full element of a UFD $R$ can be written in the form $x^2y^3$ for some $x,y \in R$.
This implies that any $R$-point of the C-pair $(\mathbb{A}^1 \setminus \{1\}, \frac{1}{2}[0])$ can be lifted along the morphism $X \to \mathbb{A}^1 \setminus \{1\}$ given by $(x,y) \mapsto x^2y^3$.
Moreover, as long as $R^\times$ is infinite, there are infinitely many distinct lifts, given by $(u^3x,u^{-2}y)$ with $u \in R^\times$.)
\end{proof}

Theorem~\ref{thm:ws_contradicts_abc_text} shows that potential density of integral points on the affine surface $X$ is equivalent to the failure of Campana's Orbifold Mordell conjecture for the C-pair curve $(\mathbb{P}^1_{\mathbb{Q}}, [0]+[\infty]+\frac{1}{2}[1])$.

Next, we consider the $\mathbb{G}_m$-stable open subset $Y = X \setminus \{(0,0)\}$.
In contrast to the previous theorem, the resulting quotient stack $[Y/\mathbb{G}_m]$ will be Deligne--Mumford, and the corresponding arithmetic statement is naturally expressed in terms of a generalized C-pair rather than an ordinary C-pair.
Namely, $[Y/\mathbb{G}_m]$ is the Campana stack associated to the generalized C-pair $(\mathbb{A}^1 \setminus \{1\}, ([0], \langle 2 \rangle \cup \langle 3 \rangle))$ introduced in Example~\ref{example:firmament}, which is closely related to Abramovich's notion of a firmament \cite[Definition~2.4.3]{AbramovichBirGeom}.

Campana's conjecture predicts that potential density of integral points on a smooth variety should persist after removing a closed subset of codimension $\geq 2$ (see \cite[Conjecture~1.17.(3)]{BJL} and \cite{HassettTschinkel}).
Thus, if $X$ has potentially dense integral points, the same should hold for $Y$.
Since $Y$ is weakly special (Corollary~\ref{cor:ab_puncture}), proving that integral points on $Y$ are not potentially dense would already contradict the Weakly Special Conjecture.

\begin{theorem} \label{thm:YmodGm_text}
Let $Y = \mathbb{A}^2 \setminus (Z(x^2y^3-1) \cup \{(0,0)\})$.
Define $\mathcal{Y} = [Y/\mathbb{G}_m]$.
Then, the following statements are equivalent.
\begin{enumerate}
\item The variety $Y_{\mathbb{Q}}$ satisfies potential density of integral points.  
\item There is a number field $K$ and a finite set $S$ of finite places of $K$ such that the image of $\pi_0(\mathcal{Y}(\mathcal{O}_{K,S}))$ in $\pi_0(\mathcal{Y}(\overline{K}))$ is infinite.  
\item The generalized C-pair $(\mathbb{A}^1\setminus \{0\}, ([1],\, \langle 2 \rangle \cup \langle 3 \rangle))$ is arithmetically special.\footnote{In other words, the generalized C-pair $(\mathbb{P}^1, ([0], \emptyset), ([\infty], \emptyset), ([1],  \langle 2 \rangle \cup \langle 3 \rangle))$.}
\item There is a number field $K$ and a finite set of finite places $S$ of $K$ such that $\mathcal{O}_{K,S}$ is a principal ideal domain and the set \[\{x\in \mathcal{O}_{K,S}^\times~|~\textrm{for every prime element } p \textrm{ of } \mathcal{O}_{K,S} \textrm{ we have } v_p(x-1) \textrm{ is divisible by }  2 \textrm{ or } 3 \}\] is infinite.
\end{enumerate}
\end{theorem}
\begin{proof}
This is similar to the proof of Theorem \ref{thm:ws_contradicts_abc_text}.

Indeed, the equivalence of $(i)$ and $(ii)$ again follows from Chevalley--Weil for $\mathbb{G}_m$-torsors (Proposition \ref{prop:cw_integrally_affine} in the appendix).

The equivalence of $(ii)$ and $(iii)$ immediately follows from Proposition~\ref{prop:C_pair_vs_Campana_stack_gen} after noting that $[Y/\mathbb{G}_m]$ is the Campana stack associated to the generalized C-pair $(\mathbb{A}^1\setminus \{0\}, ([1],\, \langle 2 \rangle \cup \langle 3 \rangle))$.

Finally, the equivalence of $(iii)$ and $(iv)$ follows immediately from the definition of an integral point on a generalized C-pair.

(One can also give an elementary argument for the equivalence of $(i)$ and $(iv)$ as follows.
Firstly, for any $\mathcal{O}_{K,S}$-point $(x,y)$ of the surface $Y$, the product $x^2y^3$ defines an element of the set in $(iv)$.
Conversely, for any PID $R$, consider an element $t \in R^\times$ such that every exponent in a prime factorization of $t-1$ is divisible by $2$ or $3$.
Then $1-t = x^2 y^3$ for some $x,y \in R^\times$, which can be chosen to be coprime.
Such a pair $(x,y)$ defines an $R$-point of the surface $\mathbb{A}^2 \setminus (Z(x^2y^3-1) \cup \{(0,0)\})$.
Moreover, for any unit $u \in R^\times$, so does the pair $(u^3x,u^{-2}y)$.
Hence, if both the set in $(iv)$ and $R^\times$ are infinite for some PID $R$, the $R$-points on the surface $Y$ are dense.)
\end{proof}

\begin{remark}
In Theorem~\ref{thm:ws_contradicts_abc_text} the stack $[X/\mathbb{G}_m]$ is not Deligne--Mumford, while in Theorem~\ref{thm:YmodGm_text}, at least after base changing to $\mathbb{Z}[1/6]$, the stack $[Y/\mathbb{G}_m]$ is Deligne--Mumford (but not separated) with finite inertia.
In fact, over $\mathbb{Z}[1/6]$, the stack $[Y/\mathbb{G}_m]$ admits a root stack description.
Let $\mathbb{G}_m(1,1)$ be the non-separated scheme obtained by doubling the origin of $\mathbb{G}_m$; let $e_1$ and $e_2$ be the origins of $\mathbb{G}_m(1,1)$.
For integers $n,m\geq 1$, consider the iterated root stack 
\[ \mathbb{G}_m(n,m):=\sqrt[n]{\mathbb{G}_m(1,1)/e_1}\times_{\mathbb{G}_m(1,1)} \sqrt[m]{\mathbb{G}_m(1,1)/e_2}. \]
The morphism $Y \to \mathbb{G}_m$ which maps $(x,y)$ to $x^2y^3$ factors over $\mathbb{G}_m(2,3)$ and induces an isomorphism of stacks $[Y/\mathbb{G}_m] \to \mathbb{G}_m(2,3)$.
Here is a drawing of this stack.

\begin{center}
\begin{tikzpicture}[scale=1.5]

    \draw[thick] (-4,0) -- (-0.04,0);
    \draw[thick] (0.04,0) -- (4,0);

    \fill (0,0.1) circle (1.5pt);
    \node at (0,0.3) {$B\mu_2$};

    \fill (0,-0.1) circle (1.5pt);
    \node at (0,-0.3) {$B\mu_3$};

\end{tikzpicture}
\end{center}
\end{remark}

\appendix

\setcounter{theorem}{0}
\renewcommand{\thetheorem}{\arabic{theorem}}

  \section{Chevalley--Weil for  torsors}\label{app:cw}
 
  In this appendix we prove a Chevalley--Weil type lifting result for torsors.
Recall that if $Y$ is a variety over a field $k$ and $G$ is a group scheme over $Y$, a morphism $X\to Y$ is a $G$-torsor (for the fppf topology) if $G$ acts freely and transitively on $X$ and there exists an fppf covering $Y'\to Y$ such that $X\times_Y Y'$ is isomorphic to $G\times_Y Y'$; see \cite{PoonenRat}.

Let $k$ be a field of characteristic zero.
Campana's notion of a special variety ascends along  finite \'etale morphisms.  Similarly,    if $G$ is  a connected algebraic group, 
the property of being special   ascends along $G$-torsors.

\begin{proposition}[Campana] \label{prop:Gtorsor}Let $Y$ be a smooth variety over $k$.  Let $G\to Y$  be a finite type separated flat group scheme with geometrically connected fibres.  
Let $X\to Y$ be a $G$-torsor with $X$ a smooth variety. Then $X$ is special if and only if $Y$ is special. 
\end{proposition}
\begin{proof}  Assume $X$ is special.  Since $X\to Y$ is surjective and $X$ is special,  it follows that $Y$ is special \cite[Lemma~2.8]{BJL}. Assume that $Y$ is special.  Since the geometric fibres of $X\to Y$ are connected algebraic groups, they are special  (see \cite[Lemma~2.11]{Bartsch}). Therefore,  as the fibres of $X\to Y $ are smooth (by Cartier's theorem), the statement follows from    \cite[Lemma~2.10]{Bartsch}.
\end{proof}

The same property holds in the setting of weakly special varieties.

\begin{proposition}\label{prop:campana}  Let $Y$ be a smooth variety over $k$.  Let $G\to Y$  be a finite type separated flat group scheme with geometrically connected fibres.  
Let $X\to Y$ be a $G$-torsor with $X$ a smooth variety. 
Then $X$ is weakly special if and only if $Y$ is weakly special. 
\end{proposition}
\begin{proof}  Assume $X$ is weakly special.  Since $X\to Y$ is surjective, it follows that $Y$ is weakly special. Conversely, assume that $Y$ is weakly special.   
Since the general fibres of $X\to Y$ are connected algebraic groups,  they are special \cite[Lemma~2.11]{Bartsch} and thus weakly special. Since $X\to Y$ is smooth and its general fibres  are   weakly special varieties,  it follows from Theorem \ref{thm:family_of_ws_vars_intro}  that $X$ is weakly special.
\end{proof}

Campana's conjecture that a variety is special if and only if it has potentially dense integral points suggests the following ascension statement: let $Y$ be an arithmetically special smooth variety and let  $X\to Y$ be a $G$-torsor with $X$ a smooth variety, where  $G\to Y$ is a finite type separated flat group scheme with geometrically connected fibres.  Then $X$ is arithmetically special.  
  In this appendix, we will prove this expectation assuming $G$ is constant (but also  in a few other cases such as when  $G\to Y$ is  an abelian scheme or a reductive group scheme).
  
\begin{theorem}[Constant group scheme] \label{thm:cw_arithmetic} 
Let $Y$ be a smooth variety over a number field $K$ and let $G$ be a finite type group scheme over $K$. Write $G_Y = G\times Y $ and  let $X\to Y$ be a $G_Y$-torsor with $X$ a smooth variety.  Then $X$ is arithmetically special if and only if $Y$ is arithmetically special.
\end{theorem}

The proof of the classical Chevalley--Weil theorem (corresponding to the case where $G$ is finite in Theorem \ref{thm:cw_arithmetic}) relies on Hermite--Minkowski's finiteness theorem: given an integer $d$, a number field $K$, and a finite set of finite places $S$ of $K$, the set of isomorphism classes of degree $d$ field extensions $L/K$ unramified outside $S$ is finite.  Our proof of Theorem \ref{thm:cw_arithmetic}  relies  on stronger finiteness results for the ring $\mathcal{O}_{K,S}$ of $S$-integers in a number field $K$. For example, if $G$ is an affine finite type group scheme over $\mathcal{O}_{K,S}$, then the set $\mathrm{H}^1_{\mathrm{fppf}}(\Spec \mathcal{O}_{K,S}, G)$ of isomorphism classes of $G$-torsors over $\mathcal{O}_{K,S}$ is finite \cite[Proposition~5.1]{GilleMoretBailly}.

More broadly, finiteness properties for objects defined over $\mathcal{O}_{K,S}$, such as $G$-torsors,  $G$-gerbes, Brauer classes, smooth proper curves of fixed genus $g>1$, or certain Galois representations, often lead to Chevalley--Weil type statements via liftability properties for integral points along suitable morphisms; see Remark \ref{remark:arithmetically_continuous} for a brief discussion.

From this perspective, one aspect of the Chevalley--Weil theorem may be viewed as a statement about the liftability of integral points, a viewpoint that extends naturally to algebraic stacks (see for example \cite[\S 5]{JLalg}). We establish this lifting result for $G$-torsors with $G$ affine in this more general setting, as this perspective originally led us to Theorem~\ref{thm:cw_arithmetic}. Our motivation comes from Campana's Orbifold Mordell conjecture, which requires working in the stacky setting. Note that no connectedness assumption on $G$ is needed in the following result; in particular, it generalizes the lifting part of the classical Chevalley--Weil theorem.

 \begin{proposition} [Linear groups] \label{prop:cw_integrally_affine} Let $\mathcal{Y}$ be a finitely presented integral algebraic stack over $\mathcal{O}_{K,S}$. 
Let $\mathcal{G}$ be an affine finite type group scheme over $\mathcal{O}_{K,S}$ and let $\mathcal{X}\to \mathcal{Y}$ be a $\mathcal{G}$-torsor.   Then, 
 there is a number field $L$ and a finite set of finite places $T$ of $L$ containing  all the places of $L$ lying over $S$ such that every $\mathcal{O}_{K,S}$-point of $\mathcal{Y}$ lifts to an $\mathcal{O}_{L,T}$-point of $\mathcal{X}$.   
\end{proposition}
\begin{proof} (This is similar to the proof of \cite[Lemma~8.1]{JLitt} where the case of $G = \mathrm{PGL}_N$ is considered.)
Since $\mathrm{H}^1_{\textrm{fppf}}(\Spec \mathcal{O}_{K,S}, \mathcal{G})$ is finite \cite[Proposition~5.1]{GilleMoretBailly}, there is a number field $L/K$ and a finite set of finite places $T$ of $L$ containing all the places of $L$ lying over $S$ such that every $\mathcal{G}$-torsor over $\mathcal{O}_{K,S}$ has an $\mathcal{O}_{L,T}$-point (and thus becomes trivial over $\mathcal{O}_{L,T}$).
Let $P\in \mathcal{Y}(\mathcal{O}_{K,S})$. 
 Define $\mathcal{X}_P \to \Spec \mathcal{O}_{K,S}$ as the pullback of $\mathcal{X}\to \mathcal{Y}$ along $P\colon \Spec \mathcal{O}_{K,S}\to \mathcal{Y}$.  Note that $\mathcal{X}_P\to \Spec \mathcal{O}_{K,S}$ is a $\mathcal{G}$-torsor, and thus has a section over $\mathcal{O}_{L,T}$ (by the construction of $L$ and $T$).  This is summarized in the following diagram:  
\[ \xymatrix{ & & 
\mathcal{X}_P \ar[rr] \ar[d]^{\mathcal{G}-\text{torsor}} & & \mathcal{X} \ar[d] \\
\Spec \mathcal{O}_{L,T} \ar[rr] \ar@{.>}[urr]^{\exists} & & \Spec \mathcal{O}_{K,S} \ar[rr]_{P} & & \mathcal{Y} }
\]  Since the composed morphism $\Spec \mathcal{O}_{L,T}\to \Spec \mathcal{O}_{K,S}\to \mathcal{Y}$ factors over $\mathcal{X}$, this shows that $P$ lifts to an $\mathcal{O}_{L,T}$-point, as required.
\end{proof} 
 
 \begin{proposition}\label{prop:cw_abelian}
  Let $\mathcal{Y}$ be a smooth finite type integral  scheme over $\mathcal{O}_{K,S}$. Let $\mathcal{G}\to \mathcal{Y}$ be a finite type group scheme  such that one of the following conditions holds.

 \begin{enumerate}
 \item $\mathcal{G}\to \mathcal{Y}$ is an abelian scheme.
 \item $ \mathcal{G}\to \mathcal{Y}$ is a smooth affine torus (i.e., there is an fppf covering $\mathcal{Y}'\to \mathcal{Y}$ such that $\mathcal{G}_{\mathcal{Y}'}$ is isomorphic to $\mathbb{G}_m^r$ over $\mathcal{Y}'$). 
 \item The morphism $\mathcal{G}\to \mathcal{Y}$ is smooth and its geometric fibres  are (affine) reductive groups.
 \end{enumerate}
Then, if $\mathcal{X}\to \mathcal{Y}$ is a $\mathcal{G}$-torsor,    
 there is a number field $L$ and a finite set of finite places $T$ of $L$ containing  all the places of $L$ lying over $S$ such that $\mathcal{Y}(\mathcal{O}_{K,S})$ is contained in the image of $ \mathcal{X}(\mathcal{O}_{L,T})\to  \mathcal{Y}(\mathcal{O}_{L,T})$. 
 \end{proposition}
 \begin{proof}
The set of $\mathcal{O}_{K,S}$-isomorphism classes of abelian schemes (resp.  smooth affine tori, resp. reductive group schemes) of bounded dimension over $\mathcal{O}_{K,S}$ is finite by Faltings's theorem \cite{Faltings2}  and \cite[\S 1]{JL2} (which itself builds on  \cite{BorelSerre}).  Thus, if  $P\in \mathcal{Y}(\mathcal{O}_{K,S})$ and $\mathcal{G}_P$ is the pullback of $\mathcal{G}\to \mathcal{Y}$ along $P$, then  the $\mathcal{G}_P$ range over finitely many $\mathcal{O}_{K,S}$-isomorphism classes of group schemes over $\mathcal{O}_{K,S}$.   For a fixed such group scheme $G$ over $\mathcal{O}_{K,S}$, the set of $G$-torsors over $\mathcal{O}_{K,S}$ is finite in case (2) and (3) by the  finiteness of $\mathrm{H}^1_{\mathrm{fppf}}(\Spec \mathcal{O}_{K,S},  G)$ (see  \cite[Proposition~5.1]{GilleMoretBailly}).
 In the case of an abelian scheme $\mathcal{A}$ over $\mathcal{Y}$,  as $\mathcal{Y}$ is a regular noetherian scheme,   the group of $\mathcal{A}$-torsors over $\mathcal{Y}$ is torsion  \cite[Proposition~XIII.2.6]{RaynaudAbelian}.  Let $n \in \mathbb{Z}_{\geq 1}$ be the order of $\mathcal{X}\to \mathcal{Y}$ (assuming $\mathcal{G}\to \mathcal{Y}$ is an abelian scheme). Then each $\mathcal{X}_P$ is  a $\mathcal{G}_P$-torsor of order dividing $n$, and thus lies in $\mathrm{H}^1_{\mathrm{fppf}}(\Spec \mathcal{O}_{K,S}, \mathcal{G}_P)[n]$ which is finite, because by the Kummer sequence it is surjected upon by $\mathrm{H}^1_{\mathrm{fppf}}(\Spec \mathcal{O}_{K,S}, \mathcal{G}_P[n])$. 
In either case, this implies that the  $\mathcal{X}_P$ run over finitely many isomorphism classes and thus can be simultanously trivialized over some suitable finite extension $\mathcal{O}_{L,T}$ (as in the proof of Proposition \ref{prop:cw_integrally_affine}).    This implies the result (again, as in the proof of Proposition \ref{prop:cw_integrally_affine}). 
 \end{proof}

 \begin{theorem}\label{thm:lift}
  Let $\mathcal{Y}$ be a smooth finite type scheme over $\mathcal{O}_{K,S}$. 
Let $\mathcal{G}$ be a  finite type group scheme over $\mathcal{O}_{K,S}$ and let $\mathcal{X}\to \mathcal{Y}$ be a $\mathcal{G}$-torsor.   Then, 
 there is a number field $L/K$ and a finite set of finite places $T$ of $L$ containing  all the places of $L$ lying over $S$ such that every $\mathcal{O}_{K,S}$-point of $\mathcal{Y}$ lifts to an $\mathcal{O}_{L,T}$-point of $\mathcal{X}$.   
 \end{theorem}
\begin{proof} 
Applying Chevalley's structure theorem \cite{ConradChevalley} to the
identity component of $\mathcal{G}_K$ and spreading out, we may choose a finite set of finite places
$S'$ containing $S$  such that, writing  $R=\mathcal{O}_{K,S'}$,  the group scheme $\mathcal{G}_R$ has a closed normal subgroup
scheme $\mathcal{G}^0$ with geometrically connected fibres and finite
\'etale quotient $\mathcal{G}_R/\mathcal{G}^0$, and that there is a
short exact sequence
\[
  1 \to \mathcal{G}^a \to \mathcal{G}^0 \to \mathcal{A} \to 1
\]
with $\mathcal{G}^a$ affine and $\mathcal{A}$ an abelian scheme over $R$. Consider the resulting tower of torsors
\[
  \mathcal{X}_R
  \longrightarrow \mathcal{X}_R/\mathcal{G}^a
  \longrightarrow \mathcal{X}_R/\mathcal{G}^0
  \longrightarrow \mathcal{Y}_R .
\]
To conclude the proof, starting from the right, apply Proposition~\ref{prop:cw_integrally_affine}
to the $\mathcal{G}_R/\mathcal{G}^0$-torsor,
Proposition~\ref{prop:cw_abelian} to the $\mathcal{A}$-torsor, and
Proposition~\ref{prop:cw_integrally_affine} to the
$\mathcal{G}^a$-torsor.  
\end{proof}

 \begin{remark}\label{remark:arithmetically_continuous}
The question of whether integral points on a variety (or stack) $X$ over a number field $K$ lift to a variety $Y$ dominating $X$ has been widely studied. It is closely related to the Hilbert property \cite[\S 3]{Serre} and to the notion of arithmetic surjectivity (see, for instance, \cite{LSS}).

In the terminology of Hassett and Tschinkel \cite[Definition~2.1]{HT01}, the above implies that $G$-torsors are arithmetically continuous when $G$ is constant or an abelian scheme over $Y$.     Projective bundles provide further examples: by Theorem~2.8 of \emph{loc.\ cit.} (due to Colliot-Th\'el\`ene), they are arithmetically continuous.  More generally,  a result of Colliot-Th\'el\`ene and Iyer shows that any smooth projective family of homogeneous spaces for linear algebraic groups is arithmetically continuous; see \cite{CTI}.  
\end{remark}

\begin{proof}[Proof of Theorem \ref{thm:cw_arithmetic}] If $X$ is arithmetically special, then so is $Y$.  Conversely,   suppose that $Y$ is arithmetically special.   Then, replacing $K$ by a finite extension if necessary, we may choose   a finite set of finite places $S$ of $K$, a model $\mathcal{Y}$ for $Y$ over $\mathcal{O}_{K,S}$ with $\mathcal{Y}(\mathcal{O}_{K,S})$ dense, a finite type group scheme $\mathcal{G}$ over $\mathcal{O}_{K,S}$   with generic fibre $G$, and a $\mathcal{G}$-torsor $\mathcal{X}\to \mathcal{Y}$ such that $\mathcal{X}_{K}$ is isomorphic to $X$.  By potential density of integral points on   finite type group schemes,  replacing $K$ by a finite field extension and $S$ by a larger finite set of finite places if necessary, we may also assume that    $\mathcal{G}(\mathcal{O}_{K,S})$ is dense.     Then, by Theorem \ref{thm:lift}, there is a number field $L/K$ and a finite set of finite places $T$ of $L$ containing  all the places of $L$ lying over $S$ such that the image of $\mathcal{X}(\mathcal{O}_{L,T})$ in $\mathcal{Y}(\mathcal{O}_{L,T})$ contains $\mathcal{Y}(\mathcal{O}_{K,S})$.   In particular, for every $P\in \mathcal{Y}(\mathcal{O}_{K,S})$, the fibre $\mathcal{X}_P$ is isomorphic to $\mathcal{G}$ over $\mathcal{O}_{L,T}$.  Thus, since $\mathcal{G}(\mathcal{O}_{K,S})$ is dense in $\mathcal{G}$, it follows that   $\mathcal{X}_P(\mathcal{O}_{L,T})$ is dense in $\mathcal{X}_P$.  
Therefore,  as $\mathcal{Y}(\mathcal{O}_{K,S})$ is dense in $Y$, we   conclude that $\mathcal{X}(\mathcal{O}_{L,T})$ is dense in $X$. 
\end{proof}



%
%
%
%
%
\section{An algebraic perspective on entire curves} \label{app:entire_curves}

Let $X$ be a scheme of finite type over $\mathbb{C}$ and let $S$ be any complex-analytic space.
Let $X^{\an}$ denote the associated complex-analytic space (see \cite[Exposé~XII]{SGA1}) and let $\mathcal{O}(S)$ denote the ring of global holomorphic functions on $S$.
In this appendix we briefly discuss the relationship between holomorphic maps $S \to X^{\an}$ and scheme morphisms $\Spec \mathcal{O}(S) \to X$. In particular, we prove Theorem \ref{thm:entire_curves_algebraic} stated below.

To do so, first recall that $X^{\an}$ by construction comes equipped with a morphism $X^{\an} \to X$ of locally ringed spaces over $\mathbb{C}$ and that $S$, like any locally ringed space, admits a natural morphism $S \to \Spec \mathcal{O}(S)$ \cite[Tag~01I1]{stacks-project}.
Composing with these morphisms, we obtain the following maps of $\Hom$-sets, which are natural in $S$ and $X$.
\[ \Hom(\Spec \mathcal{O}(S), X) \to \Hom(S, X) \leftarrow \Hom(S, X^{\an}) \]
The universal property of the analytification ensures that the right arrow is always a bijection.
If $X$ is affine, the aforementioned \cite[Tag~01I1]{stacks-project} implies that the left arrow is a bijection as well.

Next, we restrict to the case that $S$ is Stein, that is, the sheaf cohomology groups $\HH^k(S, \mathcal{F})$ vanish for all coherent sheaves $\mathcal{F}$ and all $k \geq 1$.
We refer to the book \cite{GrauertRemmertSteinSpaces} for an introduction to Stein spaces.
Suppose that $X$ is quasi-affine and choose a locally closed immersion $X \subseteq \mathbb{A}^n$, where $X = U \cap Z$ with $U \subseteq \mathbb{A}^n$ open and $Z \subseteq \mathbb{A}^n$ closed.
Then a morphism $S \to \mathbb{C}^n$ is a morphism $S \to X^{\an}$ if and only if it factors through both $U^{\an}$ and $Z^{\an}$.
Since $Z$ is affine, the natural map $\Hom(\Spec \mathcal{O}(S), Z) \to \Hom(S, Z^{\an})$ constructed above is bijective.
To prove that the natural map $\Hom(\Spec \mathcal{O}(S), U) \to \Hom(S, U^{\an})$ is a bijection, it suffices to prove that a morphism $S \to \mathbb{C}^n$ factors through $U^{\an}$ if and only if the corresponding morphism $\Spec \mathcal{O}(S) \to \mathbb{A}^n$ factors through $U$.
But this is implied by the fact that if $f_1,\ldots, f_r \in \mathcal{O}(S)$ are elements which, viewed as functions on $S$, do not have any common zeros, they generate the unit ideal in $\mathcal{O}(S)$.
This latter statement is \cite[Chapter~V, §5, Theorem~5]{GrauertRemmertSteinSpaces} and relies on the fact that $S$ is Stein.
Combining these observations, we obtain that the natural map $\Hom(\Spec \mathcal{O}(S), X) \to \Hom(S, X^{\an})$ is a bijection for any Stein space $S$ and any quasi-affine $X$.

Next, we restrict to the case that $S$ is a noncompact Riemann surface.
It is well-known that such an $S$ is Stein and it follows from the exponential short exact sequence that every holomorphic line bundle on $S$ is trivial.
Moreover, it follows from \cite[Satz~6.3]{ForsterSteinModuln} that also every rank $1$ projective module over the ring $\mathcal{O}(S)$ is trivial.
Now suppose that $X \subseteq \mathbb{P}^n$ is quasi-projective and let $\widetilde{X}$ be its preimage under the natural map $\mathbb{A}^{n+1} \setminus \{0\} \to \mathbb{P}^n$.
Then, since $\widetilde{X}$ is a locally closed subscheme of $\mathbb{A}^{n+1} \setminus \{0\}$, it is quasi-affine, so that every morphism $S \to \widetilde{X}^{\an}$ comes from a unique morphism $\Spec \mathcal{O}(S) \to \widetilde{X}$.
Moreover, since all line bundles on $S$ and $\Spec \mathcal{O}(S)$ are trivial, we also find that every morphism $S \to X^{\an}$ lifts to a morphism $S \to \widetilde{X}^{\an}$ and that every morphism $\Spec \mathcal{O}(S) \to X$ lifts to a morphism $\Spec \mathcal{O}(S) \to \widetilde{X}$, with these lifts being unique up to elements of $\mathcal{O}(S)^\times$.
It follows that the natural map $\Hom(\Spec \mathcal{O}(S), X) \to \Hom(S, X^{\an})$ is a bijection for any noncompact Riemann surface $S$ and any quasi-projective $X$.

Finally, we study how the image of a map $S \to X^{\an}$ is related to the image of a morphism $\Spec \mathcal{O}(S) \to X$.
To do so, we note that since an element of $\mathcal{O}(S)$ is zero if and only if it vanishes at every point of $S$, the image of the natural morphism $S \to \Spec \mathcal{O}(S)$ is (Zariski\nobreakdash-)dense.
It follows that the Zariski-closure of the image of a morphism $\Spec \mathcal{O}(S) \to X$ and the Zariski-closure of the image of a morphism $S \to X^{\an}$ agree.  Since $S$ (and hence $\Spec \mathcal{O}(S)$) is reduced,   it follows from  \cite[Tag~056B]{stacks-project} that the former is just the scheme-theoretic image.
Summarizing this discussion, we have thus proven the following result.

\begin{theorem} \label{thm:entire_curves_algebraic}
Let $S$ be a  noncompact Riemann surface and let $X$ be a quasi-projective scheme over $\mathbb{C}$.
Then there is a bijection
\[ \Hom(\Spec \mathcal{O}(S), X) \to \Hom(S, X^{\an}) \]
which is natural in $S$ and $X$.
Under this bijection, the scheme-theoretic image of a morphism $\Spec \mathcal{O}(S) \to X$ corresponds to the Zariski-closure of the image of a morphism $S \to X^{\an}$.
\end{theorem}

\bibliography{integral_wsc}{}

@preamble{{\def\cprime{$'$} }}

@article{ForsterSteinModuln,
    AUTHOR = {Forster, O.},
     TITLE = {Zur {T}heorie der {S}teinschen {A}lgebren und {M}oduln},
   JOURNAL = {Math. Z.},
  FJOURNAL = {Mathematische Zeitschrift},
    VOLUME = {97},
      YEAR = {1967},
     PAGES = {376--405},
      ISSN = {0025-5874,1432-1823},
   MRCLASS = {32.49},
  MRNUMBER = {213611},
MRREVIEWER = {O.\ F. G. Schilling},
       DOI = {10.1007/BF01112815},
       URL = {https://doi.org/10.1007/BF01112815},
}

@book{GrauertRemmertSteinSpaces,
    AUTHOR = {Grauert, H. and Remmert, R.},
     TITLE = {Theory of {S}tein spaces},
    SERIES = {Grundlehren der Mathematischen Wissenschaften},
    VOLUME = {236},
      NOTE = {Translated from the German by Alan Huckleberry},
 PUBLISHER = {Springer-Verlag, Berlin-New York},
      YEAR = {1979},
     PAGES = {xxi+249},
      ISBN = {3-540-90388-7},
   MRCLASS = {32-01},
  MRNUMBER = {580152},
}

@article {GraberHarrisStarr,
    AUTHOR = {Graber,  T. and Harris,  J. and Starr,  J.},
     TITLE = {Families of rationally connected varieties},
   JOURNAL = {J. Amer. Math. Soc.},
  FJOURNAL = {Journal of the American Mathematical Society},
    VOLUME = {16},
      YEAR = {2003},
    NUMBER = {1},
     PAGES = {57--67},
      ISSN = {0894-0347,1088-6834},
   MRCLASS = {14M20 (14D05 14D22)},
  MRNUMBER = {1937199},
MRREVIEWER = {Alexandr\ V.\ Pukhlikov},
       DOI = {10.1090/S0894-0347-02-00402-2},
       URL = {https://doi.org/10.1090/S0894-0347-02-00402-2},
}

@Article{	 BartschOrb,
  author	= {Bartsch, F. },
  title		= {  The morphism to the orbifold base need not be an orbifold map},
  journal	= {arXiv:2603.06083}
}

@article {BogHalPazTan,
    AUTHOR = {Bogomolov, F. and Halle, L. H.   and Pazuki, F.
              and Tanimoto,  S.},
     TITLE = {Abelian {C}alabi-{Y}au threefolds: {N}\'{e}ron models and
              rational points},
   JOURNAL = {Math. Res. Lett.},
  FJOURNAL = {Mathematical Research Letters},
    VOLUME = {25},
      YEAR = {2018},
    NUMBER = {2},
     PAGES = {367--392},
      ISSN = {1073-2780,1945-001X},
   MRCLASS = {14J32 (11G05 14D06 14G05)},
  MRNUMBER = {3826826},
MRREVIEWER = {Tyler\ L.\ Kelly},
       DOI = {10.4310/MRL.2018.v25.n2.a2},
       URL = {https://doi.org/10.4310/MRL.2018.v25.n2.a2},
}

@article {Darmon,
    AUTHOR = {Darmon, H.},
     TITLE = {Faltings plus epsilon, {W}iles plus epsilon, and the
              generalized {F}ermat equation},
   JOURNAL = {C. R. Math. Rep. Acad. Sci. Canada},
  FJOURNAL = {La Soci\'{e}t\'{e} Royale du Canada. L'Acad\'{e}mie des
              Sciences. Comptes Rendus Math\'{e}matiques (Mathematical
              Reports)},
    VOLUME = {19},
      YEAR = {1997},
    NUMBER = {1},
     PAGES = {3--14},
      ISSN = {0706-1994},
   MRCLASS = {11D41 (11G18 14G25)},
  MRNUMBER = {1479291},
MRREVIEWER = {Andrew\ Bremner},
}

@article {Cadman,
    AUTHOR = {Cadman, C.},
     TITLE = {Using stacks to impose tangency conditions on curves},
   JOURNAL = {Amer. J. Math.},
  FJOURNAL = {American Journal of Mathematics},
    VOLUME = {129},
      YEAR = {2007},
    NUMBER = {2},
     PAGES = {405--427},
      ISSN = {0002-9327,1080-6377},
   MRCLASS = {14D20 (14A20 14N35)},
  MRNUMBER = {2306040},
MRREVIEWER = {Michael\ A.\ Rose},
       DOI = {10.1353/ajm.2007.0007},
       URL = {https://doi.org/10.1353/ajm.2007.0007},
}

@article {JLitt,
    AUTHOR = {Javanpeykar, A. and Litt,  D.},
     TITLE = {Integral points on algebraic subvarieties of period domains:
              from number fields to finitely generated fields},
   JOURNAL = {Manuscripta Math.},
  FJOURNAL = {Manuscripta Mathematica},
    VOLUME = {173},
      YEAR = {2024},
    NUMBER = {1-2},
     PAGES = {23--44},
      ISSN = {0025-2611,1432-1785},
   MRCLASS = {14G05 (11G35 32Q45)},
  MRNUMBER = {4684343},
MRREVIEWER = {Daniel\ Loughran},
       DOI = {10.1007/s00229-023-01463-w},
       URL = {https://doi.org/10.1007/s00229-023-01463-w},
}

@incollection {AbramovichBirGeom,
    AUTHOR = {Abramovich, D.},
     TITLE = {Birational geometry for number theorists},
 BOOKTITLE = {Arithmetic geometry},
    SERIES = {Clay Math. Proc.},
    VOLUME = {8},
     PAGES = {335--373},
 PUBLISHER = {Amer. Math. Soc., Providence, RI},
      YEAR = {2009},
   MRCLASS = {14E30 (11G35 14E15)},
  MRNUMBER = {2498065},
MRREVIEWER = {Alexandr V. Pukhlikov},
}

@article {AbramovichVA ,
    AUTHOR = {Abramovich,  D. and V\'{a}rilly-Alvarado,  A.},
     TITLE = {Campana points, {V}ojta's conjecture, and level structures on
              semistable abelian varieties},
   JOURNAL = {J. Th\'{e}or. Nombres Bordeaux},
  FJOURNAL = {Journal de Th\'{e}orie des Nombres de Bordeaux},
    VOLUME = {30},
      YEAR = {2018},
    NUMBER = {2},
     PAGES = {525--532},
      ISSN = {1246-7405},
   MRCLASS = {11J97 (11G18 14G05 14K10 14K15)},
  MRNUMBER = {3891325},
MRREVIEWER = {Julie Tzu-Yueh Wang},
       URL = {http://jtnb.cedram.org/item?id=JTNB_2018__30_2_525_0},
}

@article {CTSSD,
    AUTHOR = {Colliot-Th\'{e}l\`ene, J.-L. and Skorobogatov, A. N. and
              Swinnerton-Dyer, P.},
     TITLE = {Double fibres and double covers: paucity of rational points},
   JOURNAL = {Acta Arith.},
  FJOURNAL = {Acta Arithmetica},
    VOLUME = {79},
      YEAR = {1997},
    NUMBER = {2},
     PAGES = {113--135},
      ISSN = {0065-1036},
   MRCLASS = {11G35 (14G05)},
  MRNUMBER = {1438597},
MRREVIEWER = {Dan Abramovich},
       DOI = {10.4064/aa-79-2-113-135},
       URL = {https://doi.org/10.4064/aa-79-2-113-135},
}

@book{OlssonBook,
	author = {Olsson, M.},
	doi = {10.1090/coll/062},
	isbn = {978-1-4704-2798-6},
	mrclass = {14D23 (14D22)},
	mrnumber = {3495343},
	mrreviewer = {Stefan\ Schr{\"o}er},
	pages = {xi+298},
	publisher = {American Mathematical Society, Providence, RI},
	series = {American Mathematical Society Colloquium Publications},
	title = {Algebraic spaces and stacks},
	url = {https://doi.org/10.1090/coll/062},
	volume = {62},
	year = {2016}
}

@incollection{BT,
	author = {Bogomolov, F. and Tschinkel, Y.},
	booktitle = {The {F}ano {C}onference},
	isbn = {88-900876-1-7},
	mrclass = {14D06 (14J27 14J30)},
	mrnumber = {2112577},
	mrreviewer = {Caryn\ Werner},
	pages = {223--234},
	publisher = {Univ. Torino, Turin},
	title = {Special elliptic fibrations},
	year = {2004}
}

@article{RTW,
	author = {Rousseau, E. and Turchet, A. and Wang, J. T.-Y.},
	doi = {10.1017/fms.2021.8},
	fjournal = {Forum of Mathematics. Sigma},
	issn = {2050-5094},
	journal = {Forum Math. Sigma},
	mrclass = {14G40 (11J97 14G05 32H30 32M99 32Q45)},
	mrnumber = {4215673},
	mrreviewer = {Zhizhong\ Huang},
	pages = {Paper No. e11, 29},
	title = {Nonspecial varieties and generalised {L}ang-{V}ojta conjectures},
	url = {https://doi.org/10.1017/fms.2021.8},
	volume = {9},
	year = {2021}
}

@article{CampanaWinkelmannBrody,
	author = {Campana,  F. and Winkelmann, J.},
	doi = {10.1007/s00229-008-0231-6},
	fjournal = {Manuscripta Mathematica},
	issn = {0025-2611},
	journal = {Manuscripta Math.},
	mrclass = {32Q45 (14E22 14J29)},
	mrnumber = {2471315},
	mrreviewer = {Qi Lin Yang},
	number = {2},
	pages = {195--212},
	title = {A {B}rody theorem for orbifolds},
	url = {https://doi.org/10.1007/s00229-008-0231-6},
	volume = {128},
	year = {2009}
}

@incollection{HassettSurvey,
	author = {Hassett, B.},
	booktitle = {Higher dimensional varieties and rational points ({B}udapest, 2001)},
	doi = {10.1007/978-3-662-05123-8_8},
	mrclass = {14G05 (11G35 14G25)},
	mrnumber = {2011748},
	mrreviewer = {Robin de Jong},
	pages = {223--282},
	publisher = {Springer, Berlin},
	series = {Bolyai Soc. Math. Stud.},
	title = {Potential density of rational points on algebraic varieties},
	url = {https://doi.org/10.1007/978-3-662-05123-8_8},
	volume = {12},
	year = {2003}
}

@article{Moerman,
	author = {Moerman, B. },
	journal = { arXiv:2407.03048    },
	title = { Generalized {C}ampana points and adelic approximation on toric varieties}
}

@article {Bartsch,
    AUTHOR = {Bartsch,  F.},
     TITLE = {New examples of geometrically special varieties: {K}3
              surfaces, {E}nriques surfaces, and algebraic groups},
   JOURNAL = {Manuscripta Math.},
  FJOURNAL = {Manuscripta Mathematica},
    VOLUME = {177},
      YEAR = {2026},
    NUMBER = {3},
     PAGES = {Paper No. 31, 22},
      ISSN = {0025-2611,1432-1785},
   MRCLASS = {14G25 (14G05 14J28 14L15 32Q99)},
  MRNUMBER = {5051679},
       DOI = {10.1007/s00229-026-01709-3},
       URL = {https://doi.org/10.1007/s00229-026-01709-3},
}

@article{BCJW,
	author = {Bartsch, F.  and Campana, F. and  Javanpeykar, A. and Wittenberg, O. },
	journal = {JEMS, to appear.  	arXiv:2410.06643  },
	title = {The {W}eakly {S}pecial {C}onjecture contracts {O}rbifold {M}ordell, and thus abc }
}

@article {BJL,
    AUTHOR = {Bartsch, F. and Javanpeykar,  A. and Levin,  A.},
     TITLE = {Symmetric products and puncturing {C}ampana-special varieties},
   JOURNAL = {Proc. Lond. Math. Soc. (3)},
  FJOURNAL = {Proceedings of the London Mathematical Society. Third Series},
    VOLUME = {131},
      YEAR = {2025},
    NUMBER = {6},
     PAGES = {Paper No. e70108},
      ISSN = {0024-6115,1460-244X},
   MRCLASS = {14G05 (11G35 32Q45)},
  MRNUMBER = {5003420},
       DOI = {10.1112/plms.70108},
       URL = {https://doi.org/10.1112/plms.70108},
}

@article{GRTW,
title = {Simply connectedness and hyperbolicity},
journal = {Journal of Number Theory},
volume = {285},
pages = {194-208},
year = {2026},
issn = {0022-314X},
doi = {https://doi.org/10.1016/j.jnt.2026.01.005},
url = {https://www.sciencedirect.com/science/article/pii/S0022314X26000302},
author = {C. Gasbarri and E. Rousseau and A. Turchet and J.  T.-Y.  Wang}
}

@article{HassettTschinkel,
	author = {Hassett, B. and Tschinkel, Y.},
	fjournal = {International Journal of Mathematics},
	issn = {0129-167X},
	journal = {Internat. J. Math.},
	mrclass = {14D06 (11G35 14D05 14J28 14K15)},
	mrnumber = {1809306},
	mrreviewer = {Tam{\'a}s Szamuely},
	number = {9},
	pages = {1163--1176},
	title = {Abelian fibrations and rational points on symmetric products},
	url = {https://doi.org/10.1142/S0129167X00000544},
	volume = {11},
	year = {2000}
}

@article{Ca04,
	author = {Campana, F.},
	fjournal = {Universit{\'e} de Grenoble. Annales de l'Institut Fourier},
	journal = {Ann. Inst. Fourier (Grenoble)},
	number = {3},
	pages = {499--630},
	title = {Orbifolds, special varieties and classification theory},
	url = {http://aif.cedram.org/item?id=AIF_2004__54_3_499_0},
	volume = {54},
	year = {2004}
}

@article{Ca11,
	author = {Campana, F.},
	doi = {10.1017/S1474748010000101},
	journal = {Journal de l'Institut de Math{\'e}matiques de Jussieu},
	number = {4},
	pages = {809--934},
	title = {Orbifoldes g{\'e}om{\'e}triques sp{\'e}ciales et classification bim{\'e}romorphe des vari{\'e}t{\'e}s k{\"a}hl{\'e}riennes compactes},
	volume = {10},
	year = {2011}
}

@article{CP07,
	author = {Campana, F. and P{\u a}un, M.},
	doi = {10.1112/S0010437X06002685},
	fjournal = {Compositio Mathematica},
	issn = {0010-437X},
	journal = {Compos. Math.},
	number = {1},
	pages = {95--111},
	title = {Vari{\'e}t{\'e}s faiblement sp{\'e}ciales {\`a} courbes enti{\`e}res d{\'e}g{\'e}n{\'e}r{\'e}es},
	url = {http://dx.doi.org/10.1112/S0010437X06002685},
	volume = {143},
	year = {2007}
}

@article{Ca05,
	author = {Campana, F.},
	fjournal = {Manuscripta Mathematica},
	issn = {0025-2611},
	journal = {Manuscripta Math.},
	mrclass = {14D06 (14J40)},
	mrnumber = {2163487},
	mrreviewer = {Lucia Caporaso},
	number = {4},
	pages = {429--461},
	title = {Fibres multiples sur les surfaces: aspects geom{\'e}triques, hyperboliques et arithm{\'e}tiques},
	url = {https://doi.org/10.1007/s00229-005-0570-5},
	volume = {117},
	year = {2005}
}

@article {Smeets,
    AUTHOR = {Smeets,  A.},
     TITLE = {Insufficiency of the \'{e}tale {B}rauer-{M}anin obstruction:
              towards a simply connected example},
   JOURNAL = {Amer. J. Math.},
  FJOURNAL = {American Journal of Mathematics},
    VOLUME = {139},
      YEAR = {2017},
    NUMBER = {2},
     PAGES = {417--431},
      ISSN = {0002-9327},
   MRCLASS = {14F22 (11G35)},
  MRNUMBER = {3636635},
MRREVIEWER = {Giancarlo Lucchini Arteche},
       DOI = {10.1353/ajm.2017.0010},
       URL = {https://doi.org/10.1353/ajm.2017.0010},
}

@article{Smeets2,
	author = {Kebekus, S. and Pereira, J.-V. and Smeets, A.},
	doi = {10.1215/00127094-2022-0045},
	fjournal = {Duke Mathematical Journal},
	issn = {0012-7094},
	journal = {Duke Math. J.},
	mrclass = {14G12 (11G35 14F22)},
	mrnumber = {4510017},
	number = {17},
	pages = {3515--3591},
	title = {Failure of the {B}rauer-{M}anin principle for a simply connected fourfold over a global function field, via orbifold {M}ordell},
	url = {https://doi.org/10.1215/00127094-2022-0045},
	volume = {171},
	year = {2022}
}

@article {LSS,
    AUTHOR = {Loughran, D.   and Skorobogatov,  A.  N. and Smeets, A.},
     TITLE = {Pseudo-split fibers and arithmetic surjectivity},
   JOURNAL = {Ann. Sci. \'{E}c. Norm. Sup\'{e}r. (4)},
  FJOURNAL = {Annales Scientifiques de l'\'{E}cole Normale Sup\'{e}rieure.
              Quatri\`eme S\'{e}rie},
    VOLUME = {53},
      YEAR = {2020},
    NUMBER = {4},
     PAGES = {1037--1070},
      ISSN = {0012-9593,1873-2151},
   MRCLASS = {14G05 (11G35)},
  MRNUMBER = {4157111},
MRREVIEWER = {Christopher\ Frei},
       DOI = {10.24033/asens.2439},
       URL = {https://doi.org/10.24033/asens.2439},
}

@article{CTI,
	author = {Colliot-Th{\'e}l{\`e}ne, J.-L. and Iyer, J. N.},
	fjournal = {Journal of the Ramanujan Mathematical Society},
	issn = {0970-1249},
	journal = {J. Ramanujan Math. Soc.},
	mrclass = {14G05 (11G35 14D10 14G25 14M17)},
	mrnumber = {2987230},
	mrreviewer = {Yong Hu},
	number = {3},
	pages = {295--303},
	title = {Potential density for some families of homogeneous spaces},
	volume = {27},
	year = {2012}
}

@article{ConradChevalley,
	author = {Conrad, B.},
	fjournal = {Journal of the Ramanujan Mathematical Society},
	issn = {0970-1249},
	journal = {J. Ramanujan Math. Soc.},
	mrclass = {20G15 (11H99 14K05 14L15)},
	mrnumber = {1906417},
	mrreviewer = {James E. Humphreys},
	number = {1},
	pages = {1--18},
	title = {A modern proof of {C}hevalley's theorem on algebraic groups},
	volume = {17},
	year = {2002}
}

@incollection{GilleMoretBailly,
	address = {Cambridge},
	author = {Gille, P. and Moret-Bailly, L.},
	booktitle = {Torsors, {\'e}tale homotopy and applications to rational points},
	mrclass = {14L15 (14G25)},
	mrnumber = {3077171},
	pages = {231--249},
	publisher = {Cambridge Univ. Press},
	series = {London Math. Soc. Lecture Note Ser.},
	title = {Actions alg{\'e}briques de groupes arithm{\'e}tiques},
	volume = {405},
	year = {2013}
}

@incollection{HT01,
	author = {Hassett, B. and Tschinkel, Y.},
	booktitle = {Rational points on algebraic varieties},
	mrclass = {14G25 (11G35 14D10 14G05 14J20)},
	mrnumber = {1875174},
	mrreviewer = {Tam{\'a}s Szamuely},
	pages = {169--197},
	publisher = {Birkh{\"a}user, Basel},
	series = {Progr. Math.},
	title = {Density of integral points on algebraic varieties},
	url = {https://doi.org/10.1007/978-3-0348-8368-9_7},
	volume = {199},
	year = {2001}
}

@article{Faltings2,
	author = {Faltings, G.},
	coden = {INVMBH},
	doi = {10.1007/BF01388432},
	fjournal = {Inventiones Mathematicae},
	issn = {0020-9910},
	journal = {Invent. Math.},
	mrclass = {11D41 (11G30 14G25)},
	mrnumber = {718935 (85g:11026a)},
	mrreviewer = {James Milne},
	number = {3},
	pages = {349--366},
	title = {Endlichkeitss{\"a}tze f{\"u}r abelsche {V}ariet{\"a}ten {\"u}ber {Z}ahlk{\"o}rpern},
	url = {http://dx.doi.org/10.1007/BF01388432},
	volume = {73},
	year = {1983}
}

@book{SGA1,
	address = {Paris},
	author = {Grothendieck, A.},
	mrclass = {14.55},
	mrnumber = {0217088 (36 \#179b)},
	mrreviewer = {T. Oda},
	pages = {i+163 pp. (not consecutively paged) (loose errata)},
	publisher = {Institut des Hautes {\'E}tudes Scientifiques},
	series = {S{\'e}minaire de G{\'e}om{\'e}trie Alg{\'e}brique},
	title = {Rev\^etements {\'e}tales et groupe fondamental (SGA I) {F}asc. {II}: {E}xpos{\'e}s 6, 8 {\`a} 11},
	volume = {1960/61},
	year = {1963}
}

@article{Elkies,
	author = {Elkies, N. D.},
	doi = {10.1155/S1073792891000144},
	fjournal = {International Mathematics Research Notices},
	issn = {1073-7928},
	journal = {Internat. Math. Res. Notices},
	mrclass = {11G30 (11D41)},
	mrnumber = {1141316 (93d:11064)},
	mrreviewer = {Joseph H. Silverman},
	number = {7},
	pages = {99--109},
	title = {{$ABC$} implies {M}ordell},
	url = {http://dx.doi.org/10.1155/S1073792891000144},
	year = {1991}
}

@article{JL2,
	author = {Javanpeykar, A. and Loughran, D.},
	doi = {10.1007/s00013-015-0728-7},
	fjournal = {Archiv der Mathematik},
	issn = {0003-889X},
	journal = {Arch. Math. (Basel)},
	mrclass = {14L15 (11E72 14G25)},
	mrnumber = {3306042},
	number = {2},
	pages = {133--143},
	title = {Good reduction of algebraic groups and flag varieties},
	url = {http://dx.doi.org/10.1007/s00013-015-0728-7},
	volume = {104},
	year = {2015}
}

@article{G,
	author = {Carlson, J. and Green, M. and Griffiths, P. and Harris, J.},
	coden = {CMPMAF},
	fjournal = {Compositio Mathematica},
	issn = {0010-437X},
	journal = {Compositio Math.},
	mrclass = {32G13 (14C30 32J25)},
	mrnumber = {720288 (86e:32026a)},
	mrreviewer = {Sampei Usui},
	number = {2-3},
	pages = {109--205},
	title = {Infinitesimal variations of {H}odge structure. {I}},
	url = {http://www.numdam.org/item?id=CM_1983__50_2-3_109_0},
	volume = {50},
	year = {1983}
}

@book{Serre,
	author = {Serre, J.-P.},
	edition = {Second},
	isbn = {978-1-56881-412-4},
	mrclass = {12F12},
	mrnumber = {2363329},
	note = {With notes by Henri Darmon},
	pages = {xvi+120},
	publisher = {A K Peters, Ltd., Wellesley, MA},
	series = {Research Notes in Mathematics},
	title = {Topics in {G}alois theory},
	volume = {1},
	year = {2008}
}

@article{CDY,
	author = { Cadorel, B. and Deng, Y. and Yamanoi, K.},
	journal = { arXiv:2512.20360 },
	title = { Hyperbolicity and fundamental groups of complex quasi-projective varieties (III): applications}
}

@incollection{Vojta3,
	author = {Vojta, P.},
	booktitle = {Arithmetic geometry ({S}torrs, {C}onn., 1984)},
	mrclass = {11G35},
	mrnumber = {861984},
	pages = {341--353},
	publisher = {Springer, New York},
	title = {A higher-dimensional {M}ordell conjecture},
	year = {1986}
}

@article{AGV,
	author = {Abramovich, D. and Graber, T. and Vistoli, A.},
	coden = {AJMAAN},
	doi = {10.1353/ajm.0.0017},
	fjournal = {American Journal of Mathematics},
	issn = {0002-9327},
	journal = {Amer. J. Math.},
	mrclass = {14N35 (14A20 53D45)},
	mrnumber = {2450211 (2009k:14108)},
	mrreviewer = {Johannes Walcher},
	number = {5},
	pages = {1337--1398},
	title = {Gromov-{W}itten theory of {D}eligne-{M}umford stacks},
	url = {http://dx.doi.org/10.1353/ajm.0.0017},
	volume = {130},
	year = {2008}
}

@article{BorelSerre,
	author = {Borel, A. and Serre, J.-P.},
	fjournal = {Commentarii Mathematici Helvetici},
	issn = {0010-2571},
	journal = {Comment. Math. Helv.},
	mrclass = {14.50},
	mrnumber = {0181643},
	mrreviewer = {T. Ono},
	pages = {111--164},
	title = {Th{\'e}or{\`e}mes de finitude en cohomologie galoisienne},
	volume = {39},
	year = {1964}
}

@article{HarrisTschinkel,
	author = {Harris, J. and Tschinkel, Y.},
	doi = {10.1215/S0012-7094-00-10436-X},
	fjournal = {Duke Mathematical Journal},
	issn = {0012-7094,1547-7398},
	journal = {Duke Math. J.},
	mrclass = {14G05 (11G35 14G25 14J20)},
	mrnumber = {1781480},
	number = {3},
	pages = {477--500},
	title = {Rational points on quartics},
	url = {https://doi.org/10.1215/S0012-7094-00-10436-X},
	volume = {104},
	year = {2000}
}

@article{BT3,
	author = {Bogomolov, F. A. and Tschinkel, Y.},
	doi = {10.4310/AJM.2000.v4.n2.a6},
	fjournal = {Asian Journal of Mathematics},
	issn = {1093-6106},
	journal = {Asian J. Math.},
	mrclass = {14G05 (11G05 11G35 14C22 14J10 14J28)},
	mrnumber = {1797587},
	mrreviewer = {Matthew H. Baker},
	number = {2},
	pages = {351--368},
	title = {Density of rational points on elliptic {$K3$} surfaces},
	url = {https://doi.org/10.4310/AJM.2000.v4.n2.a6},
	volume = {4},
	year = {2000}
}

@article{J,
	author = {Javanpeykar, A.},
	doi = {10.2140/ant.2014.8.89},
	fjournal = {Algebra \& Number Theory},
	issn = {1937-0652},
	journal = {Algebra Number Theory},
	mrclass = {14G40 (11G30 11G32 11G50 14H55 37P30)},
	mrnumber = {3207580},
	mrreviewer = {J{\"o}rg Jahnel},
	note = {With an appendix by Peter Bruin},
	number = {1},
	pages = {89--140},
	title = {Polynomial bounds for {A}rakelov invariants of {B}elyi curves},
	url = {http://dx.doi.org/10.2140/ant.2014.8.89},
	volume = {8},
	year = {2014}
}

@book {RaynaudAbelian,
    AUTHOR = {Raynaud,  M.},
     TITLE = {Faisceaux amples sur les sch\'{e}mas en groupes et les espaces
              homog\`enes},
    SERIES = {Lecture Notes in Mathematics, Vol. 119},
 PUBLISHER = {Springer-Verlag, Berlin-New York},
      YEAR = {1970},
     PAGES = {ii+218},
   MRCLASS = {14.55},
  MRNUMBER = {260758},
MRREVIEWER = {D.\ Gieseker},
}

@article {JLalg,
    AUTHOR = {Javanpeykar, A. and Loughran, D.},
     TITLE = {Arithmetic hyperbolicity and a stacky {C}hevalley-{W}eil
              theorem},
   JOURNAL = {J. Lond. Math. Soc. (2)},
  FJOURNAL = {Journal of the London Mathematical Society. Second Series},
    VOLUME = {103},
      YEAR = {2021},
    NUMBER = {3},
     PAGES = {846--869},
      ISSN = {0024-6107,1469-7750},
   MRCLASS = {14G05 (11G25 14A20 32Q45)},
  MRNUMBER = {4245824},
MRREVIEWER = {Boris\ \`E.\ Kunyavski\u{\i}},
       DOI = {10.1112/jlms.12394},
       URL = {https://doi.org/10.1112/jlms.12394},
}

@incollection{L,
	author = {Lenstra, A. K.},
	booktitle = {Computer algebra ({L}ondon, 1983)},
	doi = {10.1007/3-540-12868-9_108},
	mrclass = {12-04 (11Y05 68Q40)},
	mrnumber = {774816},
	mrreviewer = {Ray Steiner},
	pages = {245--254},
	publisher = {Springer, Berlin},
	series = {Lecture Notes in Comput. Sci.},
	title = {Factoring polynomials over algebraic number fields},
	url = {http://dx.doi.org/10.1007/3-540-12868-9_108},
	volume = {162},
	year = {1983}
}

@article{Langconj,
	author = {Lang, S.},
	fjournal = {Bulletin of the American Mathematical Society},
	journal = {Bull. Amer. Math. Soc. (N.S.)},
	mrclass = {32H20 (11D99 14G05)},
	mrnumber = {828820},
	pages = {159--205},
	title = {Hyperbolic and diophantine analysis},
	volume = {14},
	year = {1986}
}

@article{BJRSeveri,
	author = {Bartsch, F. and Javanpeykar, A. and Rousseau, E.},
	journal = {	arXiv:2602.02314},
	title = { The theorem of {M}aehara-{S}everi for maps of general type}
}

@article {BJR,
    AUTHOR = {Bartsch,  F. and Javanpeykar, A. and Rousseau, E.},
     TITLE = {Weakly special threefolds and nondensity of rational points},
   JOURNAL = {J. Lond. Math. Soc. (2)},
  FJOURNAL = {Journal of the London Mathematical Society. Second Series},
    VOLUME = {112},
      YEAR = {2025},
    NUMBER = {5},
     PAGES = {Paper No. e70348, 28},
      ISSN = {0024-6107,1469-7750},
   MRCLASS = {11G35 (14G05 32Q45)},
  MRNUMBER = {4988556},
       DOI = {10.1112/jlms.70348},
       URL = {https://doi.org/10.1112/jlms.70348},
}

@book {PoonenRat,
    AUTHOR = {Poonen, B.},
     TITLE = {Rational points on varieties},
    SERIES = {Graduate Studies in Mathematics},
    VOLUME = {186},
 PUBLISHER = {American Mathematical Society, Providence, RI},
      YEAR = {2017},
     PAGES = {xv+337},
      ISBN = {978-1-4704-3773-2},
   MRCLASS = {14G05 (11G35)},
  MRNUMBER = {3729254},
MRREVIEWER = {Daniel Loughran},
       DOI = {10.1090/gsm/186},
       URL = {https://doi.org/10.1090/gsm/186},
}

@misc{stacks-project,
	author = {{Stacks Project Authors}, The},
	howpublished = {http://stacks.math.columbia.edu},
	shorthand = {Stacks},
	title = {\emph{{S}tacks {P}roject}},
	year = {2015}
}
\bibliographystyle{alpha}

\end{document}